\documentclass[reqno]{amsart}
\usepackage{eurosym}
\usepackage{enumerate}
\usepackage{amsmath}
\usepackage{amsfonts}
\usepackage{amssymb}
\usepackage{graphicx}
\usepackage{color}
\usepackage[pdfpagelabels]{hyperref}

\numberwithin{equation}{section}
\definecolor{mycolorred}{rgb}{1, 0, 0}

\definecolor{mycolorblue}{rgb}{0, 0, 1}
\def\blue #1{{ #1}}
\definecolor{mycolorpink}{rgb}{1, 0, 1}

\newtheorem{theorem}{Theorem}[section]

\newtheorem{corollary}[theorem]{Corollary}

\newtheorem{definition}[theorem]{Definition}

\newtheorem{example}[theorem]{Example}

\newtheorem{lemma}[theorem]{Lemma}

\newtheorem{proposition}[theorem]{Proposition}
\newtheorem{remark}[theorem]{Remark}

\def\<{\langle}
\def\>{\rangle}

{}

\def\R{{\mathbb R}}

\def\N{{\mathbb N}}

\begin{document}
\def\kp{{\kappa_p}}

\title[Euler-type approximation for the invariant measure]{Euler-type approximation for the invariant measure: An abstract framework}
\author{Aur\'elien Alfonsi}
\address{CERMICS, ENPC, Institut Polytechnique de Paris, CNRS, Marne-la-Vallée, France \& MathRisk team-project, Inria Paris, France.}
\email{aurelien.alfonsi@enpc.fr}
\author{Vlad Bally}
\address{Universit\'e Gustave Eiffel, LAMA (UMR CNRS, UPEMLV, UPEC), MathRisk INRIA,
  F-77454 Marne-la-Vall\'ee, France.}
\email{vlad.bally@univ-eiffel.fr}
\author{Arturo Kohatsu-Higa}
\address{Department of Mathematical Sciences, College of Science and Engineering, Ritsumeikan University,
1-1-1 Noji-higashi, Kusatsu, 525-8577, Japan. }
\email{khts00@fc.ritsumei.ac.jp}
\maketitle

\begin{abstract} 
	We establish a general framework to study the rate of convergence of a Euler type approximation scheme with  decreasing time steps to the invariant measure, for a general class of stochastic systems. The error is measured in general Wasserstein distances, which enables to encompass cases with non global contractivity conditions. Our main assumption is a coupling property which is expressed in terms of the one-step approximation. We show that the proposed set-up can be applied to a wide range of  equations that may be law dependent, such as Langevin equations, reflected equations, Boltzmann type equations and for a recent McKean Vlasov type model for neuronal activity.
\end{abstract}
\section{Introduction}
The problem of existence and uniqueness of the invariant measure for a Markov semigroup has a long history. We refer the reader to the influential paper by Meyn and Tweedie
\cite{MeTw} and Hairer's lecture notes~\cite{Hairer}. In direct connection to this question, the speed of convergence
of the semigroup to the invariant measure has been also widely studied, see in particular 
Bakry et al.~\cite{BCG}, Cattiaux and  Guillin~\cite{CaGu}, Hairer~\cite{Hairer} and the
references therein. Unless the semigroup can be simulated exactly, one typically uses approximation schemes of the semigroup to then approximate the invariant measure.
Approximation schemes for the invariant measure associated to some
stochastic process are thus built in two steps: a first step deals with the error between the invariant measure and the distribution of the semigroup for large times, and second step deals with the error between the semigroup and its approximation.

 In this sense, approximations based on Euler type schemes has also been an important research topic in the past years, particularly in the specific framework of the Langevin diffusion that is related to gradient descent algorithms that are widely used in Statistics and Machine Learning. We may distinguish two streams of research. The first one considers a constant time step for the Euler type approximation and then analyses the error between the invariant measures of the continuous and discretised semigroup, see in particular
Talay~\cite{Talay2002}, Mattingly et al.~\cite{MSH}, Dalalyan and Tsybakov~%
\cite{DaTs}, Dalalyan~\cite{Dalalyan}, Durmus and Moulines~\cite{DuMo},  Brosse et al.~\cite{BDM}, Crisan et al.~\cite{CDO} and Chen et al.~\cite{CDSX}. The second stream of research, initiated by Lamberton and Pag\`es~\cite{LP02,LP03} and Lemaire~\cite{Lemaire} considers instead an approximation scheme with a decreasing sequence of time steps, which allows to get directly the convergence to the invariant measure of the continuous semigroup.  This point of view has been used in  several recent works: see  Pag\`es and Panloup~\cite{PaPa}, Bally and Qin~\cite{BaQi}  and \blue{Chen et al.~ \cite{chen2025}.}

Furthermore a number of stochastic equations in which the law of the solution is part of the equation, such as in the cases of McKean-Vlasov or Boltzmann type equations, do not seem to be within the framework of Markov semigroups. The literature about the approximation of the invariant measure for such systems is less developed but very active in the last years (see Schuh and Souttar \cite{schuh2}, Cormier \cite{Cormier} and Du et. al. \cite{DJL} and the references therein). Having this motivation in mind, we focus on the approximations to the invariant measures for flows of transformations on the Wasserstein space, indexed on time. First, under appropriate hypotheses, we obtain exponential speed of convergence to the invariant measure as the time goes to infinity. Second, considering an approximation scheme with a decreasing step, we obtain in a general framework the rate of convergence in Wasserstein distance between the invariant measure and its approximation.
Our work is related to the recent paper of Schuh and Souttar~\cite{schuh2} who provide also a general framework to get  
uniform in time estimates between the continuous time process and its approximation, but for a constant time step approximation.  As a consequence, there are significant differences between their framework and ours that we present in detail in Remark~\ref{re:long} below.

Let us be more precise and present the main results of the paper. In the
first section (Section~\ref{Sec_Framework}) we introduce the main objects. We consider a
separable Polish metric space $(B,d)$. We
denote $\mathcal{P}(B)$ the space of probability measures on $B$ and, for $%
p\geq 1$, $\mathcal{P}_{p}(B)$ is the set of probability measures with
finite $p$ moment, endowed
with the Wasserstein distance~$W_{p}$. The main object of study in our paper is a
family of transformations $\Theta _{s,t}:\mathcal{P}_{p}(B)\rightarrow 
\mathcal{P}_{p}(B), s\leq t$. Our first problem, is to give conditions that ensure the existence and uniqueness of the invariant probability measure, that is  a probability measure $\nu $ such that $\Theta _{s,t}(\nu )=\nu$ for all $s<t$.

The next step towards our main goal  is to define the approximation scheme. With some abuse of notation, we define for a time grid $\pi =\{s=t_{0}<t_{1}<...<t_{n}<...\}$ such that $ t_n\to\infty  $,
\begin{equation*}
\Theta _{t_{0},t_{k}}^{\pi }(\mu )=\Theta _{t_{k-1},t_{k}}\circ \Theta
_{t_{k-1},t_{k-2}}\circ ...\circ \Theta _{t_{0},t_{1}}.
\end{equation*}%
 In the present case, we have in mind that $\Theta _{t_{i},t_{i+1}}$ represents an approximation, such as the one step Euler scheme between times $t_i$ and $t_{i+1}$, and therefore $%
\Theta _{t_{0},t_{k}}^{\pi }$ represents the corresponding Euler scheme on the time grid~$\pi$. Our goal is to give conditions so that $%
\Theta _{t_{0},t_{k}}^{\pi }$ converges in $ p $-Wasserstein distance to $ \nu  $ as $ k\to \infty $ and to determine its rate of convergence.

A first hypothesis which is in force in our approach is a uniform bound on the $p$-moments for some $h>0$: 
\begin{equation*}
M_{p}(\Theta ,\mu ,h)=\sup_{\pi: t_i-t_{i-1}\le h }\sup_{n\ge 0}\left\Vert \Theta
_{t_{0},t_{n}}^{\pi }(\mu )\right\Vert _{  p}^{p}<\infty,
\end{equation*}%
with the supremum taken on all the partitions $\pi$ with decreasing time steps and all $n\in \N$. In
many examples, this property is obtained with a Foster Lyapunov type
condition~\eqref{AN2}. In some other examples, this property can be derived from existing results on the corresponding continuous time process.  

Next,  we introduce a coupling argument that plays a central role in our analysis. We say that two families $%
\Theta ^{1}$ and $\Theta ^{2}$ are $(p,b_*,\varepsilon )-$ coupled with $b_*,\varepsilon>0$ if 
\begin{equation}
\label{i22}
W_{  p}^{p}(\Theta _{s,t}^{1}(\mu ^{1}),\Theta _{s,t}^{2}(\mu ^{2}))\leq
(1-(t-s)b_*)W_{  p}^{p}(\mu ^{1},\mu ^{2})+C(1+\left\Vert \mu ^{1}\right\Vert
_{  p}^{p}+\left\Vert \mu ^{2}\right\Vert _{  p}^{p})(t-s)^{1+\varepsilon }
\end{equation}
for $0<t-s<h$. \blue{Let $t_0\ge 0$ and $t_n=t_0+ \sum_{k=1}^n \frac{1}{1+k}$. Then, Corollary~\ref{cor_speed}  proves under~\eqref{i22} that for any $\mu^1,\mu^2 \in  \mathcal{P}_p(B)$, there exists $A\in \R_+$ such that
\begin{equation}
W_{p}^{p}(\Theta _{t_0,t_n}^{1,\pi }(\mu ^{1}),\Theta _{t_0,t_n}^{2,\pi }(\mu
^{2}))\leq A \frac{1}{n^{b_*\wedge \varepsilon}} , \ n\ge 0,\label{i2'}
\end{equation}%
up to  a multiplicative logarithm term when $b_*=\varepsilon$.}

Now we present the main consequences of the above general abstract result that will be proven throughout the article\footnote{For exact definitions and conditions, we refer the reader to the main text.}. The first one concerns
existence and uniqueness of the invariant measure and exponential estimates
of the speed of convergence. Let   $\theta _{s,t}:\mathcal{P}_{p}(B)\rightarrow \mathcal{P}_{p}(B),s\leq t$
be a flow ($\theta _{r,t}\circ \theta _{s,r}=\theta _{s,t}$ for $ 0\leq s\leq r\leq t $) that is time homogeneous (i.e. $\theta _{s,t}(\mu )=\theta _{0,t-s}(\mu )$) such
that $\mu \rightarrow \theta _{s,t}(\mu )$ is continuous for the $W_{p}$~distance. We assume that $\theta$ has a uniform bound on the $p$ moments and moreover that
$\theta$ is $(  p,b_*,\varepsilon )$-coupled with itself, meaning that (\ref{i22})
holds with $\Theta ^{1}=\Theta ^{2}=\theta$. Then, there exists a unique
invariant probability measure $\nu $ (that is $\theta _{s,t}(\nu )=\nu $ for
every $s\leq t$,  and one has the following exponential speed of convergence
\begin{equation}
\forall \mu \in \mathcal{P}_{p}(B), \exists C \in \mathbb{R}_+,\forall t\ge s, \ \quad W_{   p}^{p}(\theta _{s,t}(\mu ),\nu )\leq Ce^{-(b_*\wedge \varepsilon )(t-s)}
\label{i3}
\end{equation}
when $b_*\not= \varepsilon$ (an extra multiplicative factor $(1+(t-s))$ appears when $b_*= \varepsilon$).
Our second result concerns the approximation of $\nu$ by means of an approximation scheme for the grid $\pi $ with  $\gamma_{k}=\frac{1}{1+k}$, $k\in \mathbb{N}$. If we suppose in addition to the previous hypotheses that  $\Theta _{s,t}:\mathcal{P}_{p}(B)\rightarrow \mathcal{P}_{p}(B),s\leq t$ which is $(p,b_*,\varepsilon)$-coupled with $\theta$ in
the sense of (\ref{i22}), then for every $\mu \in \mathcal{P}_{p}(B)$, there exists $C\in \mathbb{R}_+$ such that
\begin{equation}
W_{   p}(\Theta _{t_0,t_{n}}^{\pi }(\mu ),\nu )\leq \frac{C}{n^{{\varepsilon \wedge b_* } }}.
\label{i4}
\end{equation}
Again, this result holds for $b_*\not = \varepsilon$ and an extra multiplicative factor $(1+\ln(1+n))$ appears when $b_*=\varepsilon$.

In the second part of the paper, we present examples that fit our general framework. The key is to prove the $(p,b_*,\varepsilon)$ - coupling property~\eqref{i2'} between the approximation family~$\Theta$ and the limit flow~$\theta$. In most cases, this is done by using a global contraction property on the coefficients and a short time analysis as in Section~\ref{sec:ex}.
 However, in
some specific examples related to the Langevin equation,  Eberle \cite{Eberle}  (see
also \cite{Schuh}) succeeds in proving the existence of the invariant measure and
exponential convergence to equilibrium under a much weaker assumption where the 
contraction with respect to the Euclidean norm has to be in force outside a
compact set only. The argument is based on the construction of a specific distance, which is
equivalent to the Euclidean distance, under which strict contraction holds on the whole space. On~\cite[Example 1]{Eberle},  we prove under the same conditions as the ones stated in~\cite{Eberle} that the Euler scheme converges, and we give as well a rate of convergence. In fact, the formalism developed in
our paper allows to use the estimates obtained in~\cite{Eberle} for the flow  in order to study the Euler approximation. 

Now, we give the main structure of the article.
The abstract framework is presented in Section \ref{Sec_Framework} together with our main general results. Section~\ref{sec_tools} provides interesting tools to fit our general framework. Namely, we show how the sewing lemma can help to get the coupling property~\eqref{i2'} for $W_1$ type distances, and give a Foster-Lyapunov type criterion to get uniform bounds on moments. Then,  we give in Section~\ref{sec_Eberle} the
results concerning the Langevin equation mentioned above. The following
Section~\ref{sec:ex} is devoted to some more examples. In Subsection \ref{sec:ref}, we consider a
stochastic equation with reflection driven by a Brownian motion with a $W_{2}$ distance. In Subsection~\ref{sec:neu}, we
deal with a recent neuronal model  which has been introduced in~\cite{DGLP} and~\cite{FoLo} and has recently been discussed in~\cite{CTV}. This is a mean field
type model driven by a Poisson Point process. We give two results, in terms
of $W_{1}$ and $W_{2}$ using different global contraction assumptions in each case. 
In Subsection \ref{sec:bol}, we consider a Boltzmann type
equation - that is an equation driven by a Poisson Point process with
intensity depending on the law of the solution. Such equations appear in
models of interacting particle systems. We consider two possible regimes: a finite
variation regime with the distance $W_{1}$ and 
a martingale regime with the distance $W_{2}$. The further analysis of the particle systems or of the approximation based on the ergodic measure arising from these law dependent examples will be done in a further research. 
As many examples consist in doing verification of the required hypotheses of our general framework, they are often similar in nature, and we have put some of them in the additional material in Appendix~\ref{App_additional}, keeping the most illustrative proofs in the main text.

 In general, throughout the article, unless constants are written explicitly, they may change from one line to the next even though they are denoted by the same letter $C$.

\section{Framework and main results}
\label{Sec_Framework}

We introduce first a general framework for convergence that will be applied
through out the article.

Let $(B,d)$ be a Polish metric space and $x_0\in B$ a given element of~$B$. For a probability $\mu$ on $B$ and $p\geq 1$, we denote 
\begin{equation} \label{def_pseudonorm}
\left\Vert \mu \right\Vert
_{p}^{p}=\int_{B} d(x,x_0)^{p}\mu (dx).
\end{equation}%
We note $\mathcal{P}(B)$ the space of
probability measures on $B$ and $\mathcal{P}_{p}(B)=\{\mu \in {\mathcal{P}}%
:\Vert \mu \Vert _{p}<\infty \}$ the set of probability measures with finite 
$p$ moment. We endow $\mathcal{P}_{p}(B)$ with the Wasserstein distance~$%
W_{p}$ defined by 
\begin{equation*}
	W_{p}^{p}(\mu ,\nu )=\inf_{\rho \in \Pi (\mu ,\nu
	)}\int_{B}\int_{B}d(x,y)^{p}\rho (dx,dy),
	\end{equation*}
where $\Pi (\mu ,\nu )=\{\rho \in {\mathcal{P}}(B\times B):\forall A\in 
\mathcal{B}(B),\ \rho (A\times B)=\mu (A),\rho (B\times A)=\nu (A)\}$ is the
set of all probability measures on $B\times B$ with marginals $\mu $ and $%
\nu $. 
We recall that $\mathcal{P}_{p}(B)$ is complete with respect to the distance 
$W_{p}$ (see e.g.~\cite[Theorem 6.9]{Villani}) and note that  $\left\Vert \mu \right\Vert
_{p}=W_p(\mu,\delta_{x_0})$. We denote for two
probability measures $\mu ,\nu \in {\mathcal{P}}_{p}(B)$, 
\begin{equation*}
\Gamma _{p}(\mu ,\nu )=1+\left\Vert \mu \right\Vert _{p}^{p}+\left\Vert \nu
\right\Vert _{p}^{p}.
\end{equation*}%
\blue{It will be used as a shorthand notation in different upper bounds.}
By the triangle inequality, we have $d(y,x_0)^{p}\le 2^{p-1}(d(x,x_0)^{p}+d(x,y)^p)$ and get
for any $\rho \in \Pi(\mu,\nu)$
\begin{equation*}
	\left\Vert \nu \right\Vert
	_{p}^{p}=\int_{B} \int_{B} d(y,x_0)^{p}\rho (dx,dy)\le 2^{p-1}\left( \left\Vert \mu \right\Vert_{p}^{p} + \int_{B}\int_{B}d(x,y)^{p}\rho (dx,dy) \right).
	\end{equation*}
Taking the infimum over $\rho$, we get 
$\left\Vert \nu \right\Vert
_{p}^{p}\le 2^{p-1}( \left\Vert \mu \right\Vert_{p}^{p} + W_p^p(\mu,\nu) ) $
and then  
\begin{equation}
\Gamma _{p}(\mu ,\nu )\le 1+2^p\left\Vert \mu \right\Vert
_{p}^{p}+2^{p} W_{p}^{p}(\mu ,\nu )  \label{E3}
\end{equation}

The main object of interest in this section is a family of deterministic
operators $\Theta _{s,t}:\mathcal{P}_{p}(B)\rightarrow \mathcal{P}_{p}(B)$, $%
s\leq t$, such that $\Theta _{s,s}(\mu )=\mu $ for all $\mu \in \mathcal{P}%
_{p}(B)$. A simple example is when $B$ is a Banach space and $\Theta _{s,t}(\mu )$ is the distribution
of a one step Euler scheme with initial distribution~$\mu $ (for further
details, see e.g. \cite{PaPa} or \cite{BaQi}). General approximation schemes are
constructed as follows. We consider an increasing sequence of times $\pi
=\{t_{i},i\in \mathbb{N}\}$ with $0\leq t_{i}<t_{i+1}$, and we denote $%
\gamma _{i}=t_{i}-t_{i-1}\in (0,h)$ for some $h\leq 1$. We then define 
\begin{equation}
\Theta _{t_{0},t_{i}}^{\pi }(\mu )=\Theta _{t_{i-1},t_{i}}\circ ....\circ
\Theta _{t_{0},t_{1}}(\mu ).
\label{def_thetapi}
\end{equation}%
Following the usual approach of decreasing steps for approximating the
invariant measure we also assume that $(\gamma _{n})_{n\in \mathbb{N}^{\ast }}$ is a nonincreasing
sequence. If $\Theta _{s,t}(\mu )$ represents
the law of a one step Euler scheme with initial distribution~$\mu $, then $%
\Theta _{t_{0},t_{i}}^{\pi }(\mu )$ is the distribution of the Euler scheme
on the grid $\pi $ at the final time $t_{i}$ with initial law $\mu $. \blue{ If $\Theta_{s,t}$ enjoys the flow property $\Theta_{s,t}(\Theta_{r,s}(\mu))= \Theta_{r,t}(\mu)$ for $r\le s \le t$, then $\Theta_{t_{0},t_{i}}^{\pi }(\mu )=\Theta_{t_{0},t_{i}}(\mu )$.}

\begin{remark}
	Let us note here that in general, for law-dependent dynamics, the approximation $\Theta_{s,t}(\mu)$ cannot be directly simulated and has to be approximated as well. For example, the Euler scheme associated to a McKean-Vlasov SDE involves the current distribution that can be approximated by means of a particle systems or an ergodic time average. We leave the study of the convergence of a fully simulatable scheme in a related general framework for further research. \end{remark}

\blue{We now introduce the main assumption on the coupling of two families of
operators $\Theta^1_{s,t},\Theta^2_{s,t}:{\mathcal{P}}_p(B)\to {\mathcal{P}}%
_p(B)$, $0\le s\le t$.
\begin{definition}
\label{coupling} Let $b_{\ast }>0$ and $\varepsilon >0$. We say that two
families of deterministic operators $\Theta ^{1}$ and $\Theta ^{2}$ are $%
(p,b_{\ast },\varepsilon )-$coupled if one may find some constants $%
C_{\ast }>0$ and $h\in (0,1)$ such that, for every $s\leq t\leq s+h$ and
every $\mu ^{i}\in \mathcal{P}_{p}(B)$, $i=1,2$, 
\begin{equation}
W_{p}^{p}(\Theta _{s,t}^{1}(\mu ^{1}),\Theta _{s,t}^{2}(\mu ^{2}))\leq
(1-b_{\ast }(t-s))W_{p}^{p}(\mu ^{1},\mu ^{2})+C_{\ast }\Gamma_{p}(\mu
^{1},\mu ^{2})(t-s)^{1+\varepsilon }.  \label{AN1}
\end{equation}%
Last, we say that a family $\Theta _{s,t}:{\mathcal{P}}_{p}(B)\rightarrow {%
\mathcal{P}}_{p}(B)$, $0\leq s\leq t$, is $(p,b_{\ast },\varepsilon )-$%
self-coupled  if~\eqref{AN1} holds with $\Theta ^{1}=\Theta ^{2}=\Theta $.
\end{definition}
In the case where $\Theta^1$ is the flow associated to a continuous time process and $\Theta^2$ is an approximation of $\Theta^1$, the constant $b_*$ can be thought as the contraction behaviour of the process, while $\varepsilon$ is related to the accuracy of the approximation in short time. }

\blue{For technical reasons, we will need upper bounds on $p$-moments, and we introduce} the notation, for $\delta>0$, 
\begin{equation*}
M_{p}(\Theta ,\mu,\delta )=\sup_{\|\pi\|<\delta }\sup_{n}\left\Vert \Theta
_{t_{0},t_{n}}^{\pi }(\mu )\right\Vert _{p}^{p},
\end{equation*}%
with the supremum taken on all the partitions $\pi $ such that $(t_{i+1}-t_i)_{i\ge 0}$ is decreasing and all $n\in \N$. 
\begin{definition}\label{def_pbounded}
If there exists $\delta>0$ such that $M_{p}(\Theta ,\mu ,\delta)<\infty $ for every $\mu \in \mathcal{P}_{p}(B)$ we
say that $\Theta $ is $p-$bounded.
\end{definition}
The fact that $M_{p}(\Theta ,\mu, \delta )$ is finite is not trivial. So,
in Subsection~\ref{sec:51} we give a Foster Lyapunov type criterion which allows
to prove it in several concrete examples. 
However, there are other examples (see Remark~\ref{rk_boundedtheta})
where this result is already given and consequently we do not need the
Foster Lyapunov criterion.

\blue{We are now in position to state the main result of our paper.
\begin{theorem}
	\label{th:main}
	Let $ (\theta_{s,t})_{0\le s\le t}$ be a family of operators on $\mathcal{P}_p(B)$ with the flow property ($\theta_{s,t}\circ \theta_{r,s}= \theta_{r,t}$ for $r\le s\le t$). We assume that $\theta$ is time homogeneous (i.e. $\theta_{s,t}=\theta_{0,t-s}$) and  such that for all $s\le t$, $\mu \mapsto \theta_{s,t}(\mu)$ is continuous with respect to~$W_p$. We assume besides that $\theta$  is $p$-bounded and $(p,b_*,\varepsilon)$ self coupled.  
     Then, there exists a unique invariant measure  $ \nu \in \mathcal{P}_{p}(B)$  associated to $ \theta $, and for every   $%
	\mu \in \mathcal{P}_{p}(B)$, there exists $A\in \mathbb{R}_+$ such that  
	\begin{equation*}
		W_{p}^{p}(\theta _{s,t}(\mu ),\nu )\leq A e^{-\varepsilon \wedge
			b_{\ast }(t-s)}(1+(t-s)^{\mathbf{1}_{b_{\ast }=\varepsilon }}).
	\end{equation*}
	Let  $ (\Theta_{s,t})_{0\le s\le t}$ be a family of operators on $\mathcal{P}_p(B)$ such that  $\theta$ and $\Theta$ are $(p,b_*,\varepsilon)$ coupled. Let  
 $a\ge 0$, $\gamma _{n}=\frac{1}{a+n}$ and $%
	t_{n}=s+\gamma _{1}+...+\gamma _{n}$, $n\in {\mathbb{N}}$. Then, for
	every $\mu \in \mathcal{P}_{p}(B)$, there exists $A\in \R_+$ such that 
	\begin{equation*}
		W_{p}^{p}(\Theta _{s,t_{n}}^{\pi }(\mu ),\nu )\leq A n^{-b_{\ast }\wedge
			\varepsilon }(1+\ln (n+1)^{\mathbf{1}_{b_{\ast }=\varepsilon }}  ).
	\end{equation*}
	\end{theorem}
    Theorem~\ref{th:main} is a consequence of Proposition~\ref{prop:29} and Theorem~\ref{FINAL}. Note that here, the constant $A$ may depend on~$\mu$. Under further assumptions, we will be able to get $A\le C(1+\|\mu\|_p^p)$, with $C\in \R_+$ that does not depend on~$\mu$.   }

	\blue{Many examples that illustrate this theorem are given in Sections~\ref{sec_Eberle} and~\ref{sec:ex}. Here, we presents quickly two illustrative examples. The first one, perhaps the simplest one, is the Ornstein-Uhlenbeck process 
	\begin{equation}\label{OU_exactS}
		X^1_{s,t}(X^1)=X^1 -\int_s^t k X^1_{s,u}(X^1)  du+\sigma(W_t-W_s), \ s\le t,
	\end{equation}
	with $k,\sigma>0$ and $X^1$ being a $p$-integrable real random variable independent of $W$.  For $\mu^1 \in \mathcal{P}_p(\mathbb{R})$, we define $\Theta^1_{s,t}(\mu^1) $ as the law of $X^1_{s,t}(X^1)$ when $X^1\sim \mu^1$.  For $\mu^2\in \mathcal{P}_p(\mathbb{R})$ and $X^2\sim \mu^2$ independent of $W$, we define  $\Theta^2_{s,t}(\mu^2)$ as the law of the Euler-Maruyama scheme $X^2_{s,t}(X^2)$, with 
		\begin{equation}\label{OU_euler}
		X^2_{s,t}(X^2)=(1-k(t-s) X^2) +\sigma(W_t-W_s), \ s\le t.
	\end{equation}
	In Appendix~\ref{App_OU}, we prove that $\Theta^1$ is $p$-bounded and $(p,b_*p,\varepsilon)$ self-coupled for any $\varepsilon>0$ and $b_*\in (0,k)$. Besides, $\Theta^1$ and $\Theta^2$ are $(p,b_*p,\varepsilon)$ coupled for any $\varepsilon\in(0,p)$ and $b_*\in (0,k)$. Therefore, Theorem~\ref{th:main} gives, for any $\mu \in \mathcal{P}_p(\mathbb{R})$ and $\zeta \in(0,1)$, the existence of $A\ge 0$ such that
	$$ W_p^p(\theta_{s,t}(\mu),\nu)\le A e^{-pb_*(t-s)}, \ W_p^p(\Theta^{\pi}_{s,t_n}(\mu),\nu) \le A n^{-\zeta p}.$$
	By explicit calculations, we show on the other hand that $W_p^p(\theta_{s,t}(\mu),\nu)=_{t\to \infty}O(e^{-k(t-s)})$ and $W_p^p(\Theta^{\pi}_{s,t_n}(\mu),\nu)=_{n\to \infty}O(n^{-{k \wedge 1}})$. Therefore, the rates of convergence given by Theorem~\ref{th:main} are almost optimal on this example. 
}

\blue{The second example is the  McKean Vlasov equation 
\begin{equation*}
		X^1_{s,t}=X^1 +\int_s^t \beta(X^1_{s,u}, \mathcal{L}(X^1_{s,u}) )  du+ \int_s^t \sigma(X^1_{s,u}, \mathcal{L}(X^1_{s,u}) )  d W_u, \ s\le t,
\end{equation*}
with $\beta:  \mathbb{R}^d \times \mathcal{P}_2(\mathbb{R}^d ) \to \mathbb{R}^d$,  $\sigma:  \mathbb{R}^d \times \mathcal{P}_2(\mathbb{R}^d ) \to \mathbb{R}^{d\times d}$ and $W$ is a $d$-dimensional Brownian motion. To keep simple notation, we assume $d=1$ even if what follows works for any $d\ge 1$. We make the following assumption: there exists $C,k, L\ge 0$ such that for all $x,y \in \R$ and $\mu^1,\mu^2 \in \mathcal{P}_2(\R)$,
\begin{enumerate}
	\item $|\beta(x,\mu^1)-\beta(y,\mu^2)|^2+|\sigma(x,\mu^1)-\sigma(y,\mu^2)|^2\le C (|x-y|^2 +W_2^2(\mu^1,\mu^2))$,
	\item $2(\beta(x,\mu^1)-\beta(y,\mu^2))(x-y)+|\sigma(x,\mu^1)-\sigma(y,\mu^2)|^2\le L W_2^2(\mu^1,\mu^2) -k |x-y|^2$.
\end{enumerate}
Then, \cite[Theorem 2.1]{Wang} gives the strong existence and uniqueness, and we note $\Theta^1_{s,t}(\mu^1)$ the law of $X^1_{s,t}$ when $X^1\sim \mu^1$. If $k>L$, \cite[Theorem 3.1]{Wang} gives $W_2^2(\Theta^1_{s,t}(\mu^1),\Theta^1_{s,t}(\mu^2))\le W_2^2(\mu^1,\mu^2) e^{-(k-L)(t-s)}$, and therefore $\Theta^1$ is $(2,b_*,\varepsilon)$ self coupled for any $b_* \in (0,k-L)$ and  $\varepsilon>0$.  \\ Let $\Theta^2_{s,t}(\mu^2)$ the law of $$X^2_{s,t}=X^2+\beta(X^2,\mu^2)(t-s)+\sigma(X^2,\mu^2)(W_t-W_s)$$ when $X^2\sim \mu^2$. By Itô formula and using (2), we get
\begin{align*}
	\mathbb{E}[(X^1_{s,t}-X^2_{s,t})^2]\le& \mathbb{E}[(X^1-X^2)^2] \\&+\int_s^t L W_2^2(\mathcal{L}(X^1_{s,u}), \mathcal{L}(X^2_{s,u}) ) -k \mathbb{E}[(X^1_{s,u}-X^2_{s,u})^2] du   \\
	&+ \int_s^t 2\mathbb{E}[(\beta(X^2_{s,u},\mathcal{L}(X^2_{s,u}) ) -\beta(X^2,\mu^2))(X^1_{s,u}-X^2_{s,u})] du  \\
	&\hspace{-3cm}+ \int_s^t \mathbb{E}[ (\sigma(X^1_{s,u},\mathcal{L}(X^1_{s,u}))-\sigma(X^2,\mu^2 ) )^2-(\sigma(X^1_{s,u},\mathcal{L}(X^1_{s,u}))-\sigma(X^2_{s,u},\mathcal{L}(X^2_{s,u}) ) )^2] du  
\end{align*} 
We now use $W_2^2(\mathcal{L}(X^1_{s,u}), \mathcal{L}(X^2_{s,u}) )\le \mathbb{E}[(X^1_{s,u}-X^2_{s,u})^2]$ for the first term of the right hand side. For $\delta>0$ arbitrarily small, we use  $2ab\le \delta a^2+ \delta^{-1}b^2$ and $(a+b)^2\le a^2 (1+\delta) +b^2(1+\delta^{-1})$ respectively for the second and third terms, together with the Lipschitz property of $b$ and $\sigma$, and we get 
\begin{align*}
	\mathbb{E}[(X^1_{s,t}-X^2_{s,t})^2] \le &\mathbb{E}[(X^1-X^2)^2]+ \int_s^t (L +\delta -k) \mathbb{E}[(X^1_{s,u}-X^2_{s,u})^2] du \\& + C_\delta \int_s^t \mathbb{E}[(X^2_{s,u}-X^2)^2]du.  
\end{align*} 
By standard estimates, we have  $\mathbb{E}[(X^2_{s,u}-X^2)^2]\le \Gamma_2(\mu^2,\mu^2)(u-s)^{1/2}$. By taking $X^1$ and $X^2$ optimally coupled for $W_2$ and using Gronwall lemma, we deduce that $\Theta^1$ and $\Theta^2$ are $(2,b_*,1/2)$ coupled for any $b_* \in (0,k-L)$. We can thus apply Theorem~\ref{th:main} to $\theta=\Theta^1$ and $\Theta=\Theta^2$ and  get\begin{equation}\label{cv_McKean}W_{2}^{2}(\Theta_{s,t_{n}}^{\pi }(\mu ),\nu )\leq A n^{- b_*\wedge \frac{1}{2}  },  \end{equation}
for any $b_*\in (0,1/2]\cap(0,k-L)$, where $t_n=s +\sum_{i=1}^n \frac {1}{1+i}$.\\
}
\begin{remark}	 As the above example shows, the most important constants that appear in Definition~\ref{coupling}, Eq.~\eqref{AN1},  are $b_{\ast }$ and $%
	\varepsilon >0$. \blue{In practical examples, $b_*$ corresponds to the mean-reverting behaviour of the process while $\varepsilon$ is mostly related to the short time approximation error of the chosen scheme (see also examples in Section~\ref{sec:ex}).} They determine the speed of
	convergence in our approximation results of the stationary distribution in Theorem~\ref{th:main}. \blue{This will appear clearly in the proof of Lemma~\ref{lem_wass}.}
	In contrast, the constant $C_{\ast }$ only influences the constants
	multiplying the speed on convergence, and is in
	this sense is less important.
\end{remark}

\subsection{Coupling on a time grid}

In the sequel we will consider a decreasing sequence $\gamma _{n}\in
(0,1),n\in \N$ such that \blue{$\lim_{n \to \infty }\gamma_n=0$,} $\sum_{n\ge 1}\gamma _{n}=\infty $ and we denote 
\begin{equation}
\pi =\{t_{n},n\in \N\}\quad \text{with}\quad t_{n}=t_{0}+\sum_{i=1}^{n}\gamma _{i}.
\label{pi}
\end{equation}%

For $b>0$ and $\varepsilon >0$ we define 
\begin{equation}
\sigma _{b,\varepsilon }(n):=\sum_{i=1}^{n}e^{-b(t_{n}-t_{i})}\gamma
_{i}^{1+\varepsilon }.  \label{sigma}
\end{equation}
\blue{ Let us observe that $e^{-b(t_{n}-t_{i})}\gamma
_{i}^{1+\varepsilon }\le \gamma_1^\varepsilon \int_{t_{i-1}}^{t_i} e^{-b(t_{n}-s)} e^{b\gamma_1} ds $ because $\gamma$ is decreasing, and therefore 
\begin{equation}
\sigma _{b,\varepsilon }(n)\le \gamma
_{1}^{\varepsilon } \frac{e^{b\gamma_1}}{b}.  \label{majo_sigma}
\end{equation}
}

\begin{lemma}\label{lem_wass}
 Let $\Theta^{1}$ and $\Theta^{2}$ be families  which are $%
(p,b_{\ast },\varepsilon )$-coupled for some $b_{\ast },\varepsilon >0$.
 We note $C_{\ast }>0$, 
$h\in (0,1)$ some constants such that~\eqref{AN1} holds for $\Theta
_{s,t}^{1}$ and~$\Theta _{s,t}^{2}$, $s\leq t\leq s+h$. We also assume that  $\Theta ^{1}$ is $p-$bounded and $%
h\in (0,1)$ is sufficiently small to have $M_{p}(\Theta^1 ,\mu,h )<\infty$ and \blue{
\begin{equation}
2^p  C_{\ast } h^{\varepsilon} \frac{e^{bh}}{b} \leq  \frac{1%
}{2}. \label{AN14bis}
\end{equation}}
Then, for any positive decreasing sequence $(\gamma_n)_{n\ge 1}$ such that $\gamma_1<h$ and $(t_n)_{n\ge 0}$ be defined as in~\eqref{pi}, we have for every $n\in {\mathbb{N}}$ 
\begin{align}
&W_{p}^{p}(\Theta _{t_{0},t_{n}}^{1,\pi }(\mu ^{1}),\Theta
_{t_{0},t_{n}}^{2,\pi }(\mu ^{2}))  \label{AN15bis}\\
&\leq e^{-b_{\ast
}(t_{n}-t_{0})}W_{p}^{p}(\mu ^{1},\mu ^{2})+ 2^{p+1}C_* ( 1+ M_{p}(\Theta ^{1},\mu ^{1},h)+W_p^p(\mu_1,\mu_2))\sigma _{b_{\ast },\varepsilon }(n). \notag
\end{align}%
Besides, $M_{p}(\Theta^2 ,\mu,h )<\infty$ for every $\mu \in \mathcal{P}_p(B)$ and $\Theta ^{2}$ is $p-$bounded.
\end{lemma}

\begin{proof} We denote $a_{k}=$ $W_{p}^{p}(\Theta
_{t_{0},t_{k}}^{1,\pi }(\mu ^{1}),\Theta _{t_{0},t_{k}}^{2,\pi }(\mu ^{2}))$ \blue{ and first prove an upper bound for this sequence.} 
We write, for $n\ge 1$, 
\begin{equation*}
a_{n}=e^{-b_{\ast }(t_{n}-t_{0})}a_{0}+\sum_{k=1}^{n}(e^{-b_{\ast
}(t_{n}-t_{k})}a_{k}-e^{-b_{\ast }(t_{n}-t_{k-1})}a_{k-1}).
\end{equation*}%
\blue{We first observe 
\begin{equation*}
\forall n, \  \max_{k\leq n}\left\Vert \Theta _{t_{0},t_{k-1}}^{1,\pi }(\mu
^{1})\right\Vert _{p}^{p}\leq M_{p}(\Theta ^{1},\mu ^{1},h)=:{\bf M}.
\end{equation*}%
}
Using the coupling property (\ref{AN1})\textbf{\ }first and (\ref{E3})\
then, we obtain 
\begin{align}
&e^{-b_{\ast }(t_{n}-t_{k})}a_{k}-e^{-b_{\ast }(t_{n}-t_{k-1})}a_{k-1}
\label{AN12} \\
&=e^{-b_{\ast }(t_{n}-t_{k})}(1-e^{-b_{\ast }\gamma
	_{k}})a_{k-1}+e^{-b_{\ast }(t_{n}-t_{k})}(a_{k}-a_{k-1})  \notag \\
&\leq e^{-b_{\ast }(t_{n}-t_{k})}(b_{\ast }\gamma _{k}a_{k-1}-b_{\ast
}\gamma _{k}a_{k-1}+C_{\ast }\Gamma _{p}(\Theta _{t_{0},t_{k-1}}^{1,\pi
}(\mu ^{1}),\Theta _{t_{0},t_{k-1}}^{2,\pi }(\mu ^{2}))\gamma
_{k}^{1+\varepsilon })  \notag \\
&\leq C_{\ast }e^{-b_{\ast }(t_{n}-t_{k})}\!\!\left( \!1+ \! 2^p \left\Vert \Theta
_{t_{0},t_{k-1}}^{1,\pi }(\mu ^{1})\right\Vert
_{ p}^{p}\!+\! 2^{p} W_{p}^{p}(\Theta _{t_{0},t_{k-1}}^{1,\pi }(\mu
^{1}),\Theta _{t_{0},t_{k-1}}^{2,\pi }(\mu ^{2}))\! \right) \!\gamma
_{k}^{1+\varepsilon }  \notag \\
&\leq C_{\ast }e^{-b_{\ast }(t_{n}-t_{k})}\left(1+
2^p {\bf M} + 2^p a_{k-1} \right) \gamma _{k}^{1+\varepsilon }.  \notag
\end{align}
Taking the sum, we get
\begin{align*}
a_{n} &\leq e^{-b_{\ast }(t_{n}-t_{0})}a_{0}+C_{\ast }\left(1+ 2^p {\bf M}
\right) \sum_{k=1}^{n}e^{-b_{\ast }(t_{n}-t_{k})}\gamma _{k}^{1+\varepsilon
}\\&\quad+C_{\ast }2^p \sum_{k=1}^{n}e^{-b_{\ast }(t_{n}-t_{k})}\gamma
_{k}^{1+\varepsilon }a_{k-1} \\
&\leq e^{-b_{\ast }(t_{n}-t_{0})}a_{0}+C_{\ast }\left(1+ 2^p {\bf M} + 2^p \max_{k\leq n-1}a_{k}
\right)  \sigma _{b_{\ast
	},\varepsilon }(n)  \\
&\leq a_{0}+\frac 12\left( 2^{-p}+{\bf M}  +
\max_{k\leq n-1}a_{k}\right), 
\end{align*}
the last inequality being a consequence of our hypothesis \eqref{AN14bis} \blue{and the upper bound~\eqref{majo_sigma}. Taking the maximum, }
this gives%
\begin{equation*}
\max_{i\leq n}a_{i}\leq a_{0}+ \frac 12 \left( 1+{\bf M} \right) +\frac{1}{2}%
\times \max_{k\leq n-1}a_{k}
\end{equation*}%
and finally %
\begin{equation}
\max_{i\leq n}a_{i}\leq 2a_{0} + {\bf M} +1 .
\label{AN11}
\end{equation}

We now come back to (\ref{AN12}) and we use~\eqref{E3} and (\ref{AN11}) in
order to get
\begin{align*}
e^{-b_{\ast }(t_{n}-t_{k})}a_{k}-e^{-b_{\ast }(t_{n}-t_{k-1})}a_{k-1} 
&\leq e^{-b_{\ast }(t_{n}-t_{k})}C_{\ast }\left(1+2^p {\bf M} +2^p a_{k-1}  
\right) \gamma
_{k}^{1+\varepsilon } \\
&\leq e^{-b_{\ast }(t_{n}-t_{k})}2^{p+1}C_{\ast }\left(1+ {\bf M} +a_{0}  
\right) \gamma
_{k}^{1+\varepsilon } .
\end{align*}

 Taking the sum and using Lemma~\ref{lem_tech}, we conclude with 
 \begin{align*}
 a_{n} &\leq e^{-b_{\ast }(t_{n}-t_{0})}a_{0}+2^{p+1} C_* (1+{\bf M}+a_0) \sum_{k=1}^{n}e^{-b_{\ast }(t_{n}-t_{k})}\gamma _{k}^{1+\varepsilon }
 \\
 &=e^{-b_{\ast }(t_{n}-t_{0})}a_{0}+2^{p+1} C_* (1+{\bf M}+a_0) \sigma_{b_{\ast },\varepsilon }(n) . 
 \end{align*}
Now, we observe that 
$$\max_{k\leq n}\left\Vert \Theta _{t_{0},t_{k-1}}^{2,\pi }(\mu )\right\Vert
_{p}^{p}\leq 2^{p-1} \max_{k\leq n}\left\Vert \Theta _{t_{0},t_{k-1}}^{1,\pi }(\mu )\right\Vert
_{p}^{p} + 2^{p-1}  \max_{k\leq n} a_k.$$
\blue{Since the sequence $(t_i)_{i\ge 0}$ is arbitrary among those such that $t_1-t_0<h$ and $(t_{i+1}-t_i)_{i\ge 0}$ is decreasing,} this gives by~\eqref{AN11}, $M_{p}(\Theta^2 ,\mu,h )\le 2^p M_{p}(\Theta^1 ,\mu,h ) +2^p a_0 +2^{p-1} <\infty$ for every $\mu \in \mathcal{P}_p(B)$.
\end{proof}

To go further, we will make the additional assumption on $(\gamma_n)$  
\blue{
\begin{equation}\label{def_varpi}
\varpi(\pi ):=\limsup_{n\rightarrow \infty }\frac{\gamma
_{n}-\gamma _{n+1}}{\gamma _{n+1}^{2}}\in [0,\infty ),
\end{equation}%
that is needed for} the following technical lemma.

\begin{lemma}
\label{lem_tech} \textbf{A}. 
One
may find $C_{b,\varepsilon }$ such that for all $ n\in\mathbb{N}^* $
\blue{\begin{eqnarray}
\sigma _{b,\varepsilon }(n) \leq \begin{cases}  C_{b,\varepsilon }\gamma_{n}^{\varepsilon } \text{ if }  b> \varepsilon \varpi(\pi),\\
C_{b,\lambda }\gamma _{n}^{\lambda } \text{ if }  \varpi(\pi)>0 \text{ and }\lambda <\frac{b }{\varpi(\pi )}\leq \varepsilon . 
\end{cases}
\end{eqnarray}
}
The constant $C_{b,\varepsilon }$ depends on the sequence $(\gamma
_{n})_{n\in \N}$ as well, but does not depend on~$n$.

\textbf{B}. For the particular case $\gamma _{n}=\frac{1}{n+h^{-1}}$, one has $%
\varpi(\pi )=1$ and 
$$ \sigma
_{b,\varepsilon }(n)\leq \begin{cases} C_{b,\varepsilon }n^{-b\wedge \varepsilon } \text{ if }  b\neq \varepsilon,\\
C_{b,\varepsilon
}n^{-\varepsilon }\ln (1+n) \text{ if }  b= \varepsilon. 
\end{cases}$$
 Besides, for $s<t$, $t_{n}=s+\sum_{i=1}^{n}\gamma _{i}$ and  $%
n=n(t)\in \N$ such that $t_{n}\leq t<t_{n+1}$, one may find a constant that we still denote $C_{b,\varepsilon }$ such that 
\begin{equation}
\sigma _{b,\varepsilon }(n)\leq C_{b,\varepsilon }e^{-b\wedge
\varepsilon (t-s)}(1+(t-s)^{1_{b=\varepsilon }}).  \label{AN6'}
\end{equation}
\end{lemma}

A similar result has been used by Pag\`es and Panloup~\cite[Lemma A.3(i)]{PaPa}
with $\varepsilon =1$ and later on by Bally and Qin~\cite[Lemma A.1(B)]{BaQi}%
. We give a proof of our slightly different version in the Appendix~\ref%
{sec:app}.
\blue{
\begin{remark}
Let $\beta\in(0,1)$ and $\gamma _{n}=\frac{1}{n^\beta+h^{-1}}$. This sequence fulfills the required assumptions with $\varpi(\pi)=0$. Then, we have $t_n\sim \frac{n^{1-\beta}}{1-\beta}$ as $n\to \infty$. By Lemma~\ref{lem_tech}, we have $\sigma_{b,\varepsilon}(n) \le C_{b,\varepsilon} n^{-\beta \varepsilon}$. Besides,  for $s<t$, $t_{n}=s+\sum_{i=1}^{n}\gamma _{i}$ and  $%
n\in \N$ such that $t_{n}\leq t<t_{n+1}$, we have $n \sim \left( (1-\beta) (t-s) \right)^{\frac{1}{1-\beta}}$, so that $\sigma_{b,\varepsilon}(n) \le C (t-s)^{\frac{-\beta \varepsilon}{1-\beta}}$. As $t-s \to \infty$, this convergence is less sharp than~\eqref{AN6'}. For this, and also to avoid multiple cases, we will only work with $\gamma _{n}=\frac{1}{n+h^{-1}}$ in the paper. 
\end{remark}
}

\begin{corollary}\label{cor_speed}
We suppose that the hypotheses of Lemma~\ref{lem_wass} are in force and we
consider the grid $t_{n}=t_{0}+\sum_{i=1}^{n}\frac{1}{i+h^{-1}}.$ Then, for $b_*\neq \varepsilon$, $\mu^1,\mu^2 \in \mathcal{P}_p(B)$,
\begin{align*}
	W_{p}^{p}(\Theta _{t_{0},t_{n}}^{1,\pi }(\mu ^{1}),\Theta
	_{t_{0},t_{n}}^{2,\pi }(\mu ^{2}))\leq &A n^{-b_{\ast }\wedge \varepsilon}, \ 
\end{align*}
with $A=\left( (2^{p+1}C_*C_{b_{\ast },\varepsilon
	}+1) W_{p}^{p}(\mu ^{1},\mu
	^{2}) + 2^{p+1}C_*C_{b_{\ast },\varepsilon
	} (1+ M_{p}(\Theta ^{1},\mu ^{1},h)) \right)$. 
In the case $%
b_*=\varepsilon $, one has to multiply this upper bound by $1+\ln (1+n)$.\footnote{In the same way, we have $ W_{p}^{p}(\Theta _{t_{0},t_{n}}^{1,\pi }(\mu ^{1}),\Theta
	_{t_{0},t_{n}}^{2,\pi }(\mu ^{2}))\leq  A (1+ (t_{n}-t_{0}))e^{-b_{\ast
		}\wedge \varepsilon (t_{n}-t_{0})}.$} \end{corollary}
This is a direct application of Lemma~\ref{lem_wass} and Lemma~\ref{lem_tech}~{\bf B}.

\blue{\begin{remark}
In Lemma~\ref{lem_wass}, the condition $\gamma_1<h$ is required to have $\gamma_n \in (0,h)$ for all $n$, which allows to use~\eqref{AN1} for all time steps.  This is why we then take $\gamma_n= \frac{1}{n+h^{-1}}$ in Corollary~\ref{cor_speed}.\\
If  the condition $\gamma_1<h$ is not satisfied, we can still define $n(h)=\inf\{ n\ge 1 : \gamma_n<h \}$  apply Lemma~\eqref{lem_wass} (and Corollary~\ref{cor_speed}) for $n\ge n(h)$ on the grid $\{t_{n(h)-1},\dots,t_n\}$. For example, let $\gamma_n= \frac{1}{a+n}$ for some $a\ge 0$ and $t_n=t_{0}+\sum_{i=1}^{n}\frac{1}{a+i}$. We have  by Corollary~\ref{cor_speed}
\begin{equation}\label{AN16}
    W_{p}^{p}(\Theta _{t_{0},t_{n}}^{1,\pi }(\mu ^{1}),\Theta
	_{t_{0},t_{n}}^{2,\pi }(\mu ^{2}))\leq A (n-n(h)+1)^{-b_{\ast }\wedge \varepsilon} \le A' n^{-b_{\ast }\wedge \varepsilon},\end{equation}
with $A=(2^{p+1}C_*C_{b_{\ast },\varepsilon
	}+1) W_{p}^{p}(\Theta _{t_{0},t_{n(h)-1}}^{1,\pi }(\mu ^{1}) , \Theta _{t_{0},t_{n(h)-1}}^{2,\pi }(\mu ^{2})) + 2^{p+1}C_*C_{b_{\ast },\varepsilon
 	} (1+ M_{p}(\Theta ^{1},\Theta _{t_{0},t_{n(h)-1}}^{1,\pi }(\mu ^{1}) ,h)) $ and $A'=A n(h)^{b_{\ast }\wedge \varepsilon}$.\\
    Besides, if $\Theta^1$ and  $\Theta^2$ satisfy the  $p$-Foster Lyapunov criterion~\eqref{AN2}, we have $A'\le \tilde{A}'\Gamma_p(\mu^1,\mu^2)$ for a constant $\tilde{A}'\ge 0$ that does not depend on $(\mu^1,\mu^2)$. 
\end{remark}
}


\subsection{Invariant measure}

Let $\theta _{s,t}$ $:\mathcal{P}_{p}(B)\rightarrow \mathcal{P}_{p}(B),s\leq
t$ be a flow, that is $\theta _{r,t}\circ \theta _{s,r}=\theta _{s,t},s\le r
\le t$. 
\begin{remark}
For a flow~$\theta$, we have 
$$M_{p}(\theta ,\mu,\delta )=\sup_{\pi : t_i-t_{i-1}\le \delta }\sup_{n}\left\Vert \theta
_{t_{0},t_{n}}^{\pi }(\mu )\right\Vert _{p}^{p}=\sup_{t_0<t_n}\left\Vert \theta
_{t_{0},t_{n}}(\mu )\right\Vert _{p}^{p}.$$
This does not depend on~$\delta$, and we simply note this quantity $M_{p}(\theta ,\mu)$ in this case.
\end{remark}

We look in this subsection for hypotheses which guarantee that $%
\theta$ admits an invariant probability measure $\nu$, i.e. $\nu
=\theta_{s,t}(\nu)$ for every $s<t$. Obviously, if such a~$\nu$ exists, it
has to be a stationary law since $\nu=\lim_{t\to \infty}\theta_{s,t}(\nu)$.
We start with a preliminary estimate.

\begin{corollary}
\label{cor:2.8} Let $\theta _{s,t}$ be a flow which is $p$-bounded and
is $(p,b_{\ast },\varepsilon )$-self-coupled for some $b_{\ast
},\varepsilon >0$. We note $C_{\ast }>0$, $h\in (0,1)$ some constants such
that~\eqref{AN1} and~\eqref{AN14bis} hold. Then, for every $\mu ^{1},\mu ^{2}\in \mathcal{P}%
_{p}(B)$ and every $s<t$ 
\begin{equation*}
W_{p}^{p}(\theta _{s,t}(\mu ^{1}),\theta _{s,t}(\mu ^{2}))\leq A e^{-b_{\ast }\wedge \varepsilon (t-s)}(1+
(t-s)^{1_{b_*=\varepsilon }})
\end{equation*}%
where $A=\left( (2^{p+1}C_*\tilde{C}_{b_{\ast },\varepsilon
	}+1) W_{p}^{p}(\mu ^{1},\mu
	^{2}) + 2^{p+1}C_* \tilde{C}_{b_{\ast },\varepsilon
	} (1+ M_{p}(\theta ,\mu ^{1})) \right)$ and $\tilde{C}_{b_{\ast },\varepsilon
	}=C_{b_{\ast },\varepsilon
    }+1$. \blue{ Besides, $A\le C (\Gamma_p(\mu^1,\mu^2)+ M_p(\theta,\mu^1))$ with $C$ that does not depend on~$\mu^1$ and $\mu^2$.}
\end{corollary}
\noindent \blue{Let us stress that this inequality is interesting only when $t-s$ is large as it gives an exponential rate of convergence, but is not sharp when $t \to s$. }

\begin{proof}
Let $s<t$ and $\gamma _{k}=\frac{1}{k+h^{-1}},t_{k}=t_{0}+\gamma
_{1}+\dots+\gamma _{k}.$ Let $n\in {\mathbb{N}}$ be such that $t_{n}\leq
t<t_{n+1}$. Then, we consider the grid $t_{k}^{\prime }=t_{k}$ for $k\in
\{0,\dots ,n\}$ and $t_{n+1}^{\prime }=t\leq t_{n+1}$. This means that $%
\gamma _{k}^{\prime }=\gamma _{k}$ for $k=1,\dots,n$ and $\gamma
_{n+1}^{\prime }<\gamma _{n+1}.$ Recall that, by~(\ref{AN6'}), $\sigma
_{b_*,\varepsilon }(n)\leq C_{b_*,\varepsilon } e^{-b_*\wedge \varepsilon
(t-s)}(1+ (t-s)^{1_{b_*=\varepsilon }})$ and by (\ref{A3}) $\gamma _{n}\leq
h e^{-(t_{n}-s).}$. Then, with $\sigma _{b_*,\varepsilon }^{\prime }(n)$ 
associated to $(\gamma _{i}^{\prime })_{i\ge 1}$, we have 
\begin{align*}
\sigma _{b_{\ast },\varepsilon }^{\prime }(n+1)
&=\sum_{i=1}^{n+1}e^{-b_{\ast }(t_{n+1}^{\prime }-t_{0})}(\gamma
_{i}^{\prime })^{1+\varepsilon } \\
&= e^{-b_{\ast }(t_{n+1}^{\prime
	}-t_{n})}\sigma _{b_{\ast },\varepsilon }(n)+  (\gamma'_{n+1})^{1+\varepsilon }
\\
&\leq (C_{b_*,\varepsilon }+1)e^{-b_*\wedge \varepsilon
	(t-s)}(1+ (t-s)^{1_{b_*=\varepsilon }}).
\end{align*}

Since $\gamma _{k}\leq h$ one can thus use Lemma~\ref{lem_wass} to get 
\begin{eqnarray*}
&&W_{p}^{p}(\theta _{s,t}(\mu ^{1}),\theta _{s,t}(\mu
^{2}))=W_{p}^{p}(\theta _{t_{0},t_{n+1}^{\prime }}(\mu ^{1}),\theta
_{t_{0},t_{n+1}^{\prime }}(\mu ^{2})) \\
&&\ \ \leq e^{-b_{\ast }(t_{n+1}^{\prime }-t_{0})}W_{p}^{p}(\mu ^{1},\mu
^{2})+
2^{p+1}C_* ( 1+ M_{p}(\theta ,\mu ^{1})+W_p^p(\mu_1,\mu_2))\sigma'_{b_{\ast },\varepsilon }(n+1).
\end{eqnarray*}
Then, the above inequalities give the claim.
\end{proof}

\begin{example} 
Let us consider the following time-dependent Ornstein-Uhlenbeck process 
\begin{equation*}
X_{s,t}=X_s +\int_s^t k(\zeta(u)-X_u)du+\sigma(W_t-W_s), \ s\le t,
\end{equation*}
with $k,\sigma>0$ and $\zeta:{\mathbb{R}}_+ \to {\mathbb{R}}_+$ being a
bounded function. Here, $X_s$ denotes a square integrable random variable that is independent of $(W_t-W_s, t\ge s)$. By standard calculations, $X_{s,t}=X_s e^{-k(t-s)}+ \int_s^t
\zeta(u) e^{-k(t-u)} du + \sigma \int_s^t e^{-k(t-u)} dW_u$. The associated
flow $\theta_{s,t}(\mu)=\mathcal{L}(X_{s,t})$ when $X_s\sim \mu \in \mathcal{%
P}_2(\R)$ satisfies for $\mu_1,\mu_2 \in \mathcal{P}_2(\R)$ $$W_2^2(\theta_{s,t}(\mu_1),\theta_{s,t}(\mu_2))\le W_2^2(%
\mu_1,\mu_2)e^{-2k(t-s)}.$$ Let $\nu_{s,t}=\mathcal{N}%
(\int_s^t \zeta(u) e^{-k(t-u)} du,\frac{\sigma^2}{2k})$ be the law of $$\int_s^t \zeta(u) e^{-k(t-u)} du+\sigma \int_s^t  e^{-k(t-u)} dW_u +\sqrt{\frac{\sigma^2}{2k}e^{-2k(t-s)}}Y$$ with $Y\sim\mathcal{N}(0,1)$ independent of the other variables. We have
$$W^2_2(\theta_{s,t}(\mu),\nu_{s,t})\le 2 \mathbb{E}[X_s^2] e^{-2 k(t-s)} +\frac{\sigma^2}{k}e^{-2k(t-s)} =_{t\to \infty}O(e^{-2k(t-s)}).$$
 However,  we can choose $\zeta$ so that $\int_s^t \zeta(u)
e^{-k(t-u)} du$ does not converge as $t\to \infty$ (take for example $\zeta(u)=\cos(u)$) and therefore $\nu_{s,t}$ (and thus $\theta_{s,t}(\mu)$) does not weakly converge. 
\end{example}

As illustrated by this simple example, time-inhomogenous flow may not have a
stationary distribution even though they are self-coupled. To go further, we
consider time homogeneous flows which means that 
\begin{equation*}
\theta _{s,t}(\mu )=\theta _{0,t-s}(\mu )\text{ for }0\leq s\leq t.
\end{equation*}%
From this, we get easily $\theta _{0,t}\circ \theta _{0,t^{\prime }}=\theta
_{0,t+t^{\prime }}$ for any $t,t^{\prime }\geq 0$ and thus 
\begin{equation}
\theta _{t,t+\delta }\circ \theta _{s,t}=\theta _{0,t-s+\delta }=\theta
_{s,t}\circ \theta _{t,t+\delta },\ \delta \geq 0,\ t\geq s\geq 0.
\label{flow_th}
\end{equation}%
In addition, we consider a flow $\theta _{s,t}$ which is continuous for $%
W_{p}$, in the sense that for every fixed $s\leq t$ 
\begin{equation}
\lim_{n\rightarrow \infty }W_{p}(\mu _{n},\mu )=0\quad \Rightarrow \quad
\lim_{n\rightarrow \infty }W_{p}(\theta _{s,t}(\mu _{n}),\theta _{s,t}(\mu
))=0.  \label{eq:++}
\end{equation}
\begin{proposition}
\label{prop:29} Consider $\theta _{s,t}:{\mathcal{P}}_{p}(B)\rightarrow {%
\mathcal{P}}_{p}(B)$ a continuous time-homogeneous flow, which is $p$-bounded and is $(p,b_{\ast },\varepsilon )$ self-coupled for some $h\in
(0,1)$, $b_{\ast },\varepsilon >0.$ Then, there exists a unique probability
measure $\nu \in \mathcal{P}_{p}(B)$ such that for every $s\leq t$ one has $%
\nu =\theta _{s,t}(\nu ).$ It satisfies $\Vert \nu \Vert
_{p}^{p}=M_{p}(\theta ,\nu )$ and \blue{there exists $C\in \mathbb{R}_+$ such that for every $s\leq t$ and every $\mu \in \mathcal{P}_{p}(B)$, 
\begin{equation}
W_{p}^{p}(\theta _{s,t}(\mu ),\nu )\leq C(1+M_p(\theta,\mu)) e^{-\varepsilon \wedge
b_{\ast }(t-s)}(1+(t-s)^{\mathbf{1}_{b_{\ast }=\varepsilon }}).
\label{AN14}
\end{equation}%
}
\end{proposition}

\begin{proof}
Let $s\leq t,\delta >0$ and $\mu \in \mathcal{P}_{p}(B)$. We have 
\begin{equation*}
W_{p}(\mu ,\theta _{t,t+\delta }(\mu ))\leq \left\Vert \mu \right\Vert
_{p}+\left\Vert \theta _{t,t+\delta }(\mu )\right\Vert _{p}\leq
2M_{p}(\theta ,\mu )
\end{equation*}%
so that, using the time homogeneous flow property~\eqref{flow_th} and
Corollary~\ref{cor:2.8} with $\mu ^{1}=\mu $ and $\mu ^{2}=\theta
_{t,t+\delta }(\mu )$ 
\begin{align}
W_{p}^{p}(\theta _{s,t}(\mu ),\theta _{s,t+\delta }(\mu
))&=W_{p}^{p}(\theta _{s,t}(\mu ),\theta _{s,t}(\theta _{t,t+\delta }(\mu
))) \notag \\
&\leq A e^{-\varepsilon \wedge b_{\ast }(t-s)}(1+(t-s))^{\mathbf{1}_{b_{\ast }=\varepsilon }},\label{AN15}
\end{align}%
\blue{with $A\le C(1 +M_p(\theta,\mu))$ and $C\in \R_+$ does not depend on~$\mu$ (note that $\|\mu\|_p^p\le M_p(\theta,\mu)$ since $\theta$ is a continuous flow.).}
Therefore, if $t_{n}\rightarrow \infty $ , we get that $\theta
_{s,t_{n}}(\mu )$, $n\in {\mathbb{N}}$ is a Cauchy sequence. Since $\mathcal{%
P}_{p}(B)$ is complete with respect to $W_{p}$, we define its limit in the
Wasserstein metric as $\nu :=\lim_{n\rightarrow \infty }\theta
_{s,t_{n}}(\mu )\in \mathcal{P}_{p}(B)$. Coming back to \eqref{AN15} and
passing to the limit we get 
\begin{equation*}
W_{p}^{p}(\theta _{s,t}(\mu ),\nu )\leq A e^{-\varepsilon \wedge
b_{\ast }(t-s)}(1+(t-s) )^{\mathbf{1}_{b_{\ast }=\varepsilon }},
\end{equation*}%
This proves \eqref{AN14}.

We now check that $\nu =\lim_{n\rightarrow \infty }\theta _{s^{\prime
},t_{n}}(\mu )$ for any $s^{\prime }\geq 0$. For $s^{\prime }\leq s$, we
write for $n$ sufficiently large 
\begin{equation*}
W_{p}^{p}(\theta _{s^{\prime },t_{n}}(\mu ),\theta _{s,t_{n}}(\mu
))=W_{p}^{p}(\theta _{s,t_{n}}(\theta _{s^{\prime },s}(\mu )),\theta
_{s,t_{n}}(\mu ))\rightarrow 0,
\end{equation*}%
by using Corollary~\ref{cor:2.8}. We proceed similarly for $s^{\prime }>s$.

This property gives easily that $\nu$ is invariant: we get indeed from the
continuity property \eqref{eq:++} that 
\begin{equation*}
\theta_{s,t}(\nu)=\lim_{n\to \infty}
\theta_{s,t}(\theta_{t,t_n}(\mu))=\lim_{n\to \infty} \theta_{s,t_n}(\mu)=\nu.
\end{equation*}

Let us prove uniqueness. Let $\nu^{\prime }\in\mathcal{P}_p(B)$ be another
invariant measure. Then, using Corollary~\ref{cor:2.8}, $W^p_{p}(\nu
,\nu ^{\prime })=W_{p}^{p}(\theta_{s,t}(\nu), \theta_{s,t}(\nu ^{\prime
}))\rightarrow 0$ as $t\rightarrow \infty$.
\end{proof}

\bigskip

We discuss now the rate of convergence to the invariant measure using an
iterative scheme algorithm.

\begin{theorem}
\label{FINAL} Let $\theta $ be a time homogeneous continuous flow on ${%
\mathcal{P}}_{p}(B)$ which is $p$-bounded and is $(p,b_{\ast
},\varepsilon )$ self-coupled for some $h\in (0,1)$, $b_{\ast },\varepsilon >0$ and 
$\nu $ the invariant measure obtained in Proposition~\ref{prop:29}. Let $%
\Theta _{s,t}:{\mathcal{P}}_{p}(B)\rightarrow {\mathcal{P}}_{p}(B)$, $s\leq
t $, be a family of applications which is $(p,b_{\ast },\varepsilon )$
coupled with $\theta _{s,t}$. Let $h\in(0,1)$ be such that both  \eqref{AN14bis} and the coupling property~\eqref{AN1} between $\theta$ and $\Theta$  hold. \blue{ Let $a\ge 0$, $\gamma _{n}=\frac{1}{a+n}$ and $%
t_{n}=s+\gamma _{1}+...+\gamma _{n}$, $n\in {\mathbb{N}}$. Then, for
every $\mu \in \mathcal{P}_{p}(B)$, there exists $A\in \R_+$ such that
\begin{equation}
W_{p}^{p}(\Theta _{s,t_{n}}^{\pi }(\mu ),\nu )\leq A n^{-b_{\ast }\wedge
\varepsilon }(1+\ln (n+1)^{\mathbf{1}_{b_{\ast }=\varepsilon }}  ) \label{A5}
\end{equation}%
where $A$ does not depend on~$\mu$.}
\end{theorem}
\begin{proof}
	We suppose that $b_*\neq \varepsilon$: if not, we multiply the upper bound
	with $1+\ln(1+n)$. \blue{By \eqref{AN16}, we have
	\begin{equation}\label{maj_inter}
	W_{p}^{p}(\Theta _{s,t_{n}}^{\pi }(\mu ),\theta _{s,t_{n}}(\mu ))\leq
	C(1+ M_{p}(\theta,\mu) +\| \Theta^\pi_{s,t_{n(h)}}(\mu)\|_p^p) n^{-b_{\ast }\wedge \varepsilon },
	\end{equation}%
    noting that $M_{p}(\theta,\theta_{s,t}(\mu)))\le M_{p}(\theta,\mu)$ since $\theta$ is a time homogeneous flow.   
	By \eqref{AN14}, we also have  
	\begin{equation*}
	W_{p}^{p}(\theta _{s,t_{n}}(\mu ),\nu )\leq C(1 +M_p(\theta,\mu)) e^{-\varepsilon \wedge
	b_{\ast }(t_{n}-s)}\leq C(1 +M_p(\theta,\mu)) n^{-\varepsilon \wedge b_{\ast }}.
	\end{equation*}%
}	 Using the triangular inequality and $(a+b)^p\le 2^{p-1}(a^p+b^p)$, we get the claim. 
\end{proof}
\blue{\begin{remark}\label{rk_cvunifmu}
If we assume in addition that $\theta$ and $\Theta$ satisfy the p-Foster Lyapunov criterion~\eqref{AN2}, we see from the proof above that the constant $A$ in~\eqref{A5} satisfies $A\le C (1+ \|\mu\|_p^p)$, with $C$ that does not depend on $\mu$, so that 
\begin{equation}
\forall \mu \in \mathcal{P}_p(\R^d), \ W_{p}^{p}(\Theta _{s,t_{n}}^{\pi }(\mu ),\nu )\leq  C (1+ \|\mu\|_p^p)  n^{-b_{\ast }\wedge
\varepsilon }(1+\ln (n+1)^{\mathbf{1}_{b_{\ast }=\varepsilon }}  ) \label{A5_unif}
\end{equation}%
Alternatively, if (only) $\theta$ satisfies the $p$-Foster Lyapunov criterion~\eqref{AN2} and the parameter $a$ in Theorem~\ref{th:app} is such that $a\ge h^{-1}$, then $n(h)=0$ and $\Theta _{s,t_{n(h)}}^{\pi }(\mu )=\mu$ in the inequality~\eqref{maj_inter}, so that~\eqref{A5_unif} still holds.
\end{remark}
}

\begin{remark}
	\label{re:long}
Now that we have stated in Theorem~\ref{FINAL} the main result related to our abstract framework, it is interesting to compare with the recent paper of Schuh and Souttar~\cite{schuh2} that came out while working on this paper. As~\cite{schuh2}, our goal is to give a general framework to study the approximation error of a stationary measure. However, there are significant differences between our works, and we list here the main ones.  
\begin{enumerate}
	\item The framework in~\cite{schuh2} considers  approximations defined for all $t\ge 0$ and parametrized by a parameter~$\delta>0$. When applied to Euler type approximation schemes (see \cite[Subsection 3.2]{schuh2}), this parameter represents the constant time step. In contrast, our framework is meant for approximation schemes and a decreasing sequence of time steps. 
\item   \cite[General assumptions 1 (1)]{schuh2} asks for contractivity, which implies the  exponential convergence to the invariant law of one process. In our setting, the self-coupling property~\eqref{AN1} on $\theta$ implies this property in Proposition~\ref{prop:29}. Note that  \cite[General assumptions 1 (1)]{schuh2} implies condition~\eqref{AN1}, but the reciprocal is not true: $W_p ^p(\mu^1,\mu^2)$ is trivially upper bounded by $\Gamma_p(\mu^1,\mu^2)$, but we cannot upper bound $\Gamma_p(\mu^1,\mu^2)$ by $CW_p ^p(\mu^1,\mu^2)$. Condition~\eqref{AN1} allows the error term in $ (t-s)^{1+\varepsilon} $ which is convenient in order to deal with 
		Euler type schemes, see  in particular the example concerning the Boltzmann type equation in Subsection~\ref{sec:bol} below.

 \item \cite[General assumptions 1 (2)]{schuh2} can seen as a consequence of the coupling assumption~\eqref{AN1} between $\theta$ and $\Theta$, provided that $\theta$ and $\Theta$ are $p$-bounded. Indeed, if we take $\delta<h$ and set $a_k=W_p^p(\theta^{\pi}_{s,s+k\delta}(\mu),\Theta^{\pi}_{s,s+k\delta}(\mu))$ with $\pi= \{s+n\delta, n \in \N \}$ and $\mu \in \mathcal{P}_p(B)$, we get $a_0=0$ and $a_{k+1}\le e^{-b_* \delta}a_k + C \delta^{1+\varepsilon}$. This gives $a_k\le  C \frac{\delta^{1+\varepsilon}}{1-e^{-b_* \delta}}$ for all $k\in \N$, and thus  $W_p(\theta^{\pi}_{s,s+k\delta}(\mu),\Theta^{\pi}_{s,s+k\delta}(\mu))\le  (C/b_*)^{1/p} \delta^{\varepsilon/p} $
  when $\delta$ is smaller than $1/b_*$. Note also that we directly get an estimate for all $k\ge 0$ as in~\cite[Theorem 1.1]{schuh2}. \\ 
 Let us stress also that \cite[General assumptions 1 (2)]{schuh2} involves the multi step approximation scheme. Instead, our error hypothesis~\eqref{AN1} only concerns the one step approximation scheme, which is crucial then in order to deal with decreasing time steps. In practice, it is often easier to check one step estimates : though \cite[General assumptions 1 (2)]{schuh2} is weaker, one would typically use one step estimates and adapt the arguments above to prove it.
\end{enumerate}  
Thus, modulo $p$-boundedness assumptions, \cite[General assumptions 1 (1)]{schuh2} is a bit stronger than the self-coupling property on~$\theta$ while \cite[General assumptions 1 (2)]{schuh2} is a bit weaker than our coupling assumption between $\theta$ and $\Theta$. Our assumptions are thus different from those of~\cite{schuh2}. They allow us to consider a decreasing sequence of time steps and then to get a rate of convergence, which does not seem possible with~\cite[General assumptions 1]{schuh2}. Note that considering a decreasing sequence of time steps is somehow related to the question raised in~\cite[Remark 2.5]{schuh2} on the distribution limit when both $t\to \infty$ and $\delta \to 0$ and the commutativity of the limits.
	\end{remark}

\subsection{Probabilistic representation}

\label{subsec_Probrep}

Now, we give a probabilistic framework which allows to prove that $%
\Theta^{1} $ and $\Theta ^{2}$ are $(p,b_{\ast },\varepsilon )-$ coupled
(and which will be used in the sequel in our examples).

\begin{definition}
	\label{def:214}
We say that we have a ``\textbf{joint probabilistic
representation}" for $(\Theta ^{1},\Theta ^{2})$ if the following holds.
There exists some $h\in (0,1)$ such that for every $s<t$ with $t-s\leq h$ and $\mu
^{1},\mu ^{2}\in \mathcal{P}_{p}(B)$ one may find a probability space $%
(\Omega ,\mathcal{F},{\mathbb{P}})$ (which may depend on $s<t$ and $\mu
^{1},\mu ^{2})$ and some $B$-valued $p$-integrable random variables $%
X^{1},X^{2}$ and $\mathcal{X}_{s,t}^{1},\mathcal{X}_{s,t}^{2}$ such that 
\begin{equation*}
X^{i}\sim \mu ^{i}\quad \text{and}\quad \mathcal{X}_{s,t}^{i}\sim \Theta
_{s,t}^{i}(\mu ^{i})\quad i=1,2
\end{equation*}%
and the following two conditions are satisfied for some positive constants $%
\varepsilon $, $b_{\ast }$, $C_{\ast }$, 
and any $0<s\leq t$ with $t-s\leq h$
\begin{align}
 W_{p}^{p}(\mu ^{1},\mu ^{2})=&\mathbb{E}[d^{p}(X^{1},X^{2})],  \label{AN5}
\\
 \mathbb{E}[d^{p}(\mathcal{X}_{s,t}^{1},\mathcal{X}_{s,t}^{2})]\leq&
(1-b_{\ast }(t-s))\mathbb{E}[d^{p}(X^{1},X^{2})]+C_{\ast }\Gamma _{ p}(\mu
^{1},\mu ^{2})(t-s)^{1+\varepsilon }. \label{AN3}
\end{align}%
When $\Theta ^{1}=\Theta ^{2}$, we say that we have a self joint
probabilistic representation for $\Theta ^{1}$, which echoes to Definition~%
\ref{coupling} where coupling and self coupling are introduced.
\end{definition}
\noindent  We refer to Villani \cite[Theorem 4.1]{Villani} for the existence of an optimal coupling, so that condition~\eqref{AN5} can always be achieved on an atomless probability space.
\begin{proposition}
\label{stochastic}Suppose that we have a joint probabilistic representation
for $\Theta ^{1}$ and $\Theta ^{2}$. Then, $\Theta ^{1}$ and $\Theta ^{2}$
are $(p,b_{\ast },\varepsilon )-$coupled . 
\end{proposition}

\begin{proof}
We need to check the conditions in Definition \ref{coupling}. Using (\ref%
{AN3}) first and (\ref{AN5}) we have 
\begin{align*}
W_{p}^{p}(\Theta _{s,t}^{1}(\mu ^{1}),\Theta _{s,t}^{2}(\mu ^{2})) &\leq 
\mathbb{E}[d^{p}(\mathcal{X}_{s,t}^{1},\mathcal{X}_{s,t}^{2})] \\
&\leq (1-b_{\ast }(t-s))\mathbb{E}[d^{p}(X^{1},X^{2})]+C_{\ast }\Gamma
_{p}(\mu ^{1},\mu ^{2})(t-s)^{1+\varepsilon } \\
&=(1-b_{\ast }(t-s))W_{p}^{p}(\mu ^{1},\mu ^{2})+C_{\ast }\Gamma
_{p}(\mu ^{1},\mu ^{2})(t-s)^{1+\varepsilon }
\end{align*}
so (\ref{AN1}) is proved. 
\end{proof}

Now, we assume that $\Theta $ has a probabilistic representation as in Definition \ref{def:214}: \ for every $%
s<t $ and every $\mu\in\mathcal{P}_p(B) ,$ one may find a probability space $(\Omega ,\mathcal{F}%
,{\mathbb{P}})$ (which may depend on $s<t$ and $\mu )$ and some~$B$ valued
square integrable random variables $X$ and $\mathcal{X}_{s,t},$ such that $%
X\sim \mu $ and $\mathcal{X}_{s,t}\sim \Theta _{s,t}(\mu ).$ In particular 
\begin{equation*}
	\left\Vert \mathcal{X}_{s,t}\right\Vert _{p}=\left\Vert \Theta %
	_{s,t}(\mu )\right\Vert _{p}\quad \left\Vert X\right\Vert
	_{p}=\left\Vert \mu \right\Vert _{p}
\end{equation*}

\begin{lemma}
	\label{lem:2.19}
	Let $h>0$ and suppose that $\Theta $ has a probabilistic representation satisfying%
	\begin{equation*}
		\left\Vert \mathcal{X}_{s,t}\right\Vert _{p}^{p}\leq (1-b(t-s))\left\Vert
		X\right\Vert _{p}^{p}+C(t-s)
	\end{equation*}%
	for some $b$ and $C$ and any $ t-s\in [0,h] $. Then $\Theta $ verifies the $p$-Foster-Lyapunov
	condition \eqref{AN2} and consequently is $p-$bounded, and by \eqref{AN7bis}%
	\begin{equation}
		M_{p}(\Theta ,\mu ,h)\leq \left\Vert \mu \right\Vert _{p}^{p}+C\frac{e^{b}}{b}.
	\end{equation}
\end{lemma}

\section{Some useful tools}\label{sec_tools}

\subsection{Using the sewing lemma to prove coupling}
\label{sec:sew}

In this section, we explain how the sewing lemma can be used to get
automatically some assumptions that are required in Theorem~\ref{FINAL} when working with the $W_1$~distance. We
start by the following motivating lemma that requires to work with $%
p=1$, as explained after the proof. 

\begin{lemma}
\label{lem_coupling} Let $\Theta^i_{s,t}:{\mathcal{P}}_{1}(B)\rightarrow {%
\mathcal{P }}_{1}(B)$, $s\leq t$, $i=1,2$ be two families of applications
such that:

\begin{itemize}
\item $\Theta ^{1}$ is $1$-bounded and $(1,b_{\ast },\varepsilon )$
self-coupled,

\item There exists $C\in {\mathbb{R}}_+$ and $ h>0 $ such that for all $0 \leq t-s \leq h$ and $ \mu \in 
\mathcal{P}_1(B)$, $W_{1}(\Theta^1_{s,t}(\mu ), \Theta^2_{s,t}(\mu ))\le C
(1+\|\mu\|_{1} ) (t-s)^{1+\varepsilon} $.
\end{itemize}

Then, $\Theta ^{1}$ and $\Theta ^{2}$ are $(1,b_{\ast },\varepsilon )$%
-coupled. Furthermore, $\Theta ^{2}$ is $1$-bounded and $(1,b_{\ast
},\varepsilon )$ self-coupled as well.
\end{lemma}

\begin{proof}
Let $\mu^{1},\mu^{2}\in \mathcal{P}_{1}(B)$. We use the triangle inequality,  the fact that $\Theta ^{1}$ is $(1,b_{\ast },\varepsilon )$ self-coupled and the assumption to get
\begin{align*}
W_{1}(\Theta _{s,t}^{1}(\mu^{1}),\Theta _{s,t}^{2}(\mu^{2}))&\leq
W_{1}(\Theta _{s,t}^{1}(\mu^{1}),\Theta _{s,t}^{1}(\mu^{2}))+W_{1}(\Theta _{s,t}^{1}(\mu^{2}),\Theta _{s,t}^{2}(\mu^{2})) \\
&\leq (1-b_*(t-s)) W_1(\mu^1,\mu^2) +C_* \Gamma_1(\mu^1,\mu^2)(t-s)^{1+\varepsilon} \\
&\quad +C(1+\|\mu^2\|_1)(t-s)^{1+\varepsilon}.
\end{align*}%
We thus get that $\Theta ^{1}$
and $\Theta ^{2}$ are $(1,b_{\ast },\varepsilon )$-coupled. \blue{We then use  Lemma~\ref{lem_wass} and pick $h'\in(0,h)$ that fulfills~\eqref{AN14bis} to get that} $\Theta ^{2}$ is also $1-$bounded. 
\end{proof}
We see here why we need $p=1$. For $p>1$, we would apply the triangular inequality and then $(a+b)^p\le 2^{p-1}(a^p+b^b)$:  the multiplicative factor $2^{p-1}>1$  prevents to get the coupling property~\eqref{AN1}.

\bigskip

We now present the sewing lemma. To do so, we define 
\begin{equation*}
\mathcal{E}_{1}(\mathcal{P}_{1}(B))=\left\{ \Theta :\mathcal{P}%
_{1}(B)\rightarrow \mathcal{P}_{1}(B):\sup_{\mu \in \mathcal{P}_{1}(B)}\frac{%
\Vert \Theta (\mu )\Vert _{1}}{1+\Vert \mu \Vert _{1}}<\infty \right\} ,
\end{equation*}%
that we endow with the following distance 
\begin{equation*}
\mathcal{D}_{1}(\Theta ,\overline{\Theta })=\sup_{\mu \in \mathcal{P}%
_{1}(B)}\frac{W_{1}(\Theta (\mu ),\overline{\Theta }(\mu ))}{1+\Vert \mu
\Vert _{1}}.
\end{equation*}%
From \cite[Theorem 6.18]{Villani},  $(%
\mathcal{P}_{1}(B),W_{1})$ is a complete metric space. By \cite[Lemma B.1]%
{AB}, $(\mathcal{E}_{1}(\mathcal{P}_{1}(B)),\mathcal{D}_{1})$ is a
complete metric space. For $U,V\in \mathcal{E}_{1}(\mathcal{P}_{1}(B))$,
we note $UV=U\circ V$ and we check easily that $UV\in \mathcal{E}_{1}$.
From~\cite[Lemma 2.1]{AB}, we have the following sewing lemma.

\begin{lemma}
(Sewing lemma) Let $T>0$ and suppose that $(\Theta_{s,t})_{0\le s\le t \le
T} $ is a family such that $\Theta_{s,t}\in \mathcal{E}_{1}$ and:

\begin{enumerate}
\item $\sup_{0\le s\le t \le T}\mathcal{D}_{1}(\Theta_{s,t},Id)<\infty$,

\item there exists $C_{\textup{lip}}$ such that for any $0\leq s\leq t\leq
T $, $n\in \mathbb{N}^{\ast }$, any time grid $\pi =\left\{
t_{0}=s<t_{1}<\dots <t_{n}=t\right\} $ and every $U,\tilde{U}\in \mathcal{E}%
_{1}(\mathcal{P}_{1}(B))$, 
\begin{equation*}
\mathcal{D}_{1}(\Theta _{s,t}^{\pi }U,\Theta _{s,t}^{\pi }\tilde{U})\leq
C_{\textup{lip}}\mathcal{D}_{1}(U,\tilde{U}),
\end{equation*}

\item there exists $C_{\textup{sew}}$ and $\varepsilon >0$ such that for
any $0\leq s<u<t<T$ , 
\begin{equation*}
\mathcal{D}_{1}(\Theta _{s,t}U,\Theta _{u,t}\Theta _{s,u}\tilde{U})\leq C_{%
\textup{sew}}(t-s)^{1+\varepsilon }.
\end{equation*}
\end{enumerate}

Then, there exists a unique family $\theta_{s,t}\in \mathcal{E}_{1}(%
\mathcal{P}_1(B))$ which is a  flow such that 
\begin{equation*}
\exists C>0, \ \forall s\le t \le T,\ \mathcal{D}_{1}(\theta_{s,t},%
\Theta_{s,t})\le C(t-s)^{1+\varepsilon}.
\end{equation*}
Besides, it satisfies $\mathcal{D}_{1}(\theta _{s,t}U,\theta _{s,t} \tilde{U})\leq
C_{\textup{lip}}\mathcal{D}_{1}(U,\tilde{U})$, which gives in particular the Lipschitz property $W_1(\theta_{s,t}(\mu),\theta_{s,t}(\tilde{\mu}))\le C_{\textup{lip}} W_1(\mu,\tilde{\mu})$ for $\mu, \tilde{\mu}\in \mathcal{P}_1(B)$ and thus the continuity property~\eqref{eq:++}. Last, $\theta$ is time homogeneous if $\Theta_{s,t}=\Theta_{0,t-s}$.
\end{lemma}

\noindent Therefore, the sewing lemma gives the estimate 
\begin{equation}\label{eq_sew}
W_{1}(\theta_{s,t}(\mu),\Theta_{s,t}(\mu))\le C(t-s)^{1+\varepsilon}
(1+\|\mu\|_{1})
\end{equation}
that is needed to apply Lemma~\ref{lem_coupling}. More precisely, Lemma~\ref%
{lem_coupling} can be then applied in two ways:

\begin{itemize}
\item If $\Theta $ is $1$-bounded and $(1,b_{\ast },\varepsilon )$
self-coupled, then $\Theta $ and $\theta $ are $(1,b_{\ast },\varepsilon )$%
-coupled and $\theta $ is also $1$-bounded.

\item If we are able to prove independently (relying on continuous time
arguments) that $\theta $ is $1$-bounded and is $(1,b_{\ast
},\varepsilon )$ self-coupled, then $\Theta $ and $\theta $ are $%
(1,b_{\ast },\varepsilon )$-coupled and $\Theta $ is $1$-bounded. In
particular, if $\theta$ is time homogeneous and $\nu $ is the invariant measure for $\theta$, and if we take the time sequence $(t_{n})$ as in Theorem~\ref{FINAL}
then, by (\ref%
{A5}), for every $\mu \in \mathcal{P}_{1}(B)$,
\begin{equation*}
W_{1}(\Theta _{s,t_{n}}^{\pi }(\mu ),\nu )\leq A 
n^{-b_{\ast }\wedge \varepsilon }(1+\ln (n+1)^{\mathbf{1}_{b_{\ast
}=\varepsilon }}).
\end{equation*}%

\end{itemize}

\subsection{ A Foster-Lyapunov type criterion to prove $p$-boundedness}
\label{sec:51}
\noindent

We consider now a family of operators $\Theta _{s,t}:\mathcal{P}_p(B)\to\mathcal{P}_p(B)$ and we discuss the boundedness property. For a $B$-valued random variable $X$, we note $\| X \|_p=\mathbb{E} [d(X,x_0)^p]^{1/p}$.

\begin{definition}
	\label{def_FL} We say that a family of deterministic operators $\Theta
	_{s,t}:\mathcal{P}_{p}(B)\rightarrow \mathcal{P}_{p}(B)$, satisfies the $%
	p$-Foster-Lyapunov condition if there exist some constants $\overline{b},%
	\overline{C},h\in (0,1)$, such that for every $s\leq t\leq s+h$ and every $\mu \in 
	\mathcal{P}_{p}(B)$, 
	\begin{equation}
		\left\Vert \Theta _{s,t}(\mu )\right\Vert _{p}^{p}\leq (1-\overline{b}%
		(t-s))\left\Vert \mu \right\Vert _{p}^{p}+\overline{C}(t-s).  \label{AN2}
	\end{equation}
\end{definition}
Note that if we prove \eqref{AN2} for some distance $ d(x,y) $, this does not imply that the same property is satisfied for an equivalent distance due to the factor $ (1-\overline{b}%
(t-s)) $ in the property \eqref{AN2}.
Thus, in some cases, one may be able to
prove~\eqref{AN2} for some special distance~$d$, but not for~$|\cdot |$, or
conversely. The constants throughout the article depend on $p$ and $d$
without making this dependence explicit as we will always consider $p$ and $%
d $ to be fixed.

If~\eqref{AN2} holds, then we clearly have 
\begin{equation}
	\left\Vert \Theta _{s,t}(\mu )\right\Vert _{p}^{p}\leq e^{-{\bar{b}}%
		(t-s)}\left\Vert \mu \right\Vert _{p}^{p}+\overline{C}(t-s).
	\label{AN2bis}
\end{equation}%
Conversely, if~\eqref{AN2bis} holds, we get by a Taylor expansion of the
exponential function that there exists $0<h^{\prime }<h$ such that for $%
0\leq t-s\leq h^{\prime }$, \eqref{AN2} holds with $\bar{b}/2$ instead of $%
\bar{b}$. Thus, the conditions~\eqref{AN2} and~\eqref{AN2bis} are
essentially equivalent, but we prefer to use~\eqref{AN2} because we usually
directly get this type of estimates on practical examples. 

\begin{remark}
	The hypothesis (\ref{AN2}) is a variant of the Foster-Lyapunov drift
	condition (see e.g.~\cite{MeTw}, \cite[Eq.~(13)]{JoHo} and~\cite[Eq.~(6)]%
	{DuMo}) which is used in the Monte Carlo Markov Chain (MCMC) literature. Let
	us consider the case of a time homogeneous family, i.e. $\Theta_{s,t}(\mu)=
	\mu P_{t-s}$, where $P_{\gamma }(x,dy)$, $\gamma>0$ is a probability kernel
	and $\mu P_{\gamma }(dy)=\int_B P_{\gamma }(x,dy)\mu (dx) $ as usual.
	Following~\cite{DuMo}, a function $V$ satisfy a Foster-Lyapunov drift
	condition for $P_{\gamma}$ if there exists $\bar{\gamma}>0$, $\lambda \in
	(0,1)$ and $c\geq 0$ such that $P_{\gamma}V(x)\leq \lambda^{\gamma
	}V(x)+c\gamma$ for all $\gamma\in (0,\bar{\gamma}]$. If $V(x)=|x|^p$
	satisfies this Foster-Lyapunov drift condition, then integrating this
	inequality with respect to $\mu $ gives precisely~\eqref{AN2bis} with $%
	\bar{b}=-\log(\lambda)$, $h=\bar{\gamma}$ and $d(x,y)=|x-y|$. Conversely,
	if~\eqref{AN2bis} holds with $d(x,y)=|x-y|$ then we get by taking $%
	\mu=\delta_x$ that $P_{\gamma}V(x)\leq \lambda^{\gamma }V(x)+c\gamma$ for $%
	V(x)=|x|^p$.
\end{remark}

\begin{remark}
	\label{rk_AN2BIS} Suppose that one may find constants $\widehat{b},%
	\widetilde{C},\widehat{C},{\varepsilon },h>0$ such that for every $s\leq
	t\leq t+h$ and $\mu \in \mathcal{P}_{p}(B)$, 
	\begin{equation}
		\left\Vert \Theta _{s,t}(\mu )\right\Vert _{p}^{p}\leq (1-\widehat{b}%
		(t-s))\left\Vert \mu \right\Vert _{p}^{p}+\widetilde{C}(t-s)+\widehat{C}%
		(1+\left\Vert \mu \right\Vert _{p}^{p})(t-s)^{1+{\varepsilon }}.
		\label{AN2BIS}
	\end{equation}%
	Then, for every $0\leq t-s\leq \min \left( h,\left( \frac{\widehat{b}}{2%
		\widehat{C}}\right) ^{1/{\varepsilon }}\right) $, (\ref{AN2}) holds with $%
	\overline{b}=\widehat{b}/2$, and $\overline{C}=\widetilde{C}+\widehat{C}h^{{%
			\varepsilon }}.$
\end{remark}

We consider a family of deterministic operators $\Theta _{s,t}:\mathcal{P}
_{p}(B)\rightarrow \mathcal{P}_{p}(B),s\leq t$ and recall the notations and definitions given in Section \ref{Sec_Framework}.
We now prove an estimate which is an iterated version of (\ref{AN2}).

\begin{lemma}
	\label{lem_bound} Suppose that the family $\Theta _{s,t}$, $s\leq t$
	satisfies the $p$-Foster-Lyapunov condition, and let $\overline{b},%
	\overline{C},h>0$ be such that~\eqref{AN2} holds. Let $(t_{i})_{i\in {%
			\mathbb{N}}}$ be an increasing sequence such that $t_{0}\geq 0$ and $t_{i}-t_{i-1}\le h$ for $i\in {\mathbb{N}}^{\ast }$. Then, for every $n\in {\mathbb{N}}$ 
	\begin{equation}
		\left\Vert \Theta _{t_{0},t_{n}}^{\pi }(\mu )\right\Vert _{p}^{p}\leq e^{-%
			\overline{b}(t_{n}-t_{0})}\left\Vert \mu \right\Vert _{p}^{p}+\overline{C}%
		\frac{e^{\overline{b}}}{\overline{b}}.  \label{AN7}
	\end{equation}%
	In particular $\Theta $ is $p-$bounded and%
	\begin{equation}
		M_{p}(\Theta ,\mu \blue{,h })\leq \left\Vert \mu \right\Vert _{p}^{p}+\overline{C}%
		\frac{e^{\overline{b}}}{\overline{b}}.  \label{AN7bis}
	\end{equation}%
	In addition, assume that $\Theta $ is a flow (i.e. $\Theta _{s,t}=\Theta
	_{r,t}\circ \Theta _{s,r}$ for every $s<r<t)$. Then, for every $s<t$ and
	every $\mu \in \mathcal{P}_{p}(B)$, one has 
	\begin{equation}
		\left\Vert \Theta _{s,t}(\mu )\right\Vert _{p}^{p}\leq e^{-\overline{b}%
			(t-s)}\left\Vert \mu \right\Vert _{p}^{p}+\overline{C}\frac{e^{\overline{b}%
		}}{\overline{b}}.  \label{AN7'}
	\end{equation}
\end{lemma}

\begin{proof}
	We denote $a_{k}=\left\Vert \Theta _{t_{0},t_{k}}^{\pi }(\mu )\right\Vert
	_{p}^{p}$ and write 
	\begin{equation*}
		e^{\overline{b}t_{n}}a_{n}=e^{\overline{b}t_{0}}a_{0}+\sum_{k=1}^{n}(e^{%
			\overline{b}t_{k}}a_{k}-e^{\overline{b}t_{k-1}}a_{k-1}).
	\end{equation*}%
	Using the property (\ref{AN2}) (with $t_{k}-t_{k-1}=\gamma _{k}\leq \gamma
	_{1}\leq h)$ 
	\begin{align*}
		e^{\overline{b}t_{k}}a_{k}-e^{\overline{b}t_{k-1}}a_{k-1}=&e^{\overline{b}%
		t_{k}}(1-e^{-\overline{b}\gamma _{k}})a_{k-1}+e^{\overline{b}%
		t_{k}}(a_{k}-a_{k-1}) \\
	\leq& e^{\overline{b}t_{k}}(\overline{b}\gamma _{k}a_{k-1}-\overline{b}%
	\gamma _{k}a_{k-1}+\overline{C}\gamma _{k})=\overline{C}e^{\overline{b}%
		t_{k}}\gamma _{k}.
	\end{align*}
	Therefore, we get 
	\begin{equation*}
		e^{\overline{b}t_{n}}a_{n}\leq e^{\overline{b}t_{0}}a_{0}+\overline{C}%
		\sum_{k=1}^{n}e^{\overline{b}t_{k}}\gamma _{k},
	\end{equation*}%
	which gives 
	\begin{align*}
		a_{n}& \leq e^{-\overline{b}(t_{n}-t_{0})}a_{0}+\overline{C}%
		\sum_{k=1}^{n}e^{-\overline{b}(t_{n}-t_{k})}\gamma _{k} \\
		& \leq e^{-\overline{b}(t_{n}-t_{0})}a_{0}+\overline{C}e^{\overline{b}%
		}\sum_{k=1}^{n}e^{-\overline{b}(t_{n}-t_{k-1})}\gamma _{k}\leq e^{-\overline{%
				b}(t_{n}-t_{0})}a_{0}+\overline{C}\frac{e^{\overline{b}}}{\overline{b}}.
	\end{align*}%
	The last inequality uses that $e^{\overline{b}t_{k-1}}\gamma _{k}\leq
	\int_{t_{k-1}}^{t_{k}}e^{\overline{b}s}ds$ and exact integration. So (\ref%
	{AN7})  and \eqref{AN7bis} are proved.
	
	We now assume that $\Theta $ is a flow. We take $n$ such that $(t-s)/n<h$
	and consider the uniform grid $t_{k}=s+k\frac{t-s}{n}$ so that $t_{0}=s$ and 
	$t_{n}=t$. We use the previous result for the time grid $\pi =\{t_{k},k\in {%
		\mathbb{N}}\}$ and the flow property $\Theta _{t_{0},t_{n}}^{\pi }=\Theta
	_{s,t}$ so that~(\ref{AN7'}) coincides with~(\ref{AN7}).
\end{proof}

In many situations to follow in the examples, one applies 
a sequence of arguments which are a recall of the results of Proposition~\ref{stochastic}, Theorem~\ref{FINAL} and Proposition~\ref{prop:29}. They are summarized in the next theorem.
\begin{theorem}
	\label{th:app}
	Let $ \Theta, \theta :\mathcal{P}_p(B)\to \mathcal{P}_p(B) $ and that  $\theta  $ is a time homogeneous continuous flow. Assume that we have two joint probabilistic representations for $(\Theta,\theta)$ and $(\theta,\theta)$ in the sense of Definition~\ref{def:214}. Furthermore, we assume that they satisfy  \eqref{AN2BIS}. Then the invariant measure associated to $ \theta $, denoted by $ \nu $ exists and is unique.
	It satisfies $\Vert \nu \Vert
	_{p}^{p}=M_{p}(\theta ,\nu )$ and we have for every $s\leq t$ and every $%
	\mu \in \mathcal{P}_{p}(B)$, 
	\begin{equation}
		W_{p}^{p}(\theta _{s,t}(\mu ),\nu )\leq C(1+\|\mu\|_p^p) e^{-\varepsilon \wedge
			b_{\ast }(t-s)}(1+(t-s)^{\mathbf{1}_{b_{\ast }=\varepsilon }})
		\label{AN141}
	\end{equation}%
	where the constant $C\in \R_+$ does not depend on~$\mu$.
	
	 Furthermore, 
	let $a\ge 0$, $\gamma _{n}=\frac{1}{a+n}$ and $%
	t_{n}=s+\gamma _{1}+...+\gamma _{n}$, $n\in {\mathbb{N}}$. Then, there exists $C\in \R_+$ such that  for
	every $\mu \in \mathcal{P}_{p}(B)$,
	\begin{equation}
		W_{p}^{p}(\Theta _{s,t_{n}}^{\pi }(\mu ),\nu )\leq C(1+\|\mu\|_p^p) n^{-b_{\ast }\wedge
			\varepsilon }(1+\ln (n+1)^{\mathbf{1}_{b_{\ast }=\varepsilon }}  ). \label{A51}
	\end{equation}%
	\end{theorem}
\blue{In contrast with Theorem~\ref{th:main}, we assume here that the Foster Lyapunov criterion~\eqref{AN2BIS} hold for $\theta$ and $\Theta$, which gives a uniform  estimate  with respect to $\mu$, see Remark~\ref{rk_cvunifmu}. }

\section{Examples with non global contraction conditions (Langevin type equations)}\label{sec_Eberle}

 The
present section is devoted to prove that, once the convergence of the semigroup to the invariant measure is obtained, the estimate of the error done by using Euler scheme, is possible (see Lemma~\ref{lem_coupling}). We illustrate this by two non
trivial examples.

In several recent papers (see in particular Eberle~\cite{Eberle}, Eberle et al.~\cite{EGZ2} and Schuh~\cite{Schuh} ) authors have used a
special coupling procedure, called ``reflecting coupling" in order to prove
convergence to equilibrium for the Langevin type equations. An important point is that
the result holds under a weak contraction hypothesis (i.e. out of a compact set), 
in contrast with older results in which uniform construction properties
are needed. It turns out that we are able to use Eberle's result concerning
the flow in order to prove convergence of the Euler scheme as well and to
control the speed of convergence. First, we consider the classical Langevin
equation given by%
\begin{equation}
	\label{eq:lan1}
dX_{t}=b(X_{t})dt+  d B_{t}
\end{equation}%
where $B$ is an $m$ dimensional Brownian motion and $b:\R^{d}\rightarrow \R^{d}$ is a Lipschitz
continuous function. We assume that%
\begin{equation}\label{lip_b}
\left\vert b(x)-b(y)\right\vert \leq L_{b}\left\vert
x-y\right\vert
\end{equation}%
and that there exists $R,\kappa>0$ such that 
\begin{equation}\label{weak_c}
\left\langle x-y,b(x)-b(y)\right\rangle \leq -\kappa \left\vert
x-y\right\vert ^{2}\quad \text{if}\quad \left\vert x-y\right\vert \geq R.
\end{equation}%
Then, we are precisely in the framework of~\cite[Example 1]{Eberle}. In particular, in~\cite[Equation (10)]{Eberle}, the author finds a function $%
f:\R_{+}\rightarrow \R_{+}$ which is
concave, increasing, such that $f(0)=0$, $f^{\prime }(0)=1$ and $0<c_{\ast }\leq
f^{\prime }(r)\leq 1$ for every $r$. In what follows, the precise expression of $f$ is not important for us. To the
function $f$ one associates the distance%
\begin{equation*}
d_{f}(x,y)=f(\left\vert x-y\right\vert )
\end{equation*}%
which is equivalent with the Euclidean distance: $c_*|x-y|\le d_{f}(x,y)\le |x-y|$. We note for $\mu,\nu \in\mathcal{P}_1(\R^d)$
$$W_1(\mu,\nu)=\inf_{\rho \in \Pi (\mu ,\nu
	)}\int_{\R^d}\int_{\R^d}d_f(x,y) \rho (dx,dy),$$
the corresponding $1$-Wasserstein distance, and  $W^{|\cdot|}_1(\mu,\nu)=\inf_{\rho \in \Pi (\mu ,\nu
	)}\int_{\R^d}\int_{\R^d} |x-y| \rho (dx,dy)$ the classical 1-Wasserstein distance. We then have
	\begin{equation}\label{equiv_norm_W1}
		c_*W^{|\cdot|}_1(\mu,\nu)\le W_1(\mu,\nu)\le W_1^{|\cdot|}(\mu,\nu).
	\end{equation}
	It is equivalent to the standard $1$-Wasserstein distance 
Then, \cite[Corollary 2 (13) and
Corollary 3]{Eberle} gives
\begin{theorem}
	\label{th:41}
Suppose that \eqref{lip_b} and \eqref{weak_c} hold and let $\theta _{s,t}(\mu )$ be the law of $X_t$ when $X_s\sim \mu$ is independent of~$(B_u-B_s,u\ge s)$. 
Then, there is a constant $b_*>0$ such that  for any $\mu^1,\mu^2 \in \mathcal{P}_1(\R^d)$,
\begin{equation}
W_{1}(\theta _{s,t}(\mu^1 ),\theta _{s,t}(\mu^2 ))\leq W_{1}(\mu^1
,\mu^2 )e^{-b_{\ast }(t-s)}.  \label{E2}
\end{equation}%
 Moreover, $\theta $ has a
unique invariant measure $\nu\in \mathcal{P}_1(\R^d)$.
\end{theorem}

\begin{remark}\label{rk_boundedtheta}
Using that $ \theta_{s,t}(\nu)=\nu $ we obtain that $\theta $ is $1-$bounded. Indeed, following~\eqref{def_pseudonorm}, we set $\left\Vert \mu \right\Vert _{1}=\int_{\R^d} d_f(x,0) \mu(dx)$ and for any $ \mu\in\mathcal{P}_1(\mathbb{R}^d) $, we have:
\begin{align*}
	\left\Vert \theta _{s,t}(\mu )\right\Vert _{1}&\leq \left\Vert
\theta _{s,t}(\nu )\right\Vert _{1}+ W_{1}(\theta _{s,t}(\mu
),\theta _{s,t}(\nu )) \\
&\leq \left\Vert \nu \right\Vert _{1} +W_{1}(\mu ,\nu
),
\end{align*}
by using~\eqref{E2}.
\end{remark}

\begin{remark}
From~\eqref{equiv_norm_W1}, we get easily 
\begin{equation*}
W_{1}^{|\cdot|}(\theta _{s,t}(\mu ),\theta _{s,t}(\nu ))\leq \frac{%
1}{c_*}e^{-b_*(t-s)}W_1(\mu ,\nu )
\end{equation*}
but then the multiplicative constant becomes $1/c_*$ which is in general strictly greater that one (otherwise the distances $W_1$ and $W_1^{|\cdot|}$ would be the same). This breaks the
contraction property that is crucial in the convergence analysis toward the stationary measure. In several recent papers and prominently Eberle~\cite{Eberle}, the authors succeed in constructing a distance $d$ that allows to
obtain strict contraction with respect to the Wasserstein distance associated to this distance. This is why we include in our framework of Section~\ref{Sec_Framework} a
general distance~$d$.
\end{remark}

We now construct the following Euler scheme:
\begin{align*}
Y_{s,t} &=Y+b(Y)(t-s)+B_t-B_s ,\\
\Theta _{s,t}(\mathcal{L}(Y)) &=\mathcal{L}(Y_{s,t}).
\end{align*} 

\begin{lemma}
Suppose that $b$ satisfies~\eqref{lip_b}. Then, there exists $C\in \R_+$ such that for all $s\le t\le s+1$ and $\mu \in \mathcal{P}_1(\R^d)$, 
\begin{equation}
W_{1}(\theta _{s,t}(\mu ),\Theta _{s,t}(\mu ))\leq C(1+\|\mu\|_1) (t-s)^{3/2}.
\label{3}
\end{equation}%
\end{lemma}
Thus, from Theorem \ref{th:41} we have that Lemma~\ref{lem_coupling} holds with $\varepsilon =\frac{1}{2}.$

\begin{proof}
	Let $X_{s,t}\equiv X_{s,t}(Y)$ be the solution of%
	\begin{equation*}
	X_{s,t}=Y+\int_{s}^{t} b(X_{s,r})dr+B_{t}-B_{s}, \  t\ge s.
	\end{equation*}%
	Then, for $|t-s|\le 1$, we have $\mathbb{E}[\left\vert X_{s,t}-Y\right\vert] \leq 
	C(1+\mathbb{E}\left\vert Y\right\vert
	)(t-s)^{1/2}$, for a constant $C$ depending on $b(0)$ and  $L_b$. We now combine  the Lipschitz property and the previous inequality to get
	\begin{equation*}
\mathbb{E}\left[	\left\vert X_{s,t}-Y_{s,t}\right\vert \right]\leq \int_{s}^{t} C(1+\mathbb{E}[\left\vert Y\right\vert]
)(r-s)^{1/2} dr \le C(1+\mathbb{E}[\left\vert Y\right\vert]
)(t-s)^{3/2}.
	\end{equation*}%
\end{proof} 
With the above Lemma we can give the error estimate for the approximation of the invariant measure to the stochastic equation \eqref{eq:lan1}.

\begin{theorem}
Suppose that~\eqref{lip_b} and~\eqref{weak_c} hold. Let $s\ge 0$ and consider the particular
time grid $t_{k}=s+\sum_{i=1}^{k}\frac{1}{i+1}$ for $k\in \N$. Then, there exists a constant $C\in \R_+$ such that we have for all $n\ge 1$,
\begin{equation}
W_{1}^{|\cdot|}(\Theta _{s,t_{n}}^{\pi }(\nu ),\mu _{\ast })\leq Cn^{-b_{\ast }\wedge 1/2}(1+\ln(1+n)^{\mathbf{1}_{b_*=1/2}})
\label{A6}
\end{equation}%
where $b_{\ast }$ is the constant in (\ref{E2}).
\end{theorem}
\begin{proof}
	Using Lemma~\ref{lem_coupling} we get that $\Theta $ is $%
(1,b_{\ast },\frac{1}{2})-$ coupled with $\theta .$ We already know that $%
\theta $ is $1-$bounded by Remark~\ref{rk_boundedtheta}. It is also $(1,b_{\ast },\frac{1}{2})$-self coupled from~\eqref{E2} since we have $W_{1}(\theta _{s,t}(\mu ),\theta _{s,t}(\nu ))\leq W_{1}(\mu
,\nu )(1-b_{\ast }(t-s)+\frac{b_{\ast }^2}{2}(t-s)^2)$ and $W_1(\mu,\nu)\le \Gamma_1(\mu,\nu)$.
Then, by Theorem~\ref{FINAL} and~\eqref{equiv_norm_W1} we get (\ref{A6}). $\square $
\end{proof}

\begin{remark}
This analysis of the convergence of the Euler scheme has already been done in a more precise setting by Pag\`es and Panloup~\cite{PaPa}. Here,  we just want to
illustrate how this enters in our abstract framework.
\end{remark}

We give now a second example which is much more involved. Following Schuh~\cite{Schuh},
we consider the space and time Langevin equation of McKean-Vlasov type 
\begin{align}
dX_{t} &=Y_{t}dt  \label{langevin} \\
dY_{t} &=(ub^{E}(X_{t})+u\int_{\R^{d}}b^{I}(X_{t},z)\mu _{t}^{X}(dz)-\gamma
Y_{t})dt+\sqrt{2\gamma u} d B_{t}, \notag
\end{align}
where $\mu _{t}^{X}$ is the law of $X_{t},B$ is a $d$ dimensional Brownian
motion, $b^{E}$ describes some external forces and $b^{I}$ represent some
interacting forces. The parameters $u$ and $\gamma$ are positive real numbers associated to the physical model presented in \cite{Schuh}.

We assume that both $b^{E}$ and $b^{I}$ are Lipschitz continuous and
moreover, we assume the following contraction property: There exists a
positive definite matrix $K$ with smaller eigenvalue $\kappa >0$ and a
Lipschitz continuous function $g:\R^{d}\rightarrow \R^{d}$ with Lipschitz
constant $L_{g}$ such that $b^{E}(x)=-Kx+g(x)$. Moreover, the function $g$
verifies for some $R>0,$%
\begin{equation*}
\left\langle x_1-x_2,g(x_1)-g(x_2)\right\rangle \leq 0\quad \text{if}\quad \left\vert
x_1-x_2\right\vert \geq R.
\end{equation*}
The above assumption implies in particular that 
\begin{equation*}
\left\langle x_1-x_2,b^{E}(x_1)-b^{E}(x_2)\right\rangle \leq - \kappa \left\vert
x_1-x_2\right\vert ^{2}\quad \text{if}\quad \left\vert x_1-x_2\right\vert \geq R.
\end{equation*}%

Notice that the contraction property is not a classical one on $(X_{t},Y_{t})$, as the drift coefficient $ b^E $ depends only the first component and it appears only on the second component. 
In this sense, we are dealing with a degenerate problem, and then the use of the
distance $d_{f}$ from the previous example is no longer sufficient in order to
solve that problem. Remarkably, Schuh constructs a new metric $d$ on $%
\R^{d}\times \R^{d}$ which is equivalent with the Euclidean distance (see~\cite[Lemma 4.6]{Schuh}) and for which the following contraction result is obtained, see~\cite[Theorem 12]{Schuh}.
\begin{theorem}\label{thm_schuh}
	Suppose that the above assumptions hold with $2L_{g}^{2}u\gamma ^{-2}<\kappa$. Let $\mu, \mu' \in \mathcal{P}_1(\R^d \times \R^d)$,  $%
	(X_{t},Y_{t})$ and $(X_{t}^{\prime },Y_{t}^{\prime })$ be two solutions
	starting respectively from $\mu $ and $\mu ^{\prime }$. There exists a constant $b_{\ast }>0$ such that 
	\begin{equation}
	W_{1}(\mathcal{L}(X_{t},Y_{t}),\mathcal{L}(X_{t}^{\prime
	},Y_{t}^{\prime }))\leq e^{-b_{\ast }t}W_{1}(\mu ,\mu ^{\prime }).
	\label{langevin2}
	\end{equation}%
	Moreover, there is a unique invariant measure for the corresponding semigroup $\theta $ and
	in particular, $\theta $ is $1-$bounded.
	\end{theorem}
	\noindent The last statement is a straightforward consequence of Proposition~\ref{prop:29} and Remark~\ref{rk_boundedtheta}.
	
The construction
of the metric $d$ is rather heavy, but, as in the previous simpler
example, we are not concerned with the specific form of the metric. Using
Schuh's contraction results for the semigroup is enough to  directly obtain the needed
conditions for the Euler scheme.

\begin{remark}
Since $d$ is equivalent with the Euclidean distance, the above inequality
also implies, for some $M> 1$,%
\begin{equation*}
W_{1}^{|\cdot|}(\mathcal{L}(X_{t},Y_{t}),\mathcal{L}(X_{t}^{\prime },Y_{t}^{\prime
})) \leq Me^{-b_*t}W_{1}^{|\cdot|}(\mu ,\mu ^{\prime }),
\end{equation*}%
where $W_{1}^{|\cdot|}$ is the classical $1$-Wasserstein distance associated to the $1$-norm on $\R^d \times \R^d$. However, this is no longer a strict contraction for $|\cdot|$ since  $M>1$. 
\end{remark}

We define  $\theta _{s,t}(\mu )$ as the law of the solution of~\eqref{langevin} which
starts from $\mu $ at time~$s$. Let $\Theta _{s,t}(\mu $) be the law of $(\overline{X}_{s,t},%
\overline{Y}_{s,t})$, the one
step Euler scheme defined by 
\begin{align*}\overline{X}_{s,t} &=\overline{Y}(t-s) \\
	\overline{Y}_{s,t} &=\left(ub^{E}(\overline{X})+u\int_{\R^{d}}b^{I}(\overline{X}%
	,z)\overline{\mu }^{\overline{X}}(dz)-\gamma \overline{Y} \right)(t-s)+\sqrt{%
		2\gamma u}(B_{t}-B_{s})
\end{align*}
where $(\overline{X},\overline{Y})$ has law $\mu $ and $\mu ^{\overline{X}}(dz)$ is the law of $\overline{X}$.

\begin{theorem}
Let the assumptions of Theorem~\ref{thm_schuh} be in force.  Let $s\ge 0$ and $t_n=s+\sum_{i=1}^{n}\frac{1}{i+1}$ for $n\in \N$.  Then, there exists $C\in \R_+$, such that for all $n\ge 1$
\begin{equation}
W_{1}^{|\cdot|}(\Theta _{s,t_{n}}^{\pi }(\mu ),\nu )\leq Cn^{-\frac{1}{2}\wedge b_{\ast }}(1+\ln
(n+1)^{1_{\{b_{\ast }=\frac{1}{2}\}}}  ).\label{langevin3}
\end{equation}
\end{theorem}
\begin{proof}
Let $(X_{t},Y_{t})$ be the solution starting from $(\overline{%
	X},\overline{Y})$ such that $\mathbb{E}(\vert \overline{X}\vert
	+\vert \overline{Y}\vert )<\infty$. We claim that there exists a constant $C\in \R_+$ such that for $s\le t \le s+1$,
	\begin{equation}
	\mathbb{E}[\vert X_{s,t}-\overline{X}_{s,t}\vert] +\mathbb{E}[\vert Y_{s,t}-%
	\overline{Y}_{s,t}\vert] \leq C \left(1+\mathbb{E}[\vert \overline{X}\vert
	+\vert \overline{Y}\vert] \right)(t-s)^{3/2}.  \label{langevin1}
	\end{equation}
	We first prove that 
	\begin{equation}\label{bound_mom1}\mathbb{E}[\vert X_{s,t}\vert] +\mathbb{E}[\vert Y_{s,t}\vert] \leq C \left(1+\mathbb{E}[\vert \overline{X}\vert
	+\vert \overline{Y}\vert] \right).
	\end{equation}
	To get this, we use the triangle inequality in~\eqref{langevin}, observe that $\sqrt{2\gamma u} \mathbb{E}[|B_t-B_s|]$ is bounded for $t\le s+1$, and conclude with Gronwall lemma. 

	Then, we use the Lipschitz property (and thus the sublinear growth) of $b^{E}$ and $b^{I}$. Using~\eqref{bound_mom1}, this gives for $t\le s+1$, 
	\begin{align*}
		&\mathbb{E}[|X_{s,t}-\overline{X}|]\le C \left(1+\mathbb{E}[\vert \overline{X}\vert
		+\vert \overline{Y}\vert ]\right) (t-s),\\
&\mathbb{E}[|Y_{s,t}-\overline{Y}|]\le C \left(1+\mathbb{E}[\vert \overline{X}\vert
+\vert \overline{Y}\vert] \right) (t-s) + C(t-s)^{1/2}.
	\end{align*}
	With these estimates, we can finally compare the solution $ X $ with the Euler scheme $ \bar{X} $ and get 
$$	\mathbb{E}[\vert X_{s,t}-\overline{X}_{s,t}\vert] \le \int_s^t \mathbb{E}[\vert Y_{s,r}-\overline{Y}\vert ]dr \le C \left(1+\mathbb{E}[\vert \overline{X}\vert
+\vert \overline{Y}\vert ]\right) (t-s)^2 + C(t-s)^{3/2},$$
and, using the Lipschitz property of $b^E$ and $b^I$,
$$	\mathbb{E}[\vert Y_{s,t}-\overline{Y}_{s,t}\vert ]\le C \left(1+\mathbb{E}[\vert \overline{X}\vert
+\vert \overline{Y}\vert] \right) (t-s)^2 + C(t-s)^{3/2}.$$
This gives directly~\eqref{langevin1} since $t-s\le 1$. 	

From the equivalence between the distance $ d $ and the Euclidean norm, we also have \begin{equation*}
	\mathbb{E}[d(X_{s,t},\overline{X}_{s,t})] +\mathbb{E}[d(Y_{s,t},
	\overline{Y}_{s,t})] \leq C \left(1+\mathbb{E}[\vert \overline{X}\vert
	+\vert \overline{Y}\vert] \right)(t-s)^{3/2},
	\end{equation*}
and therefore for any $\mu \in \mathcal{P}_1(\R^d \times \R^d)$,
	\begin{equation*}
	W_{1}(\Theta _{s,t}(\mu ),\theta _{s,t}(\mu ))\leq C(1+\|\mu\|_1) (t-s)^{3/2}.
	\end{equation*}%
	Moreover, using the triangle inequality and (\ref{langevin2}), we get 
	\begin{equation*}
	W_{1}(\Theta _{s,t}(\nu ),\theta _{s,t}(\mu ))\leq
	C(1+\|\nu\|_1)(t-s)^{3/2}+e^{-b_*t}W_{\rho ,1}(\mu ,\nu ).
	\end{equation*}%
	So, $\theta $ and $\Theta $ are $(1,b_*,1/2)$-coupled and  (\ref{langevin3}) is then
	a consequence of Theorem~\ref{FINAL}.
\end{proof}

\section{Examples with global contraction conditions }
\label{sec:ex}

In this section, we propose to investigate different examples of stochastic systems that fall into the framework presented in Section~\ref{Sec_Framework}. 
In contrast with Section~\ref{sec_Eberle} where the coupling property between the approximation scheme and the flow is deduced from the self-coupling property of the flow, we instead present in these section cases where the coupling property is deduced from the self-coupling property of the approximation scheme. 
 When applying it we will need to prove the main  coupling properties as described in Definition~\ref{coupling}, as well as the $p$-boundedness stated in Definition~\ref{def_pbounded}. 
In most examples, the main argument is based on an application of Theorem \ref{th:app}. Therefore the proofs consist in the verification of \eqref{AN2BIS} and \eqref{AN3}. The first example on reflected SDEs considers a rather simple example
with Markovian dynamics without law dependence on the coefficients. Then, we consider a
neuronal model, which is an example of McKean-Vlasov SDE. 
To show that the general framework in Section \ref{Sec_Framework} can be applied in Wasserstein spaces with different value of $ p $, we will obtain conditions which imply the convergence of the approximation schemes to the invariant measure in $W_2$ and $W_1$. This will lead to different contraction
hypothesis. In these two examples, we can apply directly the method
presented in Subsections~\ref{subsec_Probrep} and \ref{sec:51}. Namely, we construct a
coupling between the continuous stochastic process and its approximation
scheme. In contrast, for the Boltzmann equation approximation in $ W_2 $, we are not
able to construct a coupling between the Boltzmann process and its one-step
approximation scheme. To get around this difficulty, we will couple an
approximation of the Boltzmann process (with say parameter $ m $) with the one-step approximation
scheme and then take limits as $ m\to\infty $ in all the estimates. For the Boltzmann equation in $W_1$, we can skip this step relying on the sewing lemma. 
Note that in all the examples, we have not strived to find the best conditions as our goal is to show in a compact form that the main framework applies in each situation. Considering weaker conditions or more general equations  in each example may be possible, but outside the scope of the paper. For readability, we have selected some proofs to keep them in the main text while the other ones are available in Appendix~\ref{App_additional}.

Throughout this section, we assume that the time grid is taken so that
$t_{n}=t_{0}+\sum_{i=1}^{n}\frac{1}{i+1}.$ 
The flows are always time homogeneous and continuous.

\subsection{A stochastic equation with reflection in $ W_2 $}

\label{sec:ref}

We consider the following one dimensional stochastic equation with
reflection in $ \mathbb{R}_+ $ for the initial condition $X^1\geq 0 $, $t-s\leq 1 $ on a filtered probability space $%
(\Omega,\mathcal{F},\mathbb{P}) $ supporting a Wiener process $W $: 
\begin{align}  \label{eq:ref1}
\mathcal{X}^1_{s,t}(X^1)=&X^1+\int_s^tb( \mathcal{X}^1_{s,u}(X^1))du+
\int_s^t\sigma( \mathcal{X}^1_{s,u}(X^1))dW_u+L^1_{s,t}\geq 0 ,\\
L^1_{s,t}=&\int_0^t 1_{\mathcal{X}^1_{s,u}(X^1)=0}dL^1_{s,u}.  \label{eq:ref2}
\end{align}
Here, a solution means a couple $(\mathcal{X}^1_{s,\cdot},L^1_{s,\cdot}) $
which satisfies the above conditions, where $L^1_{s,\cdot}$ is a
nondecreasing process which only increases when $\mathcal{X}^1_{s,u}(X^1)=0 $
with initial condition $L^1_{s,s}=0$. We assume that $b,\sigma :\mathbb{R}\to%
\mathbb{R}$ are Lipschitz functions. In fact, we assume that there exists
constants $C_b,C_\sigma>0 $ such that 
\begin{align}  \label{LIP_example1}
|b(y)-b(x)|^2\leq C_b|y-x|^2,\quad |\sigma(y)-\sigma(x)|^2\leq
C_\sigma|y-x|^2.
\end{align}
Therefore, there exists a unique solution to the above equation, see for
example \cite[Theorem 3.1]{LiSz}, and this solution has finite moments. In
particular, we have ${\mathbb{E}}[ \mathcal{X}^1_{s,t}(X^1)^2]\leq C(1+{%
\mathbb{E}}[(X^1)^2]) $. We also define for $X^{2}\geq 0$, the one step
approximation process 
\begin{equation} \label{eq:refX2}
\mathcal{X}_{s,t}^{2}(X^{2})=X^{2}+b(X^{2})(t-s)+\sigma
(X^{2})(W_{t}-W_{s})+L_{s,t}^{2}\geq 0,
\end{equation}
with $L^{2}$ a nondecreasing process such that $L_{s,t}^{2}= \int_{s}^{t}1_{%
\mathcal{X}_{s,u}^{2}(X^{2})=0}dL_{s,u}^{2}$ with initial condition $%
L^2_{s,s}=0$. We similarly have ${\mathbb{E}}[ \mathcal{X}%
^2_{s,t}(X^2)^2]\leq C(1+{\mathbb{E}}[(X^2)^2]) $.

Furthermore, we assume the following contraction property: for some
positive constant $\bar{b} >C_\sigma$ 
\begin{equation}
(x-y)(b(x)-b(y))\leq -\bar{b}\left\vert x-y\right\vert ^{2}.  \label{1d}
\end{equation}%
We intend to prove that the probabilistic representations implied by $ \mathcal{X}^i $, $ i=1,2 $ satisfy the   conditions in Theorem \ref{stochastic}. In fact, we define for $%
\mu=\mathcal{L}(X)\in \mathcal{P}_2(\mathbb{R}) $ 
\begin{equation*}
\theta _{s,t}(\mu )=\mathcal{L}(\mathcal{X}^1_{s,t}(X)) ,\quad \Theta
_{s,t}(\mu )= \mathcal{L}(\mathcal{X}_{s,t}^{2}(X)).
\end{equation*}

We fix now the time grid $\pi=\{t_n,n \in {\mathbb{N}}\}$, with $t_{n}=s+\sum_{i=1}^{n}\frac{1}{1+i}$ 
and define $\Theta _{s,t_{n}}^{\pi }(\mu )$ by~\eqref{def_thetapi}.
We also define the Euler scheme $X_{0}^{\pi }=X$ and $X_{k+1}^{\pi }=%
\mathcal{X}_{t_{k},t_{k+1}}^{2}(X_{k}).$ Then, if the law of $X_{0}$ is $\mu
,$ then we have $\mathcal{L}(X_{n}^{\pi })=\Theta _{s,t_{n}}^{\pi }(\mu ).$

\begin{theorem}
\label{thm_Reflected2} Suppose that~\eqref{LIP_example1} and (\ref{1d}) hold
with $2\bar{b} >C_ \sigma$. Let $b_{\ast }\in(0,2\bar{b}-C_ \sigma)$. Then, $%
\theta $ and $\Theta $ satisfy the $2$-Foster-Lyapunov condition, are $%
(2,b_{\ast },1)$-coupled and $\theta $ is self-coupled with the same
parameters.

Furthermore, the Markov process $(\mathcal{X}_{s,t}^{1}(X))_{t\geq s}$
admits a unique invariant measure $\nu\in \mathcal{P}_2(\mathbb{R}_+)$, and there exists $C\in {\mathbb{R}}_+$ such that for
every $\mu \in {\mathcal{P}}_2({\mathbb{R}}_+)$, $n\ge 1$ and $t>s$, 
\begin{align*}
W_{2}^{2}(\Theta _{s,t_{n}}^{\pi }(\mu ),\nu )\leq& C(\left\Vert \mu
\right\Vert _{2}^{2}+1) \frac{1}{n^{b_*\wedge 1}}(1+\ln(n+1)^{1_{b_*=1}}),\\W_{2}^{2}(\theta
_{s,t}(\mu ),\nu )\leq& C(\left\Vert \mu \right\Vert _{2}^{2}+1)\times
(1+(t-s)^{1_{b_*=1}})e^{-(b_*\wedge 1)(t-s)}.
\end{align*}
\end{theorem}

The proofs are based on upper bounds for the $ L^2 $ norms of $ \mathcal{X}^1(X^1) -\mathcal{X}^2(X^1) $ as well as  
$ \mathcal{X}^1(X^1) -\mathcal{X}^1(X^2) $. These are performed using Ito's formula and then the hypotheses give that the joint probabilistic representation is satisfied  as stated in Definition~\ref{def:214}. We refer to Appendix~\ref{App_additional} for the details.

\subsection{A neuronal model in $W_1$}
\label{sec:neu}

In this section, we illustrate how our general framework fits an example with 
dynamics of McKean-Vlasov type. The following neuronal model has been
developed and studied in \cite{DGLP},\cite{FoLo}, and \cite{CTV}. Let $X$ be
the unique solution of the following mean field type stochastic equation 
\begin{align}  \label{Neuronal_model}
X_t=X_0+\int_0^tb(X_s)ds+J\int_0^t{\mathbb{E}}[
f(X_s)]ds-\int_0^t\int_{%
\mathbb{R} _+}X_{u-}1_{z\leq f(X_u-)\wedge M }dN(u,z),
\end{align}
where $N$ is  a random Poisson measure\footnote{%
We have preferred the notation $dN(u,z)$ instead $N(du,dz)$ for random
Poisson measures in order to get shorter formulas, but we admit that both
notations are equivalent.} with compensator measure given by the Lebesgue
measure on $\mathbb{R}_+\times\mathbb{R}_+ $. Furthermore,  $M>0$ and $b,f:{\mathbb{R}}_+
\to {\mathbb{R}}_+$ satisfy: 
\begin{equation}  \label{Lip_neuron}
\forall x,y\ge 0, |b(x)-b(y)|\le C_b|x-y|,\ |f(x)-f(y)|\le C_f|x-y| \text{
and }f(0)=0.
\end{equation}
In the case that $\|f\|_\infty\leq M$, we recover the model
in~\cite{CTV}.

Let us define $f_M(x)=f(x)\wedge M$. Then, for $X^i\geq 0 $, we define 
\begin{align}
\mathcal{X}^i_{s,t}(X^i)=&X^i +\int_s^tb_i(u)du+J\int_s^t{\mathbb{E}}%
[f_i(u)]du  \notag \\
&-\int_s^t\int_{\mathbb{R} _+}\mathcal{X}^i_{s,u-}(X^i)1_{z\leq f_M(\mathcal{%
X}^i_{s,u-}(X^i))}dN(u,z),  \label{def:Xi_Neur}
\end{align}
and 
\begin{align*}
(b_i(u),f_i(u))=%
\begin{cases}
(b(\mathcal{X}^1_{s,u}(X^1)),f( \mathcal{X}^1_{s,u}(X^1))) ; & \text{ if }%
i=1, \\ 
(b(X^2),f(X^2)) ; & \text{ if }i=2.%
\end{cases}%
\end{align*}
We assume moreover the following contraction condition: there
exists $\bar{b}>0 $ such that for any $x,y\ge 0$, 
\begin{align}
&\mathrm{sgn}(y-x)(b(y)-b(x))-|y-x|(f_M(y)\vee f_M(x))  \notag \\
&+x(f_M(y)-f_M(x))^+ +y(f_M(x)-f_M(y))^+ \leq -\bar{b} \left\vert
x-y\right\vert . \label{eq:cont_W1}
\end{align}
This condition combines drift and jump effects in the model to create the contraction. We however remark that in the particular case that $f_M$ and $\R_+ \ni x\mapsto \frac{x}{f_M(x)}$  are  nondecreasing functions (which holds for example if $f(x)=x^a$ with $a\in[0,1]$), we have 
\begin{align*}
	-|y-x|(f_M(y)\vee f_M(x))  +x(f_M(y)-f_M(x))^+ +y(f_M(x)-f_M(y))^+ \leq  0,
\end{align*}
since, for $x<y$ we have $-(y-x)f_M(y)+x (f_M(y)-f_M(x))\le -yf_M(y) +xf_M(y)+yf_M(x)-xf_M(x)=-(y-x)(f_M(y)-f_M(x))\le 0$.
In this case, \eqref{eq:cont_W1} is satisfied if $\mathrm{sgn}(y-x)(b(y)-b(x))  \leq -\bar{b} \left\vert
x-y\right\vert $.

\begin{remark}
	Let us note that $\mathcal{X}^2_{s,t}$ stands for the approximation of the continuous process $\mathcal{X}^1_{s,t}$ but is different from the Euler scheme. Indeed, the coefficient of the driving Poisson measure depends on $\mathcal{X}^2_{s,u-}(X^2)$, not on $X_2$. This plays a crucial role in the analysis below, precisely for the term named $I_3(s,t)$ in the proof of Theorem~\ref{thm_neurW1}. Nonetheless, the process $\mathcal{X}^2_{s,t}(X^2)$ can be sampled exactly, as a piecewise deterministic Markov process.  
\end{remark}

From the Lipschitz condition~\ref{Lip_neuron}, there exists a constant $K>0 $
(which depends on $b(0)$) such that $|b(x)|\leq K (1+|x|) $. We define the following
maps for $\mu=\mathcal{L}(X)\in \mathcal{P}_2(\mathbb{R}_+) $ 
\begin{equation*}
\theta _{s,t}(\mu )=\mathcal{L}(\mathcal{X}^1_{s,t}(X)) ,\quad \Theta
_{s,t}(\mu )= \mathcal{L}(\mathcal{X}_{s,t}^{2}(X)).
\end{equation*}

We have under these assumptions the following result.

\begin{theorem}
\label{thm_neurW1} Assume that~\eqref{Lip_neuron} and~\eqref{eq:cont_W1}
hold with $\bar{b}- C_fJ>0$. Then, $\theta $ and $\Theta $ satisfy the $1$%
-Foster-Lyapunov condition, are $(1,b_*,1)$ coupled for every $b_*=\bar{b}%
-C_ f J$, and $\theta $ is self-coupled with the same parameters.

Furthermore, the Markov process $(\mathcal{X}_{s,t}^{1}(X))_{t\geq s}$
admits a unique invariant measure $\nu \in {\mathcal{P}}_1({\mathbb{R}_+})$,
and there exists $C\in {\mathbb{R}}_+$ such that for every $\mu \in {%
\mathcal{P}}_1({\mathbb{R}_+})$, $n\ge 1$ and $t>s$, 
\begin{align*}
W_{1}(\Theta _{s,t_{n}}^{\pi }(\mu ),\nu )\leq& C(\left\Vert \mu \right\Vert
_{1}+1) \frac{(1+\ln(n+1))^{\mathbf{1}_{b_*= 1}}}{n^{b_*\wedge 1}}, \\
 W_{1}(\theta _{s,t}(\mu ),\nu )\leq& C(\left\Vert \mu \right\Vert _{1}
+1)\times (1+(t-s)^{\mathbf{1}_{b_*= 1}}) e^{- (b_*\wedge 1) (t-s)},
\end{align*}
where $t_k=s+\sum_{i=1}^{k}\frac{1}{1+i}$ and $\pi=\{t_0<\dots<t_n\}$.
\end{theorem}
\begin{proof}
For $0\le t-s\le 1$, we obtain by taking expectations on the stochastic
equation that defines $\mathcal{X}^i_{s,\cdot}(X^i) $ in~\eqref{def:Xi_Neur}, we have for $i=1,2$ (recall that $\mathcal{X}%
^i_{s,t}(X^i)\ge 0$), 
\begin{align}
	\label{eq:s}
	{\mathbb{E}} [\mathcal{X}^i_{s,t}(X^i)]-{\mathbb{E}}[X^i]=&\int_s^t{\mathbb{E%
	}}[b_i(r)+J{\mathbb{E}} [f_i(r)]]dr-\int_s^t{\mathbb{E}}[\mathcal{X}%
	^i_{s,r}(X^i)f_M(\mathcal{X}^i_{s,r}(X^i))]dr.
\end{align} 
Next 
using the Lipschitz property~\eqref{Lip_neuron} 
\begin{equation*}
{\mathbb{E}}[\mathcal{X}^i_{s,t}(X^i)]\le {\mathbb{E}}[X^i] +b(0)(t-s)+
(C_b+JC_f) \int_s^t {\mathbb{E}}[\mathcal{X}^i_{s,u}]du,
\end{equation*}
since the jump contribution is nonpositive. This gives 
\begin{align}  \label{eq:est1N}
{\mathbb{E}}[\mathcal{X}^i_{s,t}(X^i)]\leq C(1+{\mathbb{E}}[X^i]),
\end{align}
with $C $ depending on $b(0) $, $C_f $ and $J $.

With these estimates we can now prove~\eqref{AN2BIS}. We use It\^o's fomula
for $F(y)=|y|$, and we use a regularization procedure as in the proof of 
Theorem~\ref{thm_Boltz_W1}. We get 
\begin{align*}
{\mathbb{E}}[| \mathcal{X}^i_{s,t}(X^i)-X^i|]&=\int_s^t{\mathbb{E}}[1_{%
\mathcal{X}^i_{s,r}(X^i)-X^i>0}(b_i(r)+J{\mathbb{E}} [f_i(r)])]dr \\
&+\int_s^t{\mathbb{E}}[(|X^i|-|\mathcal{X}^i_{s,r}(X^i)-X^i|)f_M(\mathcal{X}%
^i_{s,r}(X^i))]dr.
\end{align*}
Using \eqref{eq:est1N}, we obtain for $ i=1 $
\begin{align}
{\mathbb{E}}[| \mathcal{X}^1_{s,t}(X^1)-X^1|]&\le \int_s^t b(0)+ (C_b+C_fJ+M)%
{\mathbb{E}}[\mathcal{X}^1_{s,r}(X^1)]dr  \notag \\
&\leq C(1+{\mathbb{E}}[X^1])(t-s).  \label{eq:mn2w1}
\end{align}
For $ i=2 $ the estimation is done similarly. Note that these two inequalities also imply that for $\mathcal{Y}_{s,t}:=
\mathcal{X}^1_{s,t}(X^1)-
\mathcal{X}^2_{s,t}(X^2)$, we have
\begin{align}
	\label{eq:t1}
\left|{\mathbb{E}}[| \mathcal{Y}_{s,t}|-|\mathcal{Y}_{s,s}|]\right|\leq C(1+{\mathbb{E}}[X^1]+{\mathbb{E}}[X^2])(t-s).
\end{align}

The contraction hypothesis~\eqref{eq:cont_W1} with $x=0$ gives $%
b(y)-yf_M(y)\le b(0) - \bar{b}y$ for $y\ge 0$, and we thus obtain  applying this inequality to \eqref{eq:s}
\begin{align*}
{\mathbb{E}} [\mathcal{X}^i_{s,t}(X^i)]-{\mathbb{E}}[X^i]\leq \int_s^t(-\bar{%
b}+JC_f) {\mathbb{E}}[\mathcal{X}^i_{s,r}(X^i)]dr+b(0)(t-s)+C(1+{\mathbb{E}}%
[X^i])(t-s)^2,
\end{align*}
by using~\eqref{eq:mn2w1} for $i=2$. We now use that $\mathcal{X}%
^i_{s,r}(X^i)\le X^i+ |\mathcal{X}^i_{s,r}(X^i)- X^i|$ with \eqref{eq:mn2w1}
to obtain \eqref{AN2BIS} for $p=1 $ with $\varepsilon=1$.

Next, in order to prove \eqref{AN3}, we write
\begin{align*}
\mathcal{Y}_{s,t}=\mathcal{Y}_{s,s}+&\int_{s}^{t}b(\mathcal{X}%
^1_{s,r}(X^1))-b(X^2)dr+J \int_{s}^{t}{\mathbb{E}} [f(\mathcal{X}%
^1_{s,r}(X^1))]-{\mathbb{E}} [f({X}^2)]ds\\&-\int_{s}^{t}\int_{\mathbb{R}
_{+}}c(r,z)dN(r,z).
\end{align*}%
with%
\begin{equation*}
c(r,z)=\mathcal{X}^1_{s,r}(X^1)1_{z\leq f_M(\mathcal{X}^1_{s,r}(X^1))}-%
\mathcal{X}^2_{s,r}(X^2)1_{z\leq f_M(\mathcal{X}^2_{s,r}(X^2))}
\end{equation*}

Then, regularizing $y\mapsto|y|$ and using It\^{o}'s formula and taking the
limit as $\varepsilon \rightarrow 0$ we get%
\begin{align*}
|\mathcal{Y}_{s,t}| =&|\mathcal{Y}_{s,s}|+\int_{s}^{t}\mathrm{sgn}(\mathcal{Y%
}_{s,r})(b(\mathcal{X}^1_{s,r}(X^1))-b({X}^2)) dr \\
&+J\int_{s}^{t}\mathrm{sgn}(\mathcal{Y}_{s,r}) ({\mathbb{E}}[ f(\mathcal{X}%
^1_{s,r}(X^1))]-{\mathbb{E}} [f(X^2) ])dr \\
&+\int_{s}^{t}\int_{\mathbb{R}_{+}}|\mathcal{Y}_{s,r-}-c(r,z)|-|\mathcal{Y}%
_{s,r-}| dN(r,z) \\
=:&|\mathcal{Y}_{s,s}|+\sum_{i=1}^{3}I_{i}(s;t).
\end{align*}

We consider the estimate for ${\mathbb{E}}[I_{3}(s;t)]$. We have

\begin{align*}
{\mathbb{E}}[I_3(s;t)]=&{\mathbb{E}}\int_{s}^{t}\big(-|\mathcal{Y}_{s,r}|f_M(%
\mathcal{X}^1_{s,r}(X^1))\wedge f_M(\mathcal{X}^2_{s,r}(X^2)) \\
&+(\mathcal{X}^2_{s,r}(X^2)-|\mathcal{Y}_{s,r}|)(f_M(\mathcal{X}%
^1_{s,r}(X^1))- f_M(\mathcal{X}^2_{s,r}(X^2)))^+ \\
&+(\mathcal{X}^1_{s,r}(X^1)-|\mathcal{Y}_{s,r}|)(f_M(\mathcal{X}%
^2_{s,r}(X^2))- f_M(\mathcal{X}^1_{s,r}(X^1)))^+\big) \ dr.
\end{align*}
Let denote by $(y,x) $ the variables $(\mathcal{X}%
^1_{s,u}(X^1),\mathcal{X}^2_{s,u}(X^2)) $ and write 
\begin{align*}
&-|y-x|f_M(x)\wedge
f_M(y)+(x-|y-x|)(f_M(y)-f_M(x))^++(y-|y-x|)(f_M(x)-f_M(y))^+ \\
=&-|y-x|(f_M(y)\vee f_M(x))+x(f_M(y)-f_M(x))^++y(f_M(x)-f_M(y))^+.
\end{align*}
Therefore, using the contraction hypothesis~\eqref{eq:cont_W1} and the
Lipschitz property~\eqref{Lip_neuron}, we get 
\begin{align*}
{\mathbb{E}}[|\mathcal{Y}_{s,t}|]&\leq {\mathbb{E}}[|\mathcal{Y}_{s,s}|] +{%
\mathbb{E}}\left[ \int_s^t(-\bar{b}+JC_f)|\mathcal{Y}_{s,r}|dr \right] +(C_b
+C_fJ){\mathbb{E}} \int_s^t |\mathcal{X}^2_{s,r}(X^2)-X^2| dr \\
& \le {\mathbb{E}}[|\mathcal{Y}_{s,s}|]+{\mathbb{E}}\left[ \int_s^t(-\bar{b}%
+JC_f)|\mathcal{Y}_{s,r}|dr\right]+C(1+{\mathbb{E}}[|X^2|])(t-s)^2,
\end{align*}
by using~\eqref{eq:mn2w1} for the last inequality. 
Using \eqref{eq:t1}, we see that \eqref{AN3} is satisfied for $b_*=\bar{b}-JC_f$ since $e^{-(\bar{%
		b}-JC_f)t}=_{t\to 0} 1-(\bar{b}-JC_f)t+O(t^2)$.
The proof that $\theta $ is self-coupled can be obtained as in the above proof.
\end{proof}

\subsection{Boltzmann type equations}
\label{sec:bol} 
In this section, we present an example based on Boltzmann type equations.
These are stochastic equations driven by a Poisson point measure whose compensator depends on the law of the solution of the equation.
Such equations appear in the modelisation of interacting physical particles.

\blue{ Here, we work with Poisson
point measures that corresponds to the ``finite variation regime". This is a
simpler framework, which allows to consider a more complex dependence of the
Poisson point measure on the position of the particle. This is
the case of the Boltzmann equation with hard potentials for example (here we
deal with a drastically simplified example). In order to be able to deal
with the indicator function which appears in the stochastic integral we have
to use $L^{1}$ estimates, and the appropriate distance here is $W_{1}$. Note that an alternative setting is the ``martingale regime", for which we have to work with a compensated Poisson point measure. In this case, all the estimates have to be done in $L^{2}$ and the natural Wasserstein distance is then $W_{2}$. This case, which corresponds to Maxwellian molecules,  is addressed in Appendix~\ref{App_Boltz}.} 

 We consider the following $d-$dimensional stochastic equation 
\begin{align}  \label{def_X1_Boltzmann_W1}
X_{s,t}^{1}&(X^{1}) =X^{1}+\int_{s}^{t}b(X_{s,r}^{1}(X^{1}))dr 
\\
&+\int_{s}^{t}\int_{E}\int_{\mathbb{R}_{+}}\int_{\mathbb{R}%
^{d}}c(v,z,X_{s,r-}^{1}(X^{1}))1_{ \{u\leq \gamma
(v,z,X_{s,r-}^{1}(X^{1}))\}}dN_{\mathcal{L} (X_{s,r}^{1})}(v,u,z,r)  \notag
\end{align}
where $b:{\mathbb{R}}^d\to {\mathbb{R}}^d$, $c:{\mathbb{R}}^d \times E\times 
{\mathbb{R}}^d  \to {\mathbb{R}}^d$ and $dN_{\mathcal{L}(X_{s,r}^{1})}(v,u,z,r)$ is the Poisson point measure
on $\mathbb{R}^{d}\times \mathbb{R}_{+}\times E\times \mathbb{R}_{+}$ with
compensator $d\widehat{N}_{\mathcal{L}(X_{s,r}^{1})}(v,u,z,r)=\mathcal{L}
(X_{s,r}^{1})(dv)du\Lambda (dz)dr.$ Here, $ \Lambda $ is a measure on the measurable space $ E $.

We associate to the above equation, the one step Euler scheme 
\begin{align}
X_{s,t}^{2}(X^{2})=&X^{2}+\int_{s}^{t}b(X^{2})dr  \label{FV2}\\
&+\int_{s}^{t}\int_{\mathbb{R}%
^{d}} \int_{\mathbb{R}_{+}}\int_{E}c(v,z,X^{2})1_{\{u\leq \gamma
(v,z,X^{2})\}}dN_{ \mathcal{L}(X^{2})}(v,u,z,r) \notag
\end{align}
where $dN_{\mathcal{L}(X^{2})}(v,u,z,r)$ is the Poisson point measure with
compensator\linebreak $d\widehat{N}_{\mathcal{L}(X^{2})}(v,u,z,r)=\mathcal{L}
(X^{2})(dv)du\Lambda (dz)dr.$

Moreover, we define the maps on the Wasserstein space for $ \mathcal{L}(X)=\mu $
\begin{equation}
\theta _{s,t}(\mu )=\mathcal{L}(X_{s,t}^{1}(X)) 
,\quad \Theta _{s,t}(\mu )=\mathcal{L}(X_{s,t}^{2}(X)).  \label{FV3}
\end{equation}
We will work under the following hypotheses that guarantee that the equation~\eqref{def_X1_Boltzmann_W1} has a weak solution and that $\theta_{s,t}$ is well defined.

\textbf{H1.} For every $x,y\in \mathbb{R}^d$ 
\begin{equation}
\left\langle x-y,b(x)-b(y)\right\rangle \leq -\overline{b}\left\vert
x-y\right\vert ^{2},\quad \left\vert b(x)-b(y)\right\vert \leq C_b\left\vert
x-y\right\vert  \label{FV4}
\end{equation}

\textbf{H2}. For every $x,x^{\prime },v,v^{\prime }\in \mathbb{R}^d$ and $%
z\in E $ there exists a measurable function $\bar{c}:E\to \mathbb{R}_+ $
such that $Q:= \int_{E}\overline{c}(z)\Lambda (dz)<\infty$ and 
\begin{align}  \label{FV5}
\left\vert (c(v,z,x)-c(v^{\prime },z,x^{\prime }))\gamma(v,z,x)\right\vert
+&\left\vert c(v^{\prime },z,x^{\prime })(\gamma (v,z,x)-\gamma (v^{\prime
},z,x^{\prime }))\right\vert  \notag \\
& \leq \overline{c}(z)(\left\vert v-v^{\prime }\right\vert +\left\vert
x-x^{\prime }\right\vert ),\\
\left\vert c\gamma (v,z,x)\right\vert &\leq \overline{c}(z)(1+\left\vert
v\right\vert +\left\vert x\right\vert ).  \label{FV5'}
\end{align}

\begin{lemma}
\label{lemmetech_BW1} Under Assumption \textbf{H2}, we have 
\begin{align*}
\int_E \int_{{\mathbb{R}}_+} \left| c(v,z,x)1_{u\le
\gamma(v,z,x)}-c(v^{\prime },z,x^{\prime })1_{u\le \gamma(v^{\prime
},z,x^{\prime })}\right|du \Lambda(dz)\le Q (|v-v^{\prime }|+|x-x^{\prime }|).
\end{align*}
\end{lemma}

\begin{proof}
The integral is upper bounded by 
\begin{align*}
&\int_E \int_{{\mathbb{R}}_+} \left| c(v,z,x)-c(v^{\prime },z,x^{\prime
})\right|1_{u\le \gamma(v,z,x)} du \Lambda(dz) \\
&+ \int_E \int_{{\mathbb{R}}_+} |c(v^{\prime },z,x^{\prime })| \left|
1_{u\le \gamma(v,z,x)}-1_{u\le \gamma(v^{\prime },z,x^{\prime })}\right|du
\Lambda(dz) \\
=& \int_E \left| c(v,z,x)-c(v^{\prime },z,x^{\prime })\right| \gamma(v,z,x)
+|c(v^{\prime },z,x^{\prime })| \left| \gamma(v,z,x)-\gamma(v^{\prime
},z,x^{\prime })\right| \Lambda(dz) \\
\le & Q (|v-v^{\prime }|+|x-x^{\prime }|).\qedhere
\end{align*}
\end{proof}

Under hypotheses {\bf H1} and {\bf H2}, we know by~\cite[Theorem 3.5]{AB} that there exists
a unique  time homogeneous continuous  flow $\theta_{s,t}:\mathcal{P}_1({\mathbb{R}}^d) \to \mathcal{P}_1({%
\mathbb{R}}^d)$ such that  
\begin{equation}  \label{sewing_Boltzmann_W1}
\forall T>0,\exists C>0,\forall 0\le s\le t\le T, \ \mathcal{D}%
_{1}(\theta_{s,t},\Theta_{s,t})\le C(t-s)^2,
\end{equation}
with $\mathcal{D}_{1}(\Theta,\overline{\Theta}):=\sup_{\mu \in \mathcal{P}_1({%
\mathbb{R}}^d)}\frac{W_{1}(\Theta(\mu),\overline{\Theta}(\mu))}{1+\|\mu\|_{1}%
}$, and we also have 
\begin{equation}  \label{FN6'}
W_1(\theta_{s,t}(\mu),\Theta_{t^n_n,t^n_{n-1}}\circ\dots
\circ\Theta_{t^n_0,t^n_1}(\mu) )\to_{n\to \infty} 0
\end{equation}
for any sequence $s=t^n_0<t^n_1<\dots<t^n_n=t$ such that $\max_{1\le i\le n}
t^n_i -t^n_{i-1} \to_{n\to \infty} 0$. Besides, \cite[Theorem 3.8]{AB} gives
the existence of a solution to~\eqref{def_X1_Boltzmann_W1} such that $%
\mathcal{L}(X^1_{s,t}(X^1))=\theta_{s,t}(\mu^1)$ when $\mathcal{L}(X^1)= \mu^1 \in \mathcal{P}%
_1({\mathbb{R}}^d)$, and \cite[Proposition 3.9]{AB} gives that any solution of~\eqref{def_X1_Boltzmann_W1} has the same marginal laws. We have the following result for the convergence of a scheme towards the invariant measure.

\begin{theorem}
\label{thm_Boltz_W1} Assume that~(\ref{FV4}), (\ref{FV5}) hold with $2Q< 
\overline{b}$. Then, $\theta $ and $\Theta $ satisfy the $1$-Foster-Lyapunov
condition, are $(1,b_*,1)$ coupled for $b_*=\bar{b}-2Q$, and $\theta$ is
self-coupled with the same parameters.

Furthermore, the Markov process $({X}_{s,t}^{1}(X))_{t\geq s}$
admits a unique invariant measure $\nu \in {\mathcal{P}}_1({\mathbb{R}}^d)$,
and there exists $C\in {\mathbb{R}}_+$ such that for every $\mu \in {%
\mathcal{P}}_1({\mathbb{R}}^d)$, $n\ge 1$ and $t>s$, 
\begin{align*}
&W_{1}(\Theta _{s,t_{n}}^{\pi }(\mu ),\nu )\leq C(\left\Vert \mu \right\Vert
_{1}+1) \frac{(1+\ln(n+1)^{\mathbf{1}_{b_*= 1}})}{n^{b_*\wedge 1}}, \\
& W_{1}(\theta _{s,t}(\mu ),\nu )\leq C(\left\Vert \mu \right\Vert _{1}
+1)\times (1+(t-s)^{\mathbf{1}_{b_*= 1}}) e^{- b_*\wedge 1 (t-s)},
\end{align*}%
where $t_k=s+\sum_{i=1}^{k}\frac{1}{1+i}$ and $\pi=\{t_0<\dots<t_n\}$. 
\end{theorem}

\begin{proof}
The proof is an application of  Theorem~\ref{th:app}. From~\eqref{sewing_Boltzmann_W1} and Lemma~\ref{lem_coupling}, it
is sufficient to  check that $\Theta$ is $(1,b_*,1)$ self-coupled for $b_*=%
\overline{b}-2Q$, and satisfies the $1$-Foster-Lyapunov condition.

We consider a filtered probability space $(\Omega ,\mathcal{F},\mathbb{P})$
endowed with a Poisson point measure $dN(w,u,z,r)$ of intensity $dwdu\Lambda
(dz)dr$ on $(0,1)\times \mathbb{R}_{+}\times E\times \mathbb{R}_{+}$. We
denote by $(\mathcal{F}_{t})_t$ the filtration generated by this random measure.
Moreover, we consider $X^{i}\in L^{1}(\Omega )$, $i=1,2$, which are $\mathcal{%
F}_{s}$ measurable and such that the distribution of $(X^{1},X^{2})$ is an $%
W_{1}$-optimal coupling for $\mu ^{1}$ and $\mu ^{2}.$ In particular $%
X^{i}\sim \mu ^{i}$. We consider a measurable map $\tau:[0,1]\to {\mathbb{R}}%
^d \times {\mathbb{R}}^d$ such that $\mathcal{L}(\tau(U))=\mathcal{L}(X^1,X^2)$ when $U$ is a $ [0,1] $ uniform distributed random variable. We note $\tau^i$, $i=1,2$ its
coordinates, and we define 
\begin{equation}
x_{s,t}^{i}=X^{i}+\int_{s}^{t}b(X^{i})dr+\int_{s}^{t}
\int_{E}\int_{0}^{1}q(\tau^{i}(w),u,z,X^{i})dN(w,u,z,r)  \label{EQ2}
\end{equation}
with 
\begin{equation*}
q(v,u,z,x)=c(v,z,x)1_{\{u\leq \gamma (v,z,x)\}}.
\end{equation*}
It is easy to check that $\mathcal{L}(x_{s,t}^i)=\mathcal{L} (\Theta_{s,t}(\mu^i))$.

Next, we check the Foster-Lyapunov condition. Notice that, as an immediate
consequence of our hypothesis (\ref{FV4}) we have, 
\begin{align}
\frac{1}{\left\vert x\right\vert }\left\langle x,b(x)\right\rangle =&\frac{1 
}{\left\vert x\right\vert }\left\langle x-0,b(x)-b(0)\right\rangle +\frac{1}{
\left\vert x\right\vert }\left\langle x,b(0)\right\rangle \leq -\overline{b}
\left\vert x\right\vert +\left\vert b(0)\right\vert  \label{FN9} \\
\left\vert b(x)\right\vert \leq &K(1+\left\vert x\right\vert ).  \label{FN10}
\end{align}
We then observe that 
\begin{align}
{\mathbb{E}}[\vert x_{s,t}^{1}(X^{1})-X^{1}\vert] &\leq (t-s)%
\mathbb{E}[|b(X^{1})|]+{\mathbb{E}}\int_{s}^{t}\int_{E}\int_{0}^{1}\left\vert c\gamma (\tau^{1}(w),z,X^{1})\right\vert dw \Lambda (d z)dr  \notag \\
&\leq  (t-s)\left(K(1+ \mathbb{E}[|X^1|])+ Q(1+2\mathbb{E}[|X^1|])\right),
\label{maj_diff_BW1}
\end{align}
using~\eqref{FN10} and~\eqref{FV5'}.

We now apply It\^o's formula with the function $F(x)=\left\vert x\right\vert 
$. In order to do this we consider $\psi_\varepsilon:{\mathbb{R}}_+\to {%
\mathbb{R}}_+$ defined by $\psi_\varepsilon(x)=\frac{x^2}{2\varepsilon}$ for 
$x\in [0,\varepsilon]$ and $\psi_\varepsilon(x)=x-\varepsilon/2$ for $%
x>\varepsilon $. 
We then define $F_{\varepsilon }(x)=\psi _{\varepsilon }(\left\vert
x\right\vert )=\psi _{\varepsilon }(F(x))$: it is $C^1$ and we have $%
\lim_{\varepsilon \rightarrow 0}F_{\varepsilon }(x)=F(x)=\left\vert
x\right\vert $ and $\lim_{\varepsilon \rightarrow 0}\partial
_{i}F_{\varepsilon }(x)=\frac{x_{i}}{\left\vert x\right\vert }.$ We use now
It\^{o}'s formula (see e.g.~\cite[Theorem IV.18]{RW} for the finite
variation case) and we get 
\begin{align*}
	{\mathbb{E}}[F_{\varepsilon }(x_{s,t}^{1})] =&{\mathbb{E}}[F_{\varepsilon
	}(X^{1})]+{\mathbb{E}}\int_{s}^{t}\left\langle \nabla F_{\varepsilon
	}(x_{s,r}^{1}),b(X^{1})\right\rangle dr \\
	&+{\mathbb{E}}\int_{s}^{t}\int_{E}\int_0^\infty\int_{0}^{1} (F_{\varepsilon
	}(x_{s,r}^{1}+q(\tau^{1}(w),u,z,X^{1}))-F_{\varepsilon }(x_{s,r}^{1}))dwdu
	\Lambda(dz)dr.
\end{align*}
Passing to the limit with $\varepsilon \rightarrow 0$ we get 
\begin{align}
{\mathbb{E}}[\vert x_{s,t}^{1}\vert ]=&{\mathbb{E}}[\vert X^{1}\vert ]+{%
\mathbb{E}}\int_{s}^{t}\frac{1}{\vert x_{s,r}^{1}\vert }\left\langle
x_{s,r}^{1},b(X^{1})\right\rangle dr  \label{eq:ret} \\
&+{\mathbb{E}}\int_{s}^{t}\int_{E}\int_0^\infty\int_{0}^{1} (\left\vert
x_{s,r}^{1}+q(\tau^{1}(w),u,z,X^{1})\right\vert-|x_{s,r}^{1}| )dwdu \Lambda(dz)dr.
\notag
\end{align}

By using (\ref{FN9}), we get 
\begin{align*}
\frac{1}{\left\vert x_{s,r}^{1}\right\vert }\left\langle
x_{s,r}^{1},b(X^{1})\right\rangle &=\frac{1}{\left\vert
	x_{s,r}^{1}\right\vert }\left\langle x_{s,r}^{1},b(x_{s,r}^{1})\right\rangle
+\frac{1}{\left\vert x_{s,r}^{1}\right\vert }\left\langle
x_{s,r}^{1},b(X^{1})-b(x_{s,r}^{1})\right\rangle \\
&\leq -\overline{b}\left\vert x_{s,r}^{1}\right\vert +\left\vert
b(0)\right\vert +C_b\left\vert X^{1}-x_{s,r}^{1}\right\vert \\
&\leq -\overline{b}\vert X^{1}\vert +\left\vert b(0)\right\vert +(\overline{%
	b}+C_b)\left\vert X^{1}-x_{s,r}^{1}\right\vert .
\end{align*}
In particular, we get from~\eqref{maj_diff_BW1} 
\begin{equation*}
{\mathbb{E}}\int_{s}^{t}\frac{1}{\left\vert x_{s,r}^{1}\right\vert }%
\left\langle x_{s,r}^{1},b(X^{1})\right\rangle dr\leq -%
\overline{b}(t-s) {\mathbb{E}}[|X_1|] +\left\vert b(0)\right\vert (t-s)+C(1+{%
\mathbb{E}}[\left\vert X^{1}\right\vert] )(t-s)^{2}.
\end{equation*}

We write now 
\begin{align*}
&{\mathbb{E}}\int_{s}^{t}\int_{0}^{1}\int_{E}\int_{\mathbb{R}%
_{+}}(\left|x_{s,r}^{1}+q(\tau^{1}(w),u,z,X^{1})\right|-\left|x_{s,r}^{1}%
\right|) dwdu \Lambda(dz) dr \\
&={\mathbb{E}}\int_{s}^{t}\int_{0}^{1}\int_{E}\int_{\mathbb{R}%
_{+}}\left(\left\vert x_{s,r}^{1}+c(\tau^{1}(w),z,
X^{1})\right\vert-|x_{s,r}^{1}|\right) 1_{\{u\leq \gamma (\tau ^{1}(w),z,
X^{1})\}})dw du \Lambda(dz) dr \\
&\leq {\mathbb{E}}\int_{s}^{t}\int_{0}^{1}\int_{E}\left\vert c\gamma (\tau
^{1}(w),z,X^{1})\right\vert dw\Lambda (dz)dr \\
&\leq Q{\mathbb{E}}\int_{s}^{t}\int_{0}^{1}(1+\vert \tau^{1}(w)\vert +\vert
X^{1}\vert )dw\Lambda (dz)dr =Q(t-s)(1+2{\mathbb{E}}[\vert X^{1}\vert] ).
\end{align*}
Using the two last estimates in \eqref{eq:ret}, we get 
\begin{equation*}
{\mathbb{E}}[\vert x_{s,t}^{1}\vert ]\leq {\mathbb{E}}%
[\vert X^{1}\vert ](1-(\overline{b}-2Q)(t-s))+ (|b(0)|+Q)(t-s)+
C(1+{\mathbb{E}}[\vert X^{1}\vert] )(t-s)^{2},
\end{equation*}
which proves the Foster-Lyapunov condition by Remark~\ref{rk_AN2BIS}.

We now prove that $\Theta$ is $(1,b_*,1)$-self-coupled. Let $%
y_{s,t}=x^1_{s,t}-x^2_{s,t}$. Using It\^{o}'s formula (again, one has to
take first a regularization and then to pass to the limit) we get 
\begin{align}
{\mathbb{E}}[\vert y_{s,t}\vert] =&{\mathbb{E}}[\vert
X^{1}-X^{2}\vert] +{\mathbb{E}}\int_{s}^{t}\left\langle \frac{y_{s,r}}{%
	\left\vert y_{s,t}\right\vert } , b(X_1)-b(X_2)\right\rangle dr \label{eq:refo}\\
&+{\mathbb{E}}\int_{s}^{t}\int_{E}\int_{\mathbb{R}_{+}}\int_{0}^{1}
\left\vert y_{s,r-}+\Delta q(w,u,z)\right\vert -\left\vert
y_{s,r-}\right\vert dN(w,u,z,r) \notag
\end{align}
with $\Delta q(w,u,z)=q(\tau^{1}(w),u,z,X^{1})-q(\tau^{2}(w),u,z,X^{2})$.
We will use this equality to do two different estimations. 

First, we have
\begin{align*}
&\left|{\mathbb{E}}\int_{s}^{t}\int_{E}\int_{\mathbb{R}_{+}}\int_{0}^{1}
(\left\vert y_{s,r-}+\Delta q(r,u,w,z)\right\vert -\left\vert
y_{s,r-}\right\vert )dN(w,u,z,r)\right| \\
&\leq {\mathbb{E}}\int_{s}^{t}\int_{E}\int_{\mathbb{R}_{+}}\int_{0}^{1}
\left\vert \Delta q(w,u,z)\right\vert dwdu \Lambda (dz)dr.
\end{align*} 
By Lemma~\ref{lemmetech_BW1}, we get 
\begin{equation*}
\int_{E}\int_{\mathbb{R}_{+}} \left\vert \Delta q(w,u,z)\right\vert du \Lambda
(dz)\le Q \left(\vert \tau^{1}(w)-\tau^{2}(w)\vert +\vert
X^{1}-X^{2}\vert\right),
\end{equation*}
and thus 
\begin{align}
&{\mathbb{E}}\int_{s}^{t}\int_{E}\int_{\mathbb{R}_{+}}\int_{0}^{1}
\left\vert \Delta q(r,u,w,z)\right\vert dwdu \Lambda (dz)dr \leq 2Q (t-s) {%
\mathbb{E}}[|X^1-X^2|].
\label{eq:refo1}
\end{align}
Second, using that $ b $ is Lipschitz, we obtain 
\begin{align*}
\left|	{\mathbb{E}}[\vert y_{s,t}\vert] -{\mathbb{E}}[\vert
	X^{1}-X^{2}\vert]\right|\leq C(1+{\mathbb{E%
	}}[\vert X^{1}\vert] +{\mathbb{E}}[\vert X^{2}\vert]
	)(t-s).
\end{align*}

In order to obtain \eqref{AN3}, we perform a different estimation in the first integral term in \eqref{eq:refo}.
Notice that, by (\ref{FV4}) and then~\eqref{maj_diff_BW1}, we have 
\begin{align*}
		{\mathbb{E}}\left\langle \frac{y_{s,r}}{\left\vert y_{s,t}\right\vert }%
	,b(X_1)-b(X_2)\right\rangle \leq& -\overline{b}{\mathbb{E}}[\vert
	y_{s,r}\vert] +{\mathbb{E}}\left\langle \frac{y_{s,r}}{\left\vert
		y_{s,t}\right\vert },b(X^{1})-b(x_{s,r}^{1})\right\rangle  \\&+{\mathbb{E}}%
	\left\langle \frac{y_{s,r}}{ \left\vert y_{s,t}\right\vert }%
	,b(x_{s,r}^{2})-b(X^{2})\right\rangle \\
	\leq& -\overline{b}{\mathbb{E}}[\vert y_{s,r}\vert ]+C_{b}({%
		\mathbb{E}}[\vert x_{s,r}^{1}-X^{1}\vert]+{\mathbb{E}}[\vert
	x_{s,r}^{2}-X^{2}\vert] ) \\
	\leq& -\overline{b}{\mathbb{E}}[\vert y_{s,r}\vert] +C(1+{\mathbb{E%
	}}[\vert X^{1}\vert] +{\mathbb{E}}[\vert X^{2}\vert]
	)(r-s).
\end{align*}
We put the above estimate together with \eqref{eq:refo1} in \eqref{eq:refo} and get 
\begin{align*}
&{\mathbb{E}}[\vert y_{s,t}\vert] \\ &\le {\mathbb{E}}[\vert
X^{1}-X^{2}\vert] (1+2Q(t-s))- \overline{b}\int_{s}^{t}{\mathbb{E}}%
\vert y_{s,r}\vert dr+C(1+{\mathbb{E}}[\vert X^{1}\vert]
+{\mathbb{E}}[\vert X^{2}\vert] )(t-s)^{2} \\
&\leq \left( {\mathbb{E}}[\vert
X^{1}-X^{2}\vert] (1+(2Q-\bar{b})(t-s))+C(1+{\mathbb{E}}[\vert
X^{1}\vert] +{\mathbb{E }}[\vert X^{2}\vert] )(t-s)^{2}
\right).
\end{align*}
This proves that $\Theta$ is $(1,b_*,1)$-self-coupled. We conclude by using Lemma~\ref{lem_coupling} and~\eqref{eq_sew}, the latter being given by~\eqref{sewing_Boltzmann_W1}.
\end{proof}


 
\appendix

\section{Proof of Lemma~\protect\ref{lem_tech}}
\label{sec:app}

\textbf{Proof of A.} We first suppose that $b>\varpi(\pi
)\varepsilon $ and we prove that $\sigma _{b,\varepsilon }(n)\leq
C_{b,\varepsilon }\gamma _{n}^{\varepsilon }.$
Let $\zeta >0$ be such that $(\varpi (\pi )+\zeta )\varepsilon <b$. Notice
first that by the definition of $\varpi (\pi )$, there exists $n_{0}$ such
that for $n\geq n_{0}$, have 

\begin{equation}
\frac{\gamma _{n}}{\gamma _{n+1}}\leq 1+(\varpi (\pi )+\zeta )\gamma
_{n+1}\leq e^{(\varpi (\pi )+\zeta )\gamma _{n+1}}.  \label{Ap1}
\end{equation}%
Since $n_{0}$\ depends on the sequence $(\gamma _{n})_{n\in \N}$, so does the
constant $C_{n,\varepsilon }$ that we construct now. \ We prove that the
sequence $v_{n}:=\sigma _{b,\varepsilon }(n)\times \gamma _{n}^{-\varepsilon
}$ is bounded. We have the recurrence relation for $n\in {\mathbb{N}}$ (note
that $v_{0}=0$) 
\begin{equation*}
v_{n+1}=\theta _{n}v_{n}+\gamma _{n+1}\quad \text{ with }\quad \theta _{n}=%
\frac{\gamma _{n}^{\varepsilon }}{\gamma _{n+1}^{\varepsilon }}\times
e^{-b\gamma _{n+1}}.
\end{equation*}%
Using (\ref{Ap1}), we get $v_{n+1}\leq e^{(\varepsilon (\varpi (\pi )+\zeta
)-b)\gamma _{n+1}}v_{n}+\gamma _{n+1}$ for $n\geq n_{0}$ and thus 
\begin{equation*}
e^{(b-\varepsilon (\varpi (\pi )+\zeta ))t_{n+1}}v_{n+1}\leq
e^{(b-\varepsilon (\varpi (\pi )+\zeta ))t_{n}}v_{n}+e^{(b-\varepsilon
(\varpi (\pi )+\zeta ))t_{n+1}}\gamma _{n+1}.
\end{equation*}%
Using recursively this inequality, we obtain for $n>n_{0}$ 
\begin{align*}
&e^{(b-\varepsilon (\varpi (\pi )+\zeta ))t_{n}}v_{n} \leq e^{(b-\varepsilon
(\varpi (\pi )+\zeta
))t_{n_{0}}}v_{n_{0}}+\sum_{i=n_{0}}^{n-1}e^{(b-\varepsilon \varpi (\pi
)+\zeta )t_{i+1}}\times \gamma _{i+1} \\
&\leq e^{(b-\varepsilon (\varpi (\pi )+\zeta
))t_{n_{0}}}v_{n_{0}}+e^{(b-\varepsilon (\varpi (\pi )+\zeta
))}\sum_{i=n_{0}}^{n-1}e^{(b-\varepsilon (\varpi (\pi )+\zeta ))t_{i}}\times
\gamma _{i+1} \\
&\leq e^{(b-\varepsilon (\varpi (\pi )+\zeta
))t_{n_{0}}}v_{n_{0}}+e^{(b-\varepsilon (\varpi (\pi )+\zeta
))}\int_{t_{n_{0}}}^{t_{n}}e^{(b-\varepsilon (\varpi (\pi )+\zeta ))s}ds \\
&=e^{(b-\varepsilon (\varpi (\pi )+\zeta
))t_{n_{0}}}v_{n_{0}}+e^{(b-\varepsilon (\varpi (\pi )+\zeta ))}\times \frac{%
e^{(b-\varepsilon (\varpi (\pi )+\zeta ))t_{n}}-e^{(b-\varepsilon (\varpi
(\pi )+\zeta ))t_{n_{0}}}}{b-\varepsilon (\varpi (\pi )+\zeta )},
\end{align*}%
where we have used that $b-\varepsilon (\varpi (\pi )+\zeta )>0$ and $\gamma
_{i+1}\in (0,1)$ for $i\geq 0$. We finally get 
\begin{equation*}
v_{n}\leq v_{n_{0}}+\frac{e^{(b-\varepsilon (\varpi (\pi )+\zeta ))}}{%
b-\varepsilon (\varpi (\pi )+\zeta )},
\end{equation*}%
and thus the boundedness of the sequence $(v_{n})_{n\in \N}.$

Suppose now $b\leq \varpi(\pi )\varepsilon$ and thus $\varpi(\pi )>0$. We take $\lambda <
\frac{b}{\varpi(\pi )}\leq \varepsilon .$ We have $b>\varpi(\pi )\lambda $ so that, in view of the previous step, \ $\sigma
_{b,\varepsilon }(n)\leq \sigma _{b,\lambda }(n)\leq C_{b,\lambda
} \gamma_n^{\lambda }.$ So \textbf{A} is proved.

\textbf{Proof of B. } We now have $\gamma _{n}=\frac{1}{%
n+h^{-1}}.$ Clearly $\varpi (\pi )=1$ and, if $b>\varepsilon $ we use the
result from the point {\bf A} directly.  Suppose now the $b\leq \varepsilon
.$ One also has 
\begin{equation*}
\ln \left(\frac{n+h^{-1}}{k+h^{-1}}\right) = \int_{k}^{n}\frac{dx}{x+h^{-1}}\leq
\sum_{i=k}^{n-1}\gamma _{i}=t_{n-1}-t_{k-1}
\end{equation*}%
which gives%
\begin{equation*}
e^{-b(t_{n}-t_{k})}\leq \frac{(k+1+h^{-1})^{b}}{(n+1+h^{-1})^{b}}
\end{equation*}%
and consequently%
\begin{align*}
\sigma _{b,\varepsilon }(n)=\sum_{k=1}^{n}e^{-b(t_{n}-t_{k})}\gamma
_{k}^{1+\varepsilon } &\le \frac{1}{(n+1+h^{-1})^{b}}\sum_{k=1}^{n}(k+1+h^{-1})^{b} \gamma
_{k}^{1+\varepsilon} \\& \le \frac{(1+h^{-1})^b}{(n+1+h^{-1})^{b}}\sum_{k=1}^{n}\gamma
_{k}^{1+\varepsilon -b}.
\end{align*}%
If $b<\varepsilon $ the above quantity is bounded by $Cn^{-b}$
and if $b=\varepsilon $, it is upper bounded by $Cn^{-b}\ln (n+1)$. This gives the first upper bound.

We have 
\begin{equation*}
t_{n}-s=\sum_{i=1}^{n}\frac{1}{i+h^{-1}}\leq \int_{h^{-1}}^{n+h^{-1}}\frac{dx%
}{x}=\ln (n+h^{-1})-\ln h^{-1}
\end{equation*}%
so that 
\begin{equation}
h e^{-(t_{n}-s)}\geq \frac{1}{n+h^{-1}}  \label{A3}
\end{equation}%
which also gives, for $n\geq h^{-1},$%
\begin{equation*}
\frac{1}{n}=\frac{1}{n+h^{-1}}\times \frac{n+h^{-1}}{n}\leq
2h e^{-(t_{n}-s)}\leq 2e h e^{-(t-s)},
\end{equation*}%
since $0\le t-t_n\le t_{n+1}-t_n\le 1$. 
Then, if $b\neq \varepsilon $ one has 
\begin{equation*}
\sigma _{b,\varepsilon }(n)\leq C_{b,\varepsilon }(\frac{2e}{h})^{b\wedge
\varepsilon }e^{-b\wedge \varepsilon (t-s)}.
\end{equation*}%
If $b=\varepsilon $ we have to multiply with $\ln (1+n)\leq C'
(1+(t-s))$.
$\square $

\section{Rates of convergence for the Ornstein-Uhlenbeck process}\label{App_OU}

	\blue{Let us consider the Ornstein-Uhlenbeck process~\eqref{OU_exactS} and its Euler-Maruyama approximation~\eqref{OU_euler}. The goal of this appendix is to prove the coupling properties.\\
	We first prove that $\Theta^1$ and $\Theta^2$ are coupled, and first show the estimate~\eqref{AN1} when $X^1=x^1$ and $X^2=x^2$ are deterministic. In this case, the optimal $W_p$ coupling between $X^1_{s,t}(x^1)$ and $X^1_{s,t}(x^2)$ is given by $(x^1 e^{-k(t-s)}+\sigma \sqrt{\frac{1-e^{-2k(t-s)}}{2k}}G, x^2 (1-k(t-s)) +\sigma \sqrt{t-s} G )$ where $G\sim \mathcal{N}(0,1)$. Therefore,
\begin{align*}
&W_p^p(\Theta_{s,t} ^{1}(\delta_{x^1}),\Theta_{s,t} ^{2}(\delta_{x^2}) )\\
& =\mathbb{E} \left[  \left| x^1 e^{-k(t-s)}- x^2 (1-k(t-s)) +\sigma \left( \sqrt{\frac{1-e^{-2k(t-s)}}{2k}} -\sqrt{t-s}\right)G   \right|^p\right].
\end{align*} 
We use the following inequality for $p\ge 2$, $a,b \in \R$,
$$|a+b|^p\le |a|^p +p |a|^{p-2} a b + C_p (|a|^{p-2} b^2 + |b|^p),$$
where $C_p$ is a real constant depending only on $p$, with $a=x^1 e^{-k(t-s)}- x^2 (1-k(t-s))$ and $b= \sigma \left( \sqrt{\frac{1-e^{-2k(t-s)}}{2k}} -\sqrt{t-s}\right)G $. We get
$$W_p^p(\Theta_{s,t} ^{1}(\delta_{x^1}),\Theta_{s,t} ^{2}(\delta_{x^2}) ) \le |a|^p +C'\left(|a|^{p-2} (t-s)^3 +(t-s)^{3p/2}  \right), $$
since $b=O((t-s)^{3/2})$ as $t-s \to 0$. We write $a=a'+b'$ with $a'=(x^1-x^2)e^{-k(t-s)}$ and $b'=x^2  (e^{-k(t-s)}-(1-k(t-s)))$ and apply the same inequality to get
$$|a|^p\le |x^1-x^2|^pe^{-kp (t-s)}+ \tilde{C}_p(|a'|^{p-1}|b'|+|b'|^p).$$
Since $|b'| \le C |x^2|(t-s)^2$, we get by Young's inequality
$$|a'|^{p-1}|b'| \le \frac{(p-1)|a'|^p}{p}(t-s)^{\eta' \frac{p}{p-1} } + p |x^2|^p (t-s)^{(2-\eta')p}. $$
We take $\eta'$ such that $\eta' \frac{p}{p-1}>1$, i.e. $\eta'> \frac{p-1}{p}$, and get
$$|a|^p \le |x^1-x^2|^p(1-b_*(t-s)) + C(1+|x^2|^p)(t-s)^{(2-\eta')p},$$ 
for $t-s$ small enough and $b_*\in(0,pk)$, using that $e^{-pk (t-s)}\le (1- b_*(t-s))$ for $t-s$ sufficiently small. Note that $(2-\eta')p<p+1$.
In a similar fashion,  we have by Young's inequality for $\eta>0$ (this step is useless for the case $p=2$), 
$$|a|^{p-2} (t-s)^\eta (t-s)^{3-\eta}\le \frac{(p-2)|a|^p}{p}(t-s)^{\frac{p}{p-2}\eta} + \frac{p}{2} (t-s)^{(3-\eta) \frac {p}{2}},$$
and thus 
$$W_p^p(\Theta_{s,t} ^{1}(\delta_{x^1}),\Theta_{s,t} ^{2}(\delta_{x^2}) ) \le |a|^p (1+C_1 (t-s)^{\frac{p}{p-2}\eta}) + C_2(t-s)^{(3-\eta)p/2}. $$
We chose $\eta$ to have $\frac{p}{p-2}\eta>1$ so that the term of order~$1$ in $(t-s)$ of $|a|^p$ remains unchanged by the multiplication by $(1+C_1 (t-s)^{\frac{p}{p-2}\eta})$. Note that $(3-\eta)p/2<p+1$. Taking $\eta$ (resp. $\eta'$) arbitrarily close to $\frac{p-2}{p}$ (resp. $\frac{p-1}{p}$), we finally get, for any $b_*\in(0,kp)$ and $\varepsilon<p$, the existence of $h>0$ that do not depend on $(x_1,x_2)
$ such that for $t-s \in (0,h)$, $$ W_p^p(\Theta_{s,t} ^{1}(\delta_{x^1}),\Theta_{s,t}^{2}(\delta_{x^2}) )\le  |x^1-x^2|^p(1-b_*(t-s)) + C (1+|x^2|^p) (t-s)^{1+\varepsilon}.   $$
 Now, if we take $\mu^1,\mu^2 \in \mathcal{P}_p(\mathbb{R})$, $\pi$ an optimal $W_p$-coupling between $\mu^1$ and $\mu^2$, we have $$W_p^p(\Theta_{s,t} ^{1}(\mu^1),\Theta_{s,t} ^{2}(\mu^2) )\le \int_{\mathbb{R}^2}W_p^p(\Theta_{s,t} ^{1}(\delta_{x^1}),\Theta_{s,t} ^{2}(\delta_{x^2}) )  \pi(dx^1,dx^2), $$  and therefore
$$W_p^p(\Theta_{s,t} ^{1}(\mu^1),\Theta_{s,t} ^{2}(\mu^2) ) \le W_p^p( \mu^1, \mu^2) (1-b_*(t-s)) +C \Gamma_p(\mu^1,\mu^2)(t-s)^{1+\varepsilon}.  $$ 
Thus, $\Theta^1$ and $\Theta^2$ are $(p,b_*,\varepsilon)$ coupled for any $b_*\in (0,kp)$ and $\varepsilon \in (0,p)$. \\
Let us prove now that $\Theta^1$ is self coupled.     
  We have  $$X^1_{s,t}(X^1)=X^1e^{-k(t-s)}+\sigma \int_s^t e^{-k(t-u)} dW_u$$ and thus  $W_p^p(\theta_{s,t}(\mu^1),\theta_{s,t}(\mu^2))\le W_p^p(\mu^1,\mu^2)e^{-pk(t-s)}$. This inequality is an equality when $\mu^1$ and $\mu^2$ are Dirac masses, so we cannot get a better rate. Now, for any $0<b_*<k$, and $h$ sufficiently small, we have $e^{-pk(t-s)}\le (1-pb_*(t-s))$ for $t-s\in (0,h)$: we get that $\theta$ is $(p,pb_*,\varepsilon)$ self-coupled for any $\varepsilon>0$.
  Applying Theorem~\ref{th:main} to $\theta=\Theta^1$ and $\Theta=\Theta^2$ with $\gamma_n=\frac{1}{k+n}$, we get\begin{equation}\label{cv_OU_rate}W_p^p(\theta_{s,t}(\mu),\nu)\le A e^{-pb_*(t-s)},\end{equation}
\begin{equation}\label{cv_OU_example}W_{p}^{p}(\Theta _{s,t_{n}}^{\pi }(\mu ),\nu )\leq A n^{- \zeta p  },  \end{equation}
for any $\zeta< k \wedge 1$, where $t_n=s +\sum_{i=1}^n \frac {1}{k+i}$.\\
  We now discuss the sharpness of these rates of convergence. 
We recall that for $x_0\in \R$, $\sigma_1,\sigma_2>0$,  $(x_0+\sigma_1 G,\sigma_2 G)$ with $G\sim \mathcal{N}(0,1)$ is an optimal coupling for $W_p$ between $\mathcal{N}(x_0,\sigma_1^2)$ and $\mathcal{N}(0,\sigma_2^2)$. As $\nu\sim  \mathcal{N}(0,\frac {\sigma^2} {2k})$ is the invariant measure, we get  using Jensen's inequality
$$W_{p}^{p}(\theta_{s,t}(\delta_{x_0} ), \nu)\ge |x_0|^p e^{-pk(t-s)}, $$ 
so that~\eqref{cv_OU_rate} is almost optimal. \\
We now turn to the second inequality~\eqref{cv_OU_example}.  By elementary calculations, we have  (here, we interpret $\prod_{i=n+1}^n\equiv 1$) $$\Theta_{s,t_{n}}^{\pi }(\delta_{x_0} )=\mathcal{N}\left(x_0 \prod_{i=1}^n(1-k\gamma_i), \sigma^2 \sum_{i=1}^n \prod_{\ell=i+1}^n (1-k\gamma_\ell )^2 \gamma_i \right).$$ Then repeating the arguments for Gaussian random variables described above we have
$$W_{p}^{p}( \Theta_{s,t_{n}} ^{\pi} (\delta_{x_0} ) ,\nu ) \ge |x_0|^p \prod_{i=1}^n(1-k \gamma_i)^p .$$ 		
We have $1-k\gamma_i=\frac{i}{i+k}$ and $\prod_{i=1}^n(i+k)=\frac{\Gamma(n+1+k)}{\Gamma(1+k)}$, which gives
$$\prod_{i=1}^n(1-k \gamma_i)^p =\left( \frac{\Gamma(1+k) \Gamma(n+1)}{\Gamma(n+1+k)} \right)^p \sim c n^{-kp},$$
for some $c>0$.
We check also easily that $W_{p}^{p}( \Theta_{s,t_{n}} ^{\pi} (\delta_{x_0} ) ,\nu ) \ge W_{p}^{p}( \Theta_{s,t_{n}} ^{\pi} (\delta_{0} ) ,\nu )$ and, by some calculations, 
$$W_{p}^{p}( \Theta_{s,t_{n}} ^{\pi} (\delta_{0} ) ,\nu ) =\sigma^p \mathbb{E}[|G|^p] \left|\sqrt{\sum_{i=1}^n \prod_{\ell=i+1}^n (1-k\gamma_\ell )^2 \gamma_i } -\sqrt{\frac{1}{2k} } \right|^p \sim c' n^{-p}, $$
for some constant $c'>0$. Therefore, taking into account the above two lower estimates, we obtain $$ A_1 n^{-(k\wedge 1)p}\le W_{p}^{p}( \Theta_{s,t_{n}} ^{\pi} (\delta_{x_0} ) ,\nu ) \le A_2 n^{-(k\wedge 1)p}. $$
The upper bound follows by the triangle inequality, $$W_{p}( \Theta_{s,t_{n}} ^{\pi} (\delta_{x_0} ) ,\nu )\le W_{p}( \Theta_{s,t_{n}} ^{\pi} (\delta_{x_0} ) ,\Theta_{s,t_{n}} ^{\pi} (\delta_{0} ) ) +W_{p}( \Theta_{s,t_{n}} ^{\pi} (\delta_{0} ) ,\nu )$$ and $W_{p}^p( \Theta_{s,t_{n}} ^{\pi} (\delta_{x_0} ) ,\Theta_{s,t_{n}} ^{\pi} (\delta_{0} ) )= |x_0|^p \prod_{i=1}^n(1-k \gamma_i)^p $. As a consequence, \eqref{cv_OU_example} is almost sharp. This shows that Theorem~\ref{th:main} may lead to sharp estimates. Of course, the difficulty is however  to show the coupling properties with the best constants $b_*$ and $\varepsilon$.}

\bibliographystyle{alpha}
\bibliography{biblio}

\newpage
	
	\section{Additional technical proofs for some applications}\label{App_additional}

\subsection{A stochastic equation with reflection in $ W_2 $: Proof of Theorem~\ref{thm_Reflected2}}

\label{App_refl}

We check the conditions in Theorem \ref{stochastic} in two situations: on
the one hand to get that $\mathcal{X}^1 $ is self-coupled, and on the other
hand to obtain that $\mathcal{X}^1 $ is coupled with $\mathcal{X}^2 $, where 
$\mathcal{X}^1 $ and $\mathcal{X}^2 $ are defined respectively by~%
\eqref{eq:ref1} and~\eqref{eq:refX2}. In the present case, the coupling is the simplest one, which
is given by the strong solution of the corresponding reflected equation
driven by the same Brownian motion.

Furthermore, note that from the contraction property \eqref{1d}, we obtain
that for every $\delta>0 $, 
\begin{align}
xb(x)\leq& (-\bar{b}+\delta)|x|^2+\delta^{-1}b(0)^2. \label{a1}
\end{align}
First, we give a lemma with some moment estimates.

\begin{lemma}
\label{lem:3.2}Assume that $2\bar{b}>C_\sigma $, then there exists a
positive constant $C $ that depends on $C_b $, $C_\sigma $, $ b(0) $ and $ \sigma(0) $ such that for 
$s\le t\le s+1$ 
\begin{align}
{\mathbb{E}}[(\mathcal{X}_{s,t}^{i}(X^{i})-X^{i})^{2}]\leq& C(1+{\mathbb{E}}%
[(X^{i})^{2}])(t-s),  \label{1a} \\
\left\vert {\mathbb{E}}[\mathcal{X}_{s,t}^{i}(X^{i})^{2}]-{\mathbb{E}}%
[(X^{i})^{2}]\right\vert \leq& C(1+{\mathbb{E}}[(X^{i})^{2}])(t-s).
\label{1c}
\end{align}
Besides, (\ref{AN2BIS}) is satisfied with  $p=2 $, $%
\varepsilon^{\prime }=1$ for $\mathcal{X}_{s,t}^{i}(X^{i}) $
 and  $\mathcal{X}_{s,t}^{1}(X^{i}) $, $i=1,2 $ .
\end{lemma}

\begin{proof}
Using It\^o's formula and the fact that $\mathcal{X}%
_{s,u}^{i}(X^{i})dL^i_{s,u}=0$ for $i\in \{1,2\}$, we obtain 
\begin{align}  \label{mom2_X1}
{\mathbb{E}}[\mathcal{X}_{s,t}^{1}(X^{1})^{2}] &{=}{\mathbb{E}}%
[(X^{1})^{2}]+ 2\int_{s}^{t}
{\mathbb{E}}[\mathcal{X} _{s,u}^{1}(X^{1})b(%
\mathcal{X}_{s,u}^{1}(X^{1}))]du+\int_{s}^{t}{\mathbb{E}}[\sigma ( \mathcal{X%
}_{s,u}^{1}(X^{1}))^{2}]du \\
{\mathbb{E}}[\mathcal{X}_{s,t}^{2}(X^{2})^{2}] &{=}{\mathbb{E}}%
[(X^{2})^{2}]+ 2\int_{s}^{t}{\mathbb{E}}[\mathcal{X}
_{s,u}^{2}(X^{2})b(X^{2})]du+\int_{s}^{t}{\mathbb{E}}[\sigma ( X^{2})^{2}]du.
\end{align}
Using the inequality $2ab\leq a^2+b^2 $ then the Lipschitz properties in  \eqref{LIP_example1} together with 
Gronwall inequality, we get that there exists $C$ depending on $C_b$, $%
C_\sigma$, $|b(0)|$ and $|\sigma(0)|$ such that 
\begin{equation}  \label{mom2}
{\mathbb{E}}[\mathcal{X}_{s,t}^{i}(X^{i})^{2}] \le C(1+{\mathbb{E}}%
[(X^i)^2]) , \ s\le t \le s+1, \ i \in \{1,2 \}.
\end{equation}

In the same way, we get from It\^o's formula, the inequality $2ab\leq
a^2+b^2 $, $(\mathcal{X}_{s,t}^{i}(X^{i})-X^{i})dL_{s,t}^{i}=
-X^{i}dL_{s,t}^{i}\leq 0$ and~\eqref{mom2}: 
\begin{equation*}
{\mathbb{E}}[(\mathcal{X}_{s,t}^{i}(X^{i})-X^{i})^{2}]{\leq }\int_{s}^{t}{%
\mathbb{E}}[(\mathcal{X} _{s,u}^{i}(X^{i})-X^{i})^{2}]du+ C{\mathbb{E}}%
[1+|X^{i}|^{2}](t-s),
\end{equation*}
where $C $ depends on $C_b$, $C_\sigma$, $|b(0)|$ and $|\sigma(0)|$.
Applying Gronwall's inequality one obtains the preliminary estimate %
\eqref{1a}. In a similar fashion, we get~\eqref{1c}.
\end{proof}
\begin{lemma}
	$ \theta $ and $ \Theta $ satisfy the $ 2 $-Foster-Lyapunov condition.
\end{lemma}
\begin{proof}
We will prove the property (\ref{AN2BIS}) which is equivalent to \eqref{AN2}. From~\eqref{mom2_X1}, \eqref{a1} and 
$\sigma(x)^2\le (1+\delta)C_\sigma |x|^2+(1+\delta^{-1})\sigma(0)^2$, we get 
\begin{align*}
{\mathbb{E}}[\mathcal{X}_{s,t}^{1}(X^{1})^{2}] \leq& {\mathbb{E}}%
[(X^{1})^{2}]-(2(\bar{b}-\delta) -C_\sigma(1+\delta))\int_{s}^{t}{\mathbb{E}}%
[\mathcal{X}_{s,u}^{1}(X^{1})^{2}]du \\
& +\left(\delta^{-1}( \sigma ^{2}(0)+2b^{2}(0))+\sigma ^{2}(0)\right)(t-s) \\
{=}&{\mathbb{E}}[(X^{1})^{2}](1-(2\bar{b}-C_\sigma
-(C_\sigma+2)\delta)(t-s)) \\
&-(2\bar{b}-C_\sigma -(C_\sigma+2)\delta)\int_{s}^{t}({\mathbb{E}}[\mathcal{%
	X} _{s,u}^{1}(X^{1})^{2}]-{\mathbb{E}}[(X^{1})^{2}])du \\
&+\left(\delta^{-1}( \sigma ^{2}(0)+2b^{2}(0))+\sigma ^{2}(0)\right)(t-s) \\
\leq& {\mathbb{E}}[(X^{1})^{2}](1-(2\bar{b}-C_\sigma
-(C_\sigma+2)\delta)(t-s))\\
&+C(1+{\mathbb{E}}[(X^{1})^{2}])(t-s)^{2}+\left(\delta^{-1}( \sigma ^{2}(0)+2b^{2}(0))+\sigma ^{2}(0)\right)(t-s),
	\end{align*}
by using~\eqref{1c} for the last inequality. Now take  $\delta>0$ sufficiently
small so that $2\bar{b}-C_\sigma -(C_\sigma+2)\delta>0$. This is precisely (\ref{AN2BIS}) for $p=2 $ applied to $ \theta $ with 
 $\varepsilon=1$. 

A small variation of the above argument gives the result for ${\mathbb{E}}[%
\mathcal{X}_{s,t}^{2}(X^{2})^{2}] $. In fact, using It\^{o}'s formula and $%
\mathcal{X}_{s,u}^{2}(X^{2})dL^2_{s,u}=0$, we obtain using \eqref{a1}
\begin{align*}
{\mathbb{E}}[\mathcal{X}_{s,t}^{2}(X^{2})^{2}]{=}& {\mathbb{E}}%
[(X^{2})^{2}]+2\int_{s}^{t}{\mathbb{E}}[(\mathcal{X }%
_{s,u}^{2}(X^{2})-X^{2})b(X^{2})]du +2{\mathbb{E}}[X^{2}b(X^{2})](t-s) \\
& +(t-s){\mathbb{E}}[\sigma (X^{2})^{2}] \\
\leq & {\mathbb{E}}[(X^{2})^{2}] (1-2(\bar{b}-\delta )(t-s)) + (t-s){\mathbb{%
E}}[\sigma (X^{2})^{2}+\delta b (X^{2})^{2}] \\
&+\delta^{-1} \int_{s}^{t}{\mathbb{E}}[(\mathcal{X }%
_{s,u}^{2}(X^{2})-X^{2})^2]du +\delta^{-1}b(0)^2(t-s) \\
\leq & {\mathbb{E}}[(X^{2})^{2}] (1-(2\bar{b}-(2+2C_b+C_\sigma)\delta -C_\sigma
)(t-s))\\&+\delta^{-1} C (t-s)^{2} (1+{\mathbb{E}}[(X^{2})^{2}]),
\end{align*}
by using \eqref{1a} and the Lipschitz property of $b$ and $\sigma$.
Taking $\delta>0$ sufficiently small, we thus get (\ref{AN2BIS}) with $%
\varepsilon=1$.
\end{proof}

The next lemma shows that the condition \eqref{AN3} is satisfied.

\begin{lemma}
\label{lem:rfl2} Suppose that $2\bar{b}-C_\sigma>0 $ then \eqref{AN3} is
satisfied for $ p=2 $, $\varepsilon=1 $ and any $b_*\in(0,2\bar{b}%
-C_\sigma)$ for the following two pairs   $(\mathcal{X}_{s,t}^{1}(X^{1}) ,\mathcal{X}_{s,t}^{2}(X^{2}) )$ and $(\mathcal{X}%
_{s,t}^{1}(X^{1}),\mathcal{X}%
_{s,t}^{1}(X^{2})) $. The constant $C_* $ in \eqref{AN3} depends on $\bar{b} $, $C_b $, $C_\sigma $, 
$b(0) $ and $\sigma(0) $. 
\end{lemma}

\begin{proof}
We proceed in a similar fashion as in the proof of Lemma \ref{lem:3.2}.
Define $\mathcal{Y}:= \mathcal{X} ^{1}(X^{1})-\mathcal{X} ^{2}(X^{2})$ which
satisfies, for $s\le t$, 
\begin{align*}
\mathcal{Y}_{s,t}=&X^1-X^2+\int_s^t \left(b(\mathcal{X}%
_{s,r}^{1}(X^{1}))-b(X^2)\right) dr +\int_s^t\left(\sigma(\mathcal{X}_{s,r}^{1}(X^{1}))-\sigma(X^2)\right)dW_r\\
&
+L^1_{s,t}-L^2_{s,t}.
\end{align*}
First, we prove the following preliminary estimates: 
\begin{align}
\left\vert {\mathbb{E}}[\mathcal{Y}_{s,t}^{2}]-{\mathbb{E}}[(X^{1}-X^{2})^{2}]%
\right\vert &\leq C(1+{\mathbb{E}}[(X^{1})^{2}]+{\mathbb{E}}%
[(X^{2})^{2}])(t-s),  \label{1e} \\
\left\vert {\mathbb{E}}[(\mathcal{X}_{s,t}^{1}(X^{1})-\mathcal{X}%
_{s,t}^{1}(X^{2}))^{2}]-{\mathbb{E}}[(X^{1}-X^{2})^{2}]\right\vert &\leq C(1+{%
\mathbb{E}}[(X^{1})^{2}]+{\mathbb{E}}[(X^{2})^{2}])(t-s).  \label{1f}
\end{align}
where the constant $C $ depends on $C_b $ and $C_\sigma $ as well as $b(0) $
and $\sigma(0) $. We give the argument for \eqref{1e}, leaving \eqref{1f}
for the reader. Using It\^o's formula and $(\mathcal{X} _{s,u}^{1}(X^{1})-%
\mathcal{X} _{s,u}^{2}(X^{2}))(dL_{s,u}^{1}-dL_{s,u}^{2})\leq 0$, we obtain for $t-s\leq 1 $ 
\begin{align}
{\mathbb{E}}[(\mathcal{Y}_{s,t})^{2}]{\leq}& {\mathbb{E}}[(X^{1}-X^{2})^{2}]
+2\int_{s}^{t}{\mathbb{E}}[\mathcal{Y}_{s,u}(b(\mathcal{X}%
_{s,u}^{1}(X^{1}))-b(X^{2}))]du  \notag \\
& +\int_{s}^{t}{\mathbb{E}}[(\sigma (\mathcal{X}_{s,u}^{1}(X^{1}))-\sigma
(X^{2}))^{2}]du .  \label{eq:free}
\end{align}

Therefore, using the the inequality $%
2ab\leq a^{2}+b^{2}$ for the second term and then Lipschitz properties of $b $ and $\sigma $ as well as
bounds for the $L^2 $-moments of $\mathcal{X}^i(X^i) $, $i=1,2 $ in Lemma %
\ref{lem:3.2} , \eqref{1e} follows.

In order to prove \eqref{AN3}, we use again \eqref{eq:free} but we apply the inequality  $%
ab\leq \delta a^{2}+\delta ^{-1}b^{2}$, the Lipschitz property of $ b $ and the
following decomposition using~\eqref{1d} 
\begin{align*}
2{\mathbb{E}}[\mathcal{Y}_{s,u}&(b(\mathcal{X}_{s,u}^{1}(X^{1}))-b(X^{2}))]%
\leq -2\bar{b}{\mathbb{E}}[\mathcal{Y}_{s,u}^2]+2{\mathbb{E}}[\mathcal{Y}%
_{s,u}(b(\mathcal{X}_{s,u}^{2}(X^{2}))-b(X^{2}))] \\
\leq& -(2\bar{b}-\delta){\mathbb{E}}[\mathcal{Y}_{s,u}^2]+C_b\delta^{-1} {%
\mathbb{E}}[(\mathcal{X}_{s,u}^{2}(X^{2})-X^{2})^2] \\
\leq & -(2\bar{b}-\delta){\mathbb{E}}[(X^1-X^2)^2]+C(\delta+\delta^{-1})(1+{%
\mathbb{E}}[(X^{1})^{2}]+{\mathbb{E}}[(X^{2})^{2}])(u-s).
\end{align*}
In the last inequality, we have used \eqref{1e} for the first term, and
Lemma \ref{lem:3.2} for the second one. Similarly, using the inequality $%
ab\leq \delta a^{2}+\delta ^{-1}b^{2}$ with \linebreak ${\mathbb{E}}[(\sigma (\mathcal{X%
}_{s,u}^{2}(X^{2}))-\sigma (X^{2}))^2]\le C(1+{\mathbb{E}}[(X^{2})^{2}])(u-s)$
by~\eqref{1a}, we get 
\begin{align*}
{\mathbb{E}}[(\sigma (\mathcal{X}_{s,u}^{1}(X^{1}))-\sigma (X^{2}))^{2}]\leq
&C_\sigma(1+\delta){\mathbb{E}}[\mathcal{Y}^2_{s,u}]+(1+\delta^{-1})C(1+{%
\mathbb{E}}[(X^{2})^{2}])(u-s) \\
\leq &C_\sigma(1+\delta){\mathbb{E}}[(X^1-X^2)^2] \\
&+(1+\delta^{-1})C(1+{\mathbb{E}}[(X^{1})^{2}]+{\mathbb{E}}[(X^{2})^{2}])(u-s).
\end{align*}
Therefore putting the above estimates together and taking $\delta>0$
sufficiently small, we obtain~\eqref{AN3} for $\varepsilon=1$ and any $%
b_*\in(0,2 \bar{b}-C_\sigma)$. Similar arguments give~\eqref{AN3} for $%
\mathcal{X}_{s,t}^1(X^i)$, $i=1,2$. We conclude by applying Theorem~\ref{FINAL}: since $b_*$ can take any value in $(0,2\bar{b}-C_\sigma)$, we  take $b_*\in(1,2\bar{b}-C_\sigma)$ if $2\bar{b}-C_\sigma>1$ to get the convergence with rate one and $\beta=b_*$ otherwise.
\end{proof}

\subsection{A Boltzmann equation in the martingale regime ($W_2$)}
\label{App_Boltz}



We now consider the following $d$-dimensional stochastic equation for $X^1\in \mathcal{%
P} _2(\mathbb{R}^d)$ 
\begin{equation}  \label{eq_Boltzmann_X1}
X_{s,t}^{1}(X^{1})=X^{1}+\int_{s}^{t}b(X_{s,r}^{1}(X^{1}))dr+\int_{s}^{t}
\int_{E}\int_{\mathbb{R}^{d}}c(v,z,X_{s,r-}^{1}(X^{1}))d\widetilde{N}_{%
\mathcal{L} (X_{s,r}^{1})}(v,z,r)
\end{equation}
where $b:{\mathbb{R}}^d\to {\mathbb{R}}^d$, $c:{\mathbb{R}}^d \times E\times 
{\mathbb{R}}^d  \to {\mathbb{R}}^d$ and $%
dN_{\mathcal{L}(X_{s,r}^{1})}(v,z,r)$ is a Poisson point measure on $\mathbb{%
R}^{d}\times E\times \mathbb{R}_{+}$ independent of~$X_1$ with compensator \linebreak $d%
\widehat{N}_{\mathcal{L} (X_{s,r}^{1})}(v,z,r)=\mathcal{L}%
(X_{s,r}^{1})(dv)\Lambda (dz)dr$ and $\widetilde{N}_{\mathcal{L}%
(X_{s,r}^{1})}=N_{\mathcal{L}(X_{s,r}^{1})}- \widehat{N}_{\mathcal{L}%
(X_{s,r}^{1})}.$ Here again, $\mathcal{L} (X_{s,r}^{1})$ denotes the law of
the random variable $X_{s,r}^{1} $ and $E$ is a measure space with measure $%
\Lambda $.

We associate to the above equation the following one step Euler scheme for $%
X^2\in \mathcal{P}_2(\mathbb{R}^d)$ independent of $N_{\mathcal{L}%
(X_{s,r}^{1})}$. 
\begin{equation}  \label{eq_Boltzmann_X2}
X_{s,t}^{2}(X^{2})=X^{2}+\int_{s}^{t}b(X^{2})dr+\int_{s}^{t} \int_{E}\int_{%
\mathbb{R}^{d}}c(v,z,X^{2})d\widetilde{N}_{\mathcal{L}(X^{2})}(v,z,r).
\end{equation}
where $dN_{\mathcal{L}(X^{2})}(v,z,r)$ is a Poisson point measure
independent of $X^2$ with compensator $d\widehat{N}_{\mathcal{L}%
(X^{2})}(v,z,r)=\mathcal{L} (X^{2})(dv)\Lambda (dz)dr$ and $\widetilde{N}_{%
\mathcal{L}(X^{2})}=N_{\mathcal{L} (X^{2})}-\widehat{N}_{\mathcal{L}%
(X^{2})}. $

Moreover, we define the maps for $\mathcal{L}(X)=\mu\in \mathcal{P}_2(%
\mathbb{R}^d) $ 
\begin{equation*}
\theta _{s,t}(\mu )=\mathcal{L}(X_{s,t}^{1}(X)) ,\quad \Theta _{s,t}(\mu )= 
\mathcal{L}(X_{s,t}^{2}(X)).
\end{equation*}
We will work under the following hypotheses: 
\begin{align}
\left\langle x-y,b(x)-b(y)\right\rangle \leq & -\overline{b}\left\vert
x-y\right\vert ^{2}\quad and\quad \left\vert b(x)-b(y)\right\vert ^{2}\leq
C_b\left\vert x-y\right\vert ^{2}  \label{VL4} \\
\left\vert c(v,z,x)-c(v^{\prime },z,x^{\prime })\right\vert ^{2}\leq & 
\overline{c}^{2}(z)(\left\vert v-v^{\prime }\right\vert ^{2}+\left\vert
x-x^{\prime }\right\vert ^{2})\quad and  \label{VL5} \\
\left\vert c(0,z,0)\right\vert ^{2} \leq& \overline{c}^{2}(z)\quad with \\
Q:=& \int_{E}\overline{c}^{2}(z)\Lambda (dz)<\infty .  \label{Def_Q}
\end{align}
Under these hypotheses, \cite[Theorem 3.1]{ABC} gives the existence of a
solution to Equation~\eqref{eq_Boltzmann_X1} and $\theta_{s,t}(\mu)$ is
uniquely defined. Furthermore, it is also proven that
\begin{equation}  \label{VL5'}
W_2(\theta_{s,t}(\mu),\Theta_{t^n_n,t^n_{n-1}}\circ\dots
\circ\Theta_{t^n_0,t^n_1}(\mu) )\to_{n\to \infty} 0
\end{equation}
for any sequence $s=t^n_0<t^n_1<\dots<t^n_n=t$ such that $\max_{1\le i\le n}
(t^n_i -t^n_{i-1}) \to_{n\to \infty} 0$.

\begin{theorem}
\label{thm_Boltz_W2} Assume that (\ref{VL4}--\ref{Def_Q}) hold with $Q<%
\overline{b}$. Then, $\theta $ and $\Theta $ satisfy the $2$-Foster-Lyapunov
condition, are $(2,b_{\ast },1/2)$ coupled for every $b_{\ast }<2(\bar{b}-Q)$%
, and $\theta $ is self-coupled with the same parameters.

Furthermore, the Markov process $(X_{s,t}^{1}(X))_{t\geq s}$
admits a unique invariant measure $\nu \in {\mathcal{P}}_2({\mathbb{R}}^d)$,
and for any $\beta \in (0,2(\bar{b}-Q)) \cap (0,1/2]$ there exists $C\in {%
\mathbb{R}}_+$ such that for every $\mu \in {\mathcal{P}}_2({\mathbb{R}})$, $%
n\ge 1$ and $t>s$, 
\begin{equation*}
W_{2}^{2}(\Theta _{s,t_{n}}^{\pi }(\mu ),\nu )\leq C(\left\Vert \mu
\right\Vert _{2}^{2}+1) \frac{1}{n^{\beta}},\ W_{2}^{2}(\theta
_{s,t}(\mu ),\nu )\leq C(\left\Vert \mu \right\Vert _{2}^{2}+1)\times
e^{-\beta (t-s)},
\end{equation*}
where $t_k=s+\sum_{i=1}^{k}\frac{1}{1+i}$ and $\pi=\{t_0<\dots<t_n\}$.
\end{theorem}

This appendix section is devoted to the proof of this theorem, which follows the same steps as the one with $W_1$
estimates. However, in contrast with the $W_1$ case (see Theorem~\ref{thm_Boltz_W1}), we are not able to deduce the coupling property from the sewing lemma, see Eq.~\eqref{eq_sew}. Thus, we prove here directly the coupling property. To construct the coupling, we work with approximation schemes on a regular time grid and let then the time step go to zero to obtain the coupling property for the flow. Namely, for
each $m\in {\mathbb{N}}$ we consider the grid 
\begin{equation*}
s_{i}=s+\frac{i}{m}, \quad i\in {\mathbb{N}},
\end{equation*}
and we define 
\begin{equation*}
\Theta _{s,s_{i}}^{m}(\mu )=\Theta _{s_{i-1},s_{i}}\circ ....\circ \Theta
_{s_{0},s_{1}}(\mu ) \text{ and } \Theta _{s,t}^{m}(\mu )=
\Theta_{s_i,t}\circ\Theta _{s,s_{i}}^{m}(\mu ) \text{ for } t \in
(s_i,s_{i+1}).
\end{equation*}

In order to prove Theorem~\ref{thm_Boltz_W2}, we use Theorem~\ref{stochastic}
so we have to produce a probabilistic representation and to prove that this
representation verifies (\ref{AN3}) and (\ref{AN2BIS}). For the moment we fix $%
m$, and for $s_{i-1}\leq u<s_{i}$ we denote $\eta (u)=s_{i-1}$.

\textbf{Step 1: construction of the coupling process (stochastic
representation) }for $\Theta _{s,s_{i}}^{m}(\mu )$ and $\Theta
_{s,s_{i}}(\mu ).$\ Our first step is to produce the joint representation
(coupling process) for them. Let $\mu ^{1},\mu ^{2}$ be fixed. For each $%
i\in {\mathbb{N}}$, we define $\Pi _{s,s_{i}}(dv_{1},dv_{2})$ as the optimal 
$W_{2}$-coupling of $\Theta _{s,s_{i}}^{m}(\mu ^{1})$ and $\mu ^{2},$ and we
take $\tau _{s,s_{i}}=(\tau _{s,s_{i}}^{1},\tau
_{s,s_{i}}^{2}):(0,1)\rightarrow \mathbb{R}^{d}\times \mathbb{R}^{d}$ such
that 
\begin{equation*}
\int_{\mathbb{R}^{d}\times \mathbb{R}^{d}}F(v_{1},v_{2})\Pi
_{s,s_{i}}(dv_{1},dv_{2})=\int_{0}^{1}F(\tau _{s,s_{i}}(w))dw.
\end{equation*}
The function $\tau$ is constructed such that $(s,w)\mapsto \tau_{s,s_i}(w)$
is measurable, see Carmona and Delarue~\cite[Lemma 5.29]{CD}.

We construct now a joint probabilistic representation (in the sense defined
in Section 1). We consider a filtered probability space $(\Omega ,\mathcal{F}
,\mathbb{P})$ that endows $dN(w,z,r)$ a Poisson point measure of intensity $%
dw\Lambda (dz)dr$. We denote by $(\mathcal{F}_{t})_{t\ge 0}$ the filtration
generated by this random measure and $d\tilde{N}$ the compensated measure.
Moreover, we consider random variables $X^{i}\in L^{2}(\Omega ),i=1,2$ which
are $\mathcal{F}_{s}$ measurable and such that $(X^{1},X^{2})$ has the law $%
\Pi _{s,s}(dv_{1},dv_{2}),$ which is the $W_{2}$ optimal coupling of $\mu
^{1}$ and $\mu ^{2}.$ In particular $X^{i}\sim \mu ^{i}$. Then, we define
the equations which give the coupling 
\begin{align}
	\notag
\!\!  x_{s,t}^{1}(X^{1}) &=\! X^{1}+\!\int_{s}^{t}\!\! b(x_{s,\eta
(r)}^{1}(X^{1}))dr+\!\int_{s}^{t}\!\!\int_{E}\int_{0}^{1}\!\!\! c(\tau _{s,\eta
(r)}^{1}(w),z,x_{s,\eta (r)}^{1}(X^{1}))d\widetilde{N}(w,z,r) 
\\\!\!  x_{s,t}^{2}(X^{2})& 
=\! X^{2}+\!\int_{s}^{t}\!\! b(X^{2})dr+\! \int_{s}^{t}\!\!\int_{E}\int_{0}^{1}\!\!\! c(\tau
_{s,\eta (r)}^{2}(w),z,X^{2})d\widetilde{N}(w,z,r).
\label{eq:x}
\end{align}
We remark that by induction all above integrals driven by $\widetilde{N}$
are square integrable martingales and therefore their expectation is zero.
This and similar arguments for other integrals driven by $\widetilde{N}$
will be frequently used in what follows.

\begin{lemma}
\label{law} For any $t\ge s$, $x_{s,t}^{1}(X^{1})$ has the law $\Theta
_{s,t}^{m}(\mu ^{1}),$ and $x_{s,t}^{2}(X^{2})$ has the law $%
\Theta_{s,t}(\mu ^{2})$.
\end{lemma}

\begin{proof}
Notice first that $\Theta _{s,t}^{m}(\mu ^{1})$ is the law of $%
X_{s,t}^{1}(X^{1})$ constructed recursively from~\eqref{eq_Boltzmann_X2} by 
\begin{align*}
X_{s,t}^{1}(X^{1}) =&X^{1}+\int_{s}^{t}b(x_{s,\eta
	(r)}^{1}(X^{1}))dr\\&+\sum_{s_{j}\leq t}\int_{s_{j-1}}^{s_{j}}\int_{E}\int_{%
	\mathbb{R}^{d}}c(v,z,X_{s,s_{j-1}}^{1}(X^{1}))d \widetilde{N}_{\rho
	_{r}}(v,z,r) \\
&+\int_{\eta (t)}^{t}\int_{E}\int_{\mathbb{R}^{d}}c(v,z,x_{s,\eta
	(t)}^{1}(X^{1}))d \widetilde{N}_{\rho _{r}}(v,z,r)
\end{align*}
with $\rho _{r}(dv)=\Theta _{s,s_{j-1}}^{m}(\mu ^{1})(dv)$ for $s_{j-1}\leq
r<s_{j}.$ Let
\begin{align*}
	I_{j}(x) &=\int_{s_{j-1}}^{s_{j}}\int_{E}\int_{0}^{1}c(\tau
	_{s,s_{j-1}}^{1}(w),z,x)d\widetilde{N}(w,z,r) ,\\
	I_{j}^{\prime }(x) &=\int_{s_{j-1}}^{s_{j}}\int_{E}\int_{\mathbb{R}%
		^{d}}c(v,z,x)d \widetilde{N}_{\rho _{r}}(v,z,r).
\end{align*} 
Using It\^{o}'s formula one easily checks that 
\begin{equation*}
{\mathbb{E}}(e^{i\left\langle h,I_{j}(x)\right\rangle
})=e^{(s_{j}-s_{j-1})\lambda _{h}(x)}\quad \text{and}\quad {\mathbb{E}}%
(e^{i\left\langle h,I_{j}^{\prime }(x)\right\rangle
})=e^{(s_{j}-s_{j-1})\lambda _{h}^{\prime }(x)}
\end{equation*}
with 
\begin{align*}
\lambda _{h}(x) &=\int_{0}^{1}\int_{E}\left(e^{i\left\langle h,c(\tau
	_{s,s_{j-1}}^{1}(w),z,x)\right\rangle }-1-i\langle h,c(\tau
_{s,s_{j-1}}^{1}(w),z,x)\rangle\right)\Lambda (dz)dw \\
\lambda _{h}^{\prime }(x) &=\int_{\mathbb{R}^{d}}\int_{E}\left(e^{i \langle h,c(v,z,x) \rangle } -1-i\langle h,c(v,z,x)\rangle \right)\Lambda
(dz)\Theta _{s,s_{j-1}}^{m}(\mu^{1})(dv).
\end{align*}
By the very definition of $\tau _{s,s_{j-1}}^{1}(w)$, the image of the
uniform distribution is \linebreak $\Theta _{s,s_{j-1}}^{m}(\mu^{1})$. Therefore, $%
\lambda _{h}(x)=\lambda _{h}^{\prime }(x)$ and the law of $I_{i}(x)$
coincides with the law of $I_{i}^{\prime }(x).$ Then, conditioning on $%
\mathcal{F}_{s_j}$, one proves by recurrence that $x_{s,t}^{1}(X^{1})$ has
law $\Theta _{s,t}^{m}(\mu ^{1})$ for $t\in [s_j,t_{j+1}]$ and thus for any $%
t\ge s$. In a similar way (but without recurrence) one checks that $%
x_{s,t}^{2}(X^{1})$ has law $\Theta _{s,t}(\mu ^{2})$.
\end{proof}

Both $x_{s,t}^{1}(X^{1})$ and $x_{s,t}^{2}(X^{2})$ are defined on the same
probability space - so they represent a joint probabilistic representation.
But we have to check the hypothesis concerning the coupling: see (\ref{AN5}%
), (\ref{AN3} ),(\ref{AN2BIS}). By the very construction (\ref{AN5}) is
verified.\newline

\noindent \textbf{Step 2: The Foster-Lyapunov condition (\ref{AN2BIS}). }

From the assumptions (\ref{VL4}) and (\ref{VL5}), we have for every $\delta
>0$ 
\begin{align}
	\left\langle x,b(x)\right\rangle &\leq (-\overline{b}+\delta )\left\vert
	x\right\vert ^{2}+\delta ^{-1}\left\vert b(0)\right\vert ^{2}  \label{VL1} \\
	\left\vert c(v,z,x)\right\vert ^{2} &\leq \overline{c}^{2}(z)(1+\left\vert
	v\right\vert ^{2}+\left\vert x\right\vert ^{2})\quad and\quad \left\vert
	b(x)\right\vert \leq K(1+\left\vert x\right\vert ).  \label{VL2}
\end{align}

\begin{lemma}
\label{lem_Boltzmann_W2} Under hypotheses~(\ref{VL4}) and (\ref{VL5}), we
have the following:

\textbf{A }There exists a constant $C$ (depending on $Q,\overline{b}$ and $K 
$, but not on the time step $1/m$) such that 
\begin{equation}
{\mathbb{E}}[\vert x_{s,t}^{i}(X^{i})-X^{i}\vert ^{2}]\leq C(1+{%
\mathbb{E}}[\vert X^{i}\vert ^{2}] )(t-s).  \label{VL3'}
\end{equation}

\textbf{B }If $\overline{b}>Q+\delta $\ then for every $s<t\leq s+1$ 
\begin{align}
{\mathbb{E}}[\vert x_{s,t}^{1}\vert ^{2}] \leq& {\mathbb{E}}%
[\vert X^{1}\vert ^{2}](1-2(\overline{b}-Q-\delta )(t-s))
\label{VL3'''} \\
&+(2\delta ^{-1}\left\vert b(0)\right\vert ^{2}+Q)(t-s)+C(1+{\mathbb{E}}
[\vert X^{1}\vert ^{2}])(t-s)^{3/2}.  \notag
\end{align}
\end{lemma}

Note that the result in \eqref{VL3'''} differs in rate from the one in Lemma %
\ref{lem:3.2}.

\begin{proof}
	Through the proof, we simplify the notation using $ (x_{s,t}^{1},x_{s,t}^{2})=(x_{s,t}^{1}(X^1),x_{s,t}^{2}(X^2)) $.
	
We start from the equation satisfied by $x_{s,t}^{1}$ in \eqref{eq:x}. Using the inequality $%
|x+y|^2\leq 2(|x|^2+|y|^2) $, Cauchy-Schwarz inequality, the $ L^2 $-isometry for martingale driven stochastic integrals, \eqref{Def_Q},
(\ref{VL1}) and (\ref{VL2}),  
 one
gets 
\begin{equation*}
{\mathbb{E}}[\vert x_{s,t}^{1}\vert ^{2}]\leq C\left(1+{\mathbb{E}%
}[\vert X^{1}\vert ^{2}]+\int_{s}^{t}{\mathbb{E}}\left[\vert
x_{s,\eta (r)}^{1}\vert ^{2}\right]dr\right).
\end{equation*}
Here, we have used that $\int_{0}^{1}|\tau _{s,\eta (r)}^{1}(w)|^{2}dw={%
\mathbb{E}}[\vert x_{s,\eta (r)}^{1}\vert ^{2}]$ and the constant $ C $ depends on $ Q $. This identity will be used
frequently without any further mention of it. Next, due to a Gronwall type
inequality we obtain 
\begin{equation}
{\mathbb{E}}[\vert x_{s,t}^{1}\vert ^{2}]\leq C(1+{\mathbb{E}}[\vert
X^{1}\vert ^{2}]).  \label{a}
\end{equation}

\textbf{Proof of A:} Let us prove (\ref{VL3'}). As before, note that using the inequality $%
|x+y|^2\leq 2(|x|^2+|y|^2) $, Cauchy-Schwarz inequality, the $ L^2 $-isometry for martingale driven stochastic integrals,~%
\eqref{VL2} and (\ref{a}) one obtains 
\begin{align*}
&{\mathbb{E}}\left[\vert x_{s,t}^{1}(X^{1})-X^{1}\vert ^{2}\right] \\
&\leq 2(t-s){\mathbb{E}}\int_{s}^{t}\vert b(x_{s,\eta (r)}^{1})\vert
^{2}dr+2{\mathbb{E}}\int_{s}^{t}\int_{E}\int_{0}^{1}\vert c(\tau _{s,\eta
	(r)}^{1}(w),z,x_{s,\eta (r)}^{1})\vert ^{2}dw \Lambda (dz)dr \\
&\leq 2K(t-s){\mathbb{E}}\int_{s}^{t}(1+\vert x_{s,\eta (r)}^{1}\vert
^{2})dr+2Q\int_{s}^{t}(1+2{\mathbb{E}}[\vert x_{s,\eta (r)}^{1}\vert ^{2}])dr
\\
&\leq C^{\prime }(t-s)(1+{\mathbb{E}}[\vert X^{1}\vert ^{2}]).
\end{align*}

The proof in the case $i=2$ is the same.

\textbf{Proof of B:} One employs (\ref{VL3'}) and (\ref{VL1}) to obtain 
\begin{align*}
{\mathbb{E}}\left[\langle x_{s,r}^{1},b(x_{s,\eta (r)}^{1})\rangle \right]
\leq& \left({\mathbb{E}}[|x_{s,r}^{1}-x_{s,\eta(r)}^{1}|^2]{\mathbb{E}}%
[|b(x_{s,\eta (r)}^{1})|^2]\right)^{1/2}\\&+{\mathbb{E}}\left[\langle x_{s,\eta
	(r)}^{1},b(x_{s,\eta (r)}^{1})\rangle \right]\\
\leq &C(r-\eta (r))^{1/2}(1+{\mathbb{E}}[\vert X^{1}\vert ^{2}])+(-%
\overline{b }+\delta ){\mathbb{E}}[\vert x_{s,\eta (r)}^{1}\vert
^{2}]\\&+\delta ^{-1}|b(0)|^{2}.
\end{align*}

Since $\int_{0}^{1}\left\vert \tau _{s,\eta (r)}^{1}(w)\right\vert ^{2}dw={%
\mathbb{E}} [(x_{s,\eta (r)}^{1})^{2}]$ and we have
assumed ( \ref{Def_Q}) we also have 
\begin{equation*}
{\mathbb{E}}\left[\int_{s}^{t}\!\!\int_{E}\int_{0}^{1}\!\! \overline{c}%
^{2}(z)(1+\vert \tau _{s,\eta (r)}^{1}(w)\vert ^{2}+\vert x_{s,\eta
(r)}^{1}\vert ^{2})dw \Lambda (dz)dr\! \right]\!\leq Q\!\!\int_{s}^{t}\!\!(1+2{\mathbb{E}}%
[\vert x_{s,\eta (r)}^{1}\vert ^{2}])dr.
\end{equation*}
Then, using It\^{o}'s formula for the function $x\rightarrow \left\vert
x\right\vert ^{2},$ and the identity $\left\vert x+c\right\vert
^{2}-\left\vert x\right\vert ^{2}-2\left\langle x,c\right\rangle =\left\vert
c\right\vert ^{2},$ we obtain 
\begin{align*}
{\mathbb{E}}\left[ \vert x_{s,t}^{1}\vert ^{2} \right]=&{\mathbb{E}}[\vert
X^{1}\vert ^{2}]+2{\mathbb{E}}\int_{s}^{t}\left\langle
x_{s,r}^{1},b(x_{s,\eta (r)}^{1})\right\rangle dr \\
&+{\mathbb{E}}\int_{s}^{t}\int_{E}\int_{0}^{1}\left\vert c(\tau _{s,\eta
(r)}^{1}(w),z,x_{s,\eta (r)}^{1})\right\vert ^{2}dw \Lambda (dz)dr.  \label{VL2'}
\\
\leq &{\mathbb{E}}[\vert X^{1}\vert ^{2}]+C(t-s)^{3/2}(1+{%
\mathbb{E}}[\vert X^{1}\vert ^{2}])-2(\overline{b}-Q-\delta ){\mathbb{E}}%
\int_{s}^{t}\vert x_{s,\eta (r)}^{1}\vert ^{2}dr  \notag \\
&+(2\delta ^{-1}\left\vert b(0)\right\vert ^{2}+Q)(t-s)  \notag
\end{align*}

We have $\vert{\mathbb{E}}[\vert x_{s,\eta (r)}^{1}\vert^{2}]-{\mathbb{E}}%
[\vert X^{1}\vert ^{2}]\vert\le \sqrt{{\mathbb{E}}[(\vert x_{s,\eta
(r)}^{1}\vert+\vert X^{1}\vert)^{2}]}\sqrt{{\mathbb{E}}
[\vert x_{s,\eta
(r)}^{1}- X^{1}\vert^{2}]}$, and we get by using \textbf{A} and~\eqref{a} 
\begin{equation*}
\int_{s}^{t}\left\vert {\mathbb{E}}[\vert x_{s,\eta (r)}^{1}\vert ^{2}]-{%
\mathbb{E}}[\vert X^{1}\vert ^{2}]\right\vert dr\leq C(1+{\mathbb{E}}%
[\vert X^{1}\vert ^{2}])(t-s)^{3/2}.
\end{equation*}
The above inequalities give 
\begin{align*}
	{\mathbb{E}}[\vert x_{s,t}^{1}\vert ^{2}] \leq& (1-2(\overline{b}-Q-\delta
	)(t-s)){\mathbb{E}}[\vert X^{1}\vert ^{2}] \\
	&+C(t-s)^{3/2}(1+{\mathbb{E}}[\vert X^{1}\vert ^{2}])+(2\delta
	^{-1}\left\vert b(0)\right\vert ^{2}+Q)(t-s).\qedhere
\end{align*}
\end{proof}

We conclude from Lemma~\ref{lem_Boltzmann_W2} that 
\begin{equation*}
\left\Vert \Theta _{s,t}^{m}(\mu )\right\Vert _{2}^{2}\leq (1-2(\overline{b}
-Q-\delta ))(t-s))\left\Vert \mu \right\Vert _{2}^{2}+C(1+\left\Vert \mu
\right\Vert _{2}^{2})(t-s)^{3/2}+(Q+2\delta ^{-1}\left\vert b(0)\right\vert
^{2})(t-s).
\end{equation*}
Taking limits as $m\rightarrow \infty $ we obtain 
\begin{equation*}
\left\Vert \theta _{s,t}(\mu )\right\Vert _{2}^{2}\leq (1-2(\overline{b}
-Q-\delta ))(t-s))\left\Vert \mu \right\Vert _{2}^{2}+C(1+\left\Vert \mu
\right\Vert _{2}^{2})(t-s)^{3/2}+(Q+2\delta ^{-1}\left\vert b(0)\right\vert
^{2})(t-s).
\end{equation*}
We remark that the difference in the value of $\varepsilon^{\prime }$
between \eqref{VL3'''} and the result in Lemma \ref{lem:3.2} is due to the
fact that the above proof uses the modulus of continuity of the underlying
process. This difference will also appear in the rest of the proof.\newline

\textbf{Step 3: Proof of the coupling property (\ref{AN3}).}

\begin{lemma}
\textbf{A }We assume that assumptions (\ref{VL4}) and (\ref{VL5}) hold
with $Q<\overline{b}.$ Then, for every $b_*<2(\overline{b}-Q)$ there exists
a constant $C$ such that, if $0\leq t-s\leq 1,$ then 
\begin{align}
	\notag
{\mathbb{E}}[\vert x_{s,t}^{1}(X^{1})-x_{s,t}^{2}(X^{2})\vert
^{2}]\leq &{\mathbb{E}}[\vert X^{1}-X^{2}\vert ^{2}](1-b_*(t-s))+C(1+{%
	\mathbb{E}}[\vert X^{1}\vert ^{2}]\\&+{\mathbb{E}}[\vert
X^{2}\vert ^{2}])(t-s)^{3/2}.  \label{VL6}
\end{align}
In particular, $\theta _{s,t}$ and $\Theta _{s,t}$ are $(2,b_*,1/2)$-coupled.

\textbf{B }Moreover $\theta $ is $(2,b_*,1)-$self-coupled for some $h_0>0 $
and more precisely 
\begin{equation}
W^2_{2}(\theta _{s,t}(\mu ^{1}),\theta _{s,t}(\mu ^{2}))\leq (1-b_*
(t-s))W^2_{2}(\mu ^{1},\mu ^{2}){ +C(1+\|\mu_1\|^2_2+\|\mu_2\|^2)(t-s)^{2}  }.  \label{VL7}
\end{equation}
\end{lemma}

\bigskip

\begin{remark}
 When dealing with $W_{2}(\theta _{s,t}(\mu ^{1}),\theta _{s,t}(\mu
^{2}))$ no error terms of the type ${\mathbb{E}[}\vert
b(x_{s,r}^{2}(X^{2}))-b(X^{2})\vert ^{2}]+{\mathbb{E}}[\vert
b(x_{s,\eta (r)}^{1})-b(x_{s,r}^{1})\vert ^{2}]$ appear in the proof. 
This is the reason why the term $C(1+{\mathbb{E}}[\vert X^{1}\vert
^{2}]+{\mathbb{E}}[\vert X^{2}\vert ^{2}])(t-s)^{3/2}$ does not
appear in the final result in comparison with \eqref{VL6}.
\end{remark}

\begin{proof}
	As in the previous proof we simplify the notation using \linebreak$ (x_{s,t}^{1},x_{s,t}^{2})=(x_{s,t}^{1}(X^1),x_{s,t}^{2}(X^2)) $.
	
\textbf{Proof of A. } The process $y_{s,t}=x_{s,t}^{1}-x_{s,t}^{2}$
satisfies 
\begin{equation*}
y_{s,t}=X^{1}-X^{2}+\int_{s}^{r}\Delta
b(r)dr+\int_{s}^{t}\int_{E}\int_{0}^{1}\Delta c(r,w,z)d\widetilde{N} (w,z,r),
\end{equation*}
with 
\begin{align*}
	\Delta b(r) &=b(x_{s,\eta (r)}^{1})-b(X^{2}) \\
	\Delta c(r,w,z) &=c(\tau _{s,\eta (r)}^{1}(w),z,x_{s,\eta
		(r)}^{1})-c(\tau _{s,\eta (r)}^{2}(w),z,X^{2}).
\end{align*}
From~\eqref{VL4} and using Lemma~\ref{lem_Boltzmann_W2}, we get for every $%
\delta >0$ 
\begin{align*}
{\mathbb{E}}[\left\langle y_{s,r},\Delta b(r)\right\rangle] \leq& -\overline{b%
} {\mathbb{E}}[\vert y_{s,r}\vert ^{2}]\\&+{\mathbb{E}}\left\langle
y_{s,r},b(x_{s,\eta (r)}^{1})-b(x_{s,r}^{1})\right\rangle +{\mathbb{E}}%
\left\langle y_{s,r},b(x_{s,r}^{2})-b(X^{2})\right\rangle \\
\leq& -\overline{b}{\mathbb{E}}[\vert y_{s,r}\vert ^{2}]+2\delta {\mathbb{E}%
}[\vert y_{s,r}\vert ^{2}] \\
&+{\delta^{-1} }\left({\mathbb{E}}\left[\vert
b(x_{s,r}^{2})-b(X^{2})\vert^{2}\right]+{\mathbb{E}}\left[\vert
b(x_{s,\eta (r)}^{1})-b(x_{s,r}^{1})\vert ^{2}\right]\right) \\
\leq& (2\delta -\overline{b}){\mathbb{E}}[\vert y_{s,r}\vert ^{2}]+%
{ \delta^{-1} }C_b\left({\mathbb{E}}\left[\vert
x_{s,r}^{2}-X^{2}\vert ^{2}\right]+{\mathbb{E}}[\vert x_{s,\eta
	(r)}^{1}-x_{s,r}^{1}\vert ^{2}]\right) \\
\leq& (2\delta -\overline{b}){\mathbb{E}}[\vert y_{s,r}\vert ^{2}]+%
{ \delta^{-1} }C_bC(1+{\mathbb{E}}[\vert X^{1}\vert^2] +{\mathbb{%
		E}}[\vert X^{2}\vert^2] )(r-s).
\end{align*}
Using the inequality $\vert \vert a\vert ^{2}-\vert
b\vert ^{2}\vert \leq \vert a-b\vert (\vert
a\vert +\vert b\vert )$, Cauchy-Schwarz inequality and (\ref%
{VL3'}), we have 
\begin{equation}
\vert {\mathbb{E}}[\vert y_{s,r}\vert ^{2}]-{\mathbb{E}}%
[\vert X^{1}-X^{2}\vert ^{2}]\vert \leq C(1+{\mathbb{E}}%
[\vert X^{1}\vert^2] +{\mathbb{E}}[\vert X^{2}\vert^2]
)(r-s)^{1/2}.
\label{eq:1a}
\end{equation}
Combining the two last inequalities, we finally get for $r\le s+1$ and a constant $ C>0 $ which depends on $ \delta $
\begin{equation*}
{\mathbb{E}}\left\langle y_{s,r},\Delta b(r)\right\rangle \leq (2\delta -%
\overline{b} ){\mathbb{E}}[\vert X^{1}-X^{2}\vert ^{2}]+C(1+{%
\mathbb{E}}[\vert X^{1}\vert ^2]+{\mathbb{E}}[\vert
X^{2}\vert^2] )(r-s)^{\frac{1}{2}}.
\end{equation*}
We now focus on the jump term and have by (\ref{VL5}) 
\begin{equation*}
\left\vert \Delta c(r,w,z)\right\vert ^{2}\leq \overline{c} ^{2}(z)(\vert
\tau _{s,\eta (r)}^{1}(w)-\tau _{s,\eta (r)}^{2}(w)\vert ^{2}+\vert
x_{s,\eta (r)}^{1}(X^{1})-X^{2}\vert ^{2}).
\end{equation*}
By (\ref{VL3'}), ${\mathbb{E}}[\vert x_{s,r}^{2}-X^{2}\vert
^{2}]\leq C(1+{\mathbb{E}}[\vert X^{2}\vert ])(r-s).$ We also recall
that by the very definition of $\tau $, the inequality $2xy\leq \delta
x^2+\delta^{-1}y^2 $ and \eqref{VL3'}, we have 
\begin{align*}
\!\!\int_{0}^{1}\!\!\left\vert \tau _{s,r}^{1}(w)-\tau _{s,r}^{2}(w)\right\vert
^{2}\!dw \!&= W_2^2(\Theta^m_{s,r}(\mu^1),\mu^2)\le {\mathbb{E}}\left[\vert
x_{s,r}^{1}(X^{1})-X^{2}\vert ^{2}\right] \\
&\leq (1+\delta ){\mathbb{E}}[\vert X^{1}-X^{2}\vert ^{2}]+(1+{\delta }%
^{-1}) {\mathbb{E}}\left[\vert x_{s,r}^{1}-X^{1}\vert ^{2}\right] \\
&\leq (1+\delta ){\mathbb{E}}[\vert X^{1}-X^{2}\vert ^{2}]+C(1+{\delta }%
^{-1})(1+{\mathbb{E}}[\vert X^{1}\vert ^{2}])(t-s).
\end{align*}
We conclude that 
\begin{align}
&\label{eq:2a} {\mathbb{E}}\int_{s}^{t}\int_{E}\int_{0}^{1}\left\vert \Delta
c(r,w,z)\right\vert ^{2}dw\Lambda (dz)dr \\
\leq\notag &{\mathbb{E}}\int_{s}^{t}\int_{E}\int_{0}^{1}\overline{c}%
^{2}(z)(\left\vert \tau _{s,r}^{1}(w)-\tau _{s,r}^{2}(w)\right\vert
^{2}+\left\vert x_{s,r}^{1}-X^{2}\right\vert ^{2})dw\Lambda (dz)dr \\\notag
\leq &2Q(1+\delta ){\mathbb{E}}[\vert X^{1}-X^{2}\vert ^{2}]+C(1+{\delta }%
^{-1})(1+{\mathbb{E}}[\vert X^{1}\vert ^{2}])(t-s).
\end{align}

Then, the above estimates in \eqref{eq:1a} and \eqref{eq:2a} allow us to take expectations in the It\^{o}'s
formula for the function $x\rightarrow \left\vert x\right\vert ^{2}$. We
also use the identity $\left\vert y+\Delta c\right\vert ^{2}-\left\vert
y\right\vert ^{2}-2\left\langle y,\Delta c\right\rangle =\left\vert \Delta
c\right\vert ^{2},$ so that we obtain 
\begin{align*}
	{\mathbb{E}}[\vert y_{s,t}\vert ^{2}] =&{\mathbb{E}}[\vert X^{1}-X^{2}\vert
	^{2}]+2{\mathbb{E}}\int_{s}^{t}\left\langle y_{s,r},\Delta b(r)\right\rangle
	dr\\&+{\mathbb{E}}\int_{s}^{t}\int_{E}\int_{0}^{1}\left\vert \Delta
	c(r,w,z)\right\vert ^{2}dw \Lambda (dz)dr. \\
	\leq& {\mathbb{E}}[\vert X^{1}-X^{2}\vert ^{2}](1+2(Q+\delta(2+Q) -%
	\overline{b} )(t-s))\\&+C(1+{\mathbb{E}}[\vert X^{1}\vert ^{2}]+{\mathbb{E}}%
	[\vert X^{2}\vert ^{2}])(t-s)^{3/2}
\end{align*}
so (\ref{VL6}) is proved by choosing $\delta $ small enough. In particular,
since ${\mathbb{E}}[\vert X^{1}-X^{2}\vert ^{2}]=W_{2}^{2}(\mu
^{1},\mu ^{2}),$ this gives 
\begin{align*}
&W_{2}^{2}(\Theta _{s,t}^{m}(\mu ^{1}),\Theta _{s,t}(\mu ^{2})) \\
&\leq {\mathbb{E}}\left[\vert x_{s,t}^{1}-x_{s,t}^{2}\vert
^{2}\right]={\mathbb{E}}[\vert y_{s,t}\vert ^{2}] \\
&\leq W_{2}^{2}(\mu ^{1},\mu ^{2})(1+2(Q+\delta(2+Q) -\overline{b}
)(t-s))+C(1+\left\Vert \mu ^{1}\right\Vert _{2}^{2}+\left\Vert \mu
^{2}\right\Vert _{2}^{2})(t-s)^{3/2}.
\end{align*}

Now, using \eqref{VL5'} 
and taking the limit as $m\rightarrow \infty $ we get 
\begin{align*}
W_{2}^{2}(\theta _{s,t}(\mu ^{1}),\Theta _{s,t}(\mu ^{2}))\leq&
W_{2}^{2}(\mu ^{1},\mu ^{2})(1+2(Q+\delta(2+Q) -\overline{b} )(t-s)) \\
&+C(1+\left\Vert \mu ^{1}\right\Vert _{2}^{2}+\left\Vert \mu
^{2}\right\Vert _{2}^{2})(t-s)^{3/2}.
\end{align*}
We get that $\theta$ and $\Theta$ are $(b_*,1,1/2)$-coupled for any $0<b_*<2(%
\bar{b}-Q)$.

The proof of \textbf{B} is analogue but simpler.
\end{proof}

\noindent {\bf Acknowledgment:} {A.A. acknowledges the support of the "chaire Risques Financiers", Fondation du Risque.
}
\end{document}